\documentclass[oneside,hidelinks,indentfirst,10pt]{book}
\usepackage{sty/preamble}
\title{Regularity Preserving Sum of Squares Decompositions}
\author{Sullivan \textsc{MacDonald}}

\usepackage{lineno}

\begin{document}

\frontmatter
\pagestyle{plain}
\pagenumbering{gobble}

\null\vfill
\begin{center}
REGULARITY PRESERVING SUM OF SQUARES DECOMPOSITIONS
\end{center}
\vspace{15cm}
\newpage

\begin{center}
\vfill\textsc{\Large \ttitle\\}\vfill
By \authorname, B.Sc.\\\vfill
{\large \textit{A Thesis Submitted to the School of Graduate Studies in the Partial Fulfillment of the Requirements for the Degree Master of Science}}\\\vfill
{\large \href{http://www.mcmaster.ca/}{McMaster University} \copyright\, Copyright by \authorname\,  March 1, 2023}\\[4cm]
\end{center}
\newpage

\pagenumbering{roman}
\setcounter{page}{2}
\noindent Master of Science, Mathematics (\the\year) \\
\href{https://www.math.mcmaster.ca/}{Department of Mathematics \& Statistics}\\
\href{http://www.mcmaster.ca/}{McMaster University}\\
Hamilton, Ontario, Canada\\[1.5cm]
Title: \ttitle\\
Author: \authorname, B.Sc.\\
Supervisor: Dr. Eric \textsc{Sawyer}\\ 
Number of pages: \pageref{lastoffront}, \pageref{LastPage}
\newpage

\section*{\Huge Abstract}
\addchaptertocentry{Abstract}

In their study of pseudo-differential operators, Fefferman \& Phong showed in \cite{Fefferman-Phong} that every non-negative function \(f\in C^{3,1}(\mathbb{R}^n)\) can be written as a finite sum of squares of functions in \(C^{1,1}(\mathbb{R}^n)\). In this thesis, we generalize their decomposition result to show that if \(f\in C^{k,\alpha}(\mathbb{R}^n)\) is non-negative for \(0\leq k\leq 3\) and \(0<\alpha\leq 1\), then \(f\) can be written as a finite sum of squares of functions that are each `half' as regular as \(f\), in the sense that they belong to the H\"older space
\[
    C^\frac{k+\alpha}{2}(\mathbb{R}^n)=\begin{cases}\hfil C^{\frac{k}{2},\frac{\alpha}{2}}(\mathbb{R}^n) & k\textrm{ even}, \\ C^{\frac{k-1}{2},\frac{1+\alpha}{2}}(\mathbb{R}^n) & k\textrm{ odd}.
    \end{cases}
\]
We also investigate sufficient conditions for such regularity preserving decompositions to exist when \(k\geq 4\), and we construct examples of functions which cannot be decomposed into finite sums of half-regular squares.

The aforementioned result of \cite{Fefferman-Phong}, as well as its subsequent refinements by Tataru in \cite{Tataru} and Sawyer \& Korobenko in \cite{SOS_I}, have been repeatedly used to investigate the properties of certain differential operators. We discuss similar applications of our generalized decomposition result to several problems in partial differential equations. In addition, we develop techniques for constructing non-negative polynomials which are not sums of squares of polynomials, and we prove related results which could not be found in a review of the literature.

\newpage

\section*{\Huge Acknowledgements}
\addchaptertocentry{Acknowledgements}
My greatest thanks goes to my supervisor, Dr. Eric Sawyer, for his indispensable guidance and helpful advice throughout my graduate studies at McMaster University. His encouragement and kind words have made the preparation of this thesis a thoroughly enjoyable experience, and I have learned a great deal under his mentorship.

I would also like to thank the Department of Mathematics and Statistics at McMaster University for fostering a lively and supportive work environment, and my fellow graduate students for many helpful conversations and plenty of memorable experiences. I also graciously acknowledge my mentor and collaborator, Dr. Scott Rodney at Cape Breton University, for encouraging me to study at McMaster and for providing helpful advice and good conversation along the way. 

Finally, I want to thank all of my family and friends for their unwavering support and affirmations as I conducted this research. I am grateful to have you all in my life.
\newpage

\setcounter{tocdepth}{1}
\tableofcontents
\label{lastoffront}
\mainmatter
\newpage
\pagestyle{thesis}
\chapter{Introduction}\label{chap:intro}

Given a non-negative function \(f:\mathbb{R}^n\rightarrow\mathbb{R}\), it is often useful to decompose \(f\) into a finite sum of squares of real-valued functions and write
\begin{equation}\label{eq:sos}
    f=\sum_{j=1}^m g_j^2.
\end{equation}
Usually, though, it is not straightforward to see what properties the functions \(g_1,\dots,g_m\) can inherit from \(f\). In this thesis we assume that \(f\) belongs to the H\"older space \(C^{k,\alpha}(\mathbb{R}^n)\) for a non-negative integer \(k\) and \(0<\alpha\leq 1\), and we investigate sufficient conditions for there to exist functions \(g_1,\dots,g_m\) which satisfy \eqref{eq:sos} and which belong to the `half-regular' H\"older space
\[
    C^\frac{k+\alpha}{2}(\mathbb{R}^n)=\begin{cases}\hfil C^{\frac{k}{2},\frac{\alpha}{2}}(\mathbb{R}^n) & k\textrm{ even}, \\ C^{\frac{k-1}{2},\frac{1+\alpha}{2}}(\mathbb{R}^n) & k\textrm{ odd}.
    \end{cases}
\]
In other words, we explore when \(f\) can be decomposed into a finite sum of squares of functions that are essentially half as regular as \(f\). Remarkably, we find that \eqref{eq:sos} always holds for some half-regular functions \(g_1,\dots,g_m\) whenever \(k\leq 3\) or \(n=1\). When \(k\geq 4\) and \(n\geq 2\) we are able to construct examples of non-decomposable functions, except in the special case of \(C^{4,\alpha}(\mathbb{R}^2)\).

The problem of finding sum of squares decompositions that preserve regularity is interesting in the setting of classical analysis, and it has important consequences elsewhere in mathematics, especially in the study of partial differential equations. It is shown in \cite{SOS_III} that if the symbol of a second-order differential operator \(L\) can be written as sums of squares with enough regularity, then one can find sufficient and sometimes necessary conditions for \(L\) to be hypoelliptic. This means that a distribution \(u\) must be smooth if \(Lu\) is smooth.

For another application, Guan shows in \cite{Guan} that if a non-negative function \(f\) defined on a domain \(\Omega\subset\mathbb{R}^n\) can be written as a sum of \(C^{1,1}(\Omega)\) squares, then solutions to the Dirichlet problem for the Monge-Amp\`ere equation
\begin{align*}
    D^2u=f & \;\;\textrm{ in }\Omega,\\
    \hfil u=0 & \;\;\textrm{ on }\partial\Omega,
\end{align*}
also belong to \(C^{1,1}(\Omega)\). A similar regularity preserving decomposition is employed by Tataru to study pseudodifferential operators in \cite{Tataru}, and these examples highlight the wide applicability of decompositions like \eqref{eq:sos}. In this thesis we do not present any novel applications in partial differential equations, but we emphasize that our decomposition results may be useful in addressing several open problems in the field; see \Cref{sec:applics} for further discussion.

Regularity preserving sum of squares decompositions have already been studied in several settings. Fefferman \& Phong showed in \cite{Fefferman-Phong} that if \(f\in C^{3,1}(\mathbb{R}^n)\) and \(f\) is non-negative, then there exist functions \(g_1,\dots,g_m\in C^{1,1}(\mathbb{R}^n)\) for which \eqref{eq:sos} holds. This result has been extended to various other H\"older spaces by Bony et al. in \cite{Bony}, Tataru in \cite{Tataru}, and Sawyer \& Korobenko in \cite{SOS_I}. In this work, we generalize the decomposition results from each of these works by studying non-negative \(C^{k,\alpha}(\mathbb{R}^n)\) functions for arbitrary \(k\geq 0\) and \(0<\alpha\leq 1\).

Finding a regularity preserving decomposition is often difficult. At a glance, one might suspect that taking a square root of a non-negative function \(f\in C^{k,\alpha}(\mathbb{R}^n)\) affords the desired decomposition with one function. However, this approach can fail spectacularly -- in \cite{Bony}, Bony gives an example of a non-negative \(C^\infty(\mathbb{R})\) function whose square root is not in \(C^{1,\alpha}(\mathbb{R})\) for any positive \(\alpha\). This example is given by taking \(f(0)=0\) and for \(x\neq0\) defining
\[
    f(x)=e^{-\frac{1}{|x|}}\bigg(\sin^2\bigg(\frac{\pi}{|x|}\bigg)+e^{-\frac{1}{x^2}}\bigg).
\]
Thus, if the functions \(g_1,\dots,g_m\) in \eqref{eq:sos} are to inherit half of the regularity of \(f\), then in general it is necessary that \(m> 1\) even when working with functions on the line. The exception is when \(f\in C^{\alpha}(\mathbb{R}^n)\) or \(C^{1,\alpha}(\mathbb{R}^n)\), in which case \(\sqrt{f}\) is indeed half as regular as \(f\) as we show in \Cref{chap:holderchap}.

There do not always exist decompositions of the form \eqref{eq:sos} which inherit half-regularity. Bony shows in \cite{Bony} that there are \(C^4(\mathbb{R}^4)\) functions that are not sums of squares in \(C^2(\mathbb{R}^4)\), suggesting that additional hypotheses are needed to guarantee the existence of regularity preserving decompositions. We introduce such conditions in \Cref{chap:higher}, to try and understand when functions in \(C^{k,\alpha}(\mathbb{R}^n)\) can be written as sums of half-regular squares when \(k\geq 4\). Though we make progress on this problem, a characterization of the decomposable functions remains elusive.

In \Cref{sec:gen} we also develop some techniques for constructing non-decomposable functions in \(C^{k,\alpha}(\mathbb{R}^n)\) for \(k\geq 4\). Our examples turn out to be polynomials which cannot be written as sums of squares of polynomials. The existence of such polynomials was proved by Hilbert in \cite{Hilbert}, but an explicit example did not appear in the literature until Motzkin found one in \cite{Motzkin} nearly 80 years later. Several more examples have since been found, many in \cite{reznik_extpsd}. Motivated by the question of decomposability of functions in \(C^{k,\alpha}(\mathbb{R}^n)\), we devise a method to construct non-negative polynomials which are not sums of squares in all possible cases.

Our findings concerning the possibility of forming a regularity preserving sum of squares decompositions in \(C^{k,\alpha}(\mathbb{R}^n)\), for every \(k\geq 0\) and \(n\geq 1\), are summarized in the following table.

\medskip
\renewcommand{\arraystretch}{1.3}
\newcommand\highbrace[2]{%
  \left.\rule{0pt}{#1}\right\}\text{#2}}
\setlength{\tabcolsep}{3.5pt}
\begin{center}
Is every non-negative \(C^{k,\alpha}(\mathbb{R}^n)\) function a sum of squares in \(C^\frac{k+\alpha}{2}(\mathbb{R}^n)\) (\(0<\alpha\leq 1\))?\\
\medskip
\begin{tabular}{|c|cc|cc|cc|cc|cc|cc|}
\hline
\diagbox[height=2em,width=3em]{\(n\)}{\(k\)} & \multicolumn{2}{c|}{0} & \multicolumn{2}{c|}{1} & \multicolumn{2}{c|}{2} & \multicolumn{2}{c|}{3} & \multicolumn{2}{c|}{4} & \multicolumn{2}{c|}{\(\geq 4\)}\\
\hline
1 & Yes &\multirow{4}{*}{\kern-0.4em\(\highbrace{7ex}{\makecell{Cor.\\\ref{cor:easycase}}}\)}& Yes &\multirow{4}{*}{\kern-0.4em\(\highbrace{7ex}{\makecell{Thm.\\\ref{thm:rootsin}}}\)}& Yes&\multirow{4}{*}{\kern-0.4em\(\highbrace{7ex}{\makecell{Thm.\\\ref{thm:c3main}}}\)} & Yes&\multirow{4}{*}{\kern-0.4em\(\highbrace{7ex}{\makecell{Thm.\\\ref{thm:c3main}}}\)} & \multicolumn{2}{c|}{Yes (\cite{Bony2})} & \multicolumn{2}{c|}{Yes (\cite{Bony2})} \\
2 & Yes && Yes && Yes && Yes && \multicolumn{2}{c|}{Unknown (\S \ref{sec:maybe})} & No & \multirow{3}{*}{\kern-2em\(\highbrace{5.3ex}{\makecell{Cor.\\\ref{cor:existence}}}\)} \\
3 & Yes && Yes && Yes && Yes && No &\multirow{2}{*}{\quad\(\highbrace{3.5ex}{\makecell{Cor.\\\ref{cor:existence}}}\)}& No &\\
\(\geq 3\) & Yes && Yes && Yes && Yes && No && No &\\
\hline
\end{tabular}
\end{center}
\bigskip

The plan for the remainder of this thesis is as follows. In \Cref{chap:holderchap} we describe the basic properties of H\"older spaces, before specializing to non-negative H\"older functions and studying their pointwise behaviour. Using these preliminary results, we show that the roots of \(C^\alpha(\mathbb{R}^n)\) and \(C^{1,\alpha}(\mathbb{R}^n)\) functions are half-regular. Then  in \Cref{chap:decomps}, we adapt the arguments employed in \cite{Tataru} and \cite{SOS_I} to prove our main result \Cref{thm:c3main}, which shows that non-negative \(C^{2,\alpha}(\mathbb{R}^n)\) and \(C^{3,\alpha}(\mathbb{R}^n)\) functions can be decomposed into sums of half-regular squares for every \(0<\alpha\leq 1\). Following \Cref{chap:poly}, in which we construct non-decomposable polynomials, we seek conditions in \Cref{chap:higher} for functions in \(C^{k,\alpha}(\mathbb{R}^n)\) to be partially decomposable when \(k\geq 4\).

In \Cref{chap:alg} we establish supplementary results which are only tangentially related to our central study of sums of squares. Similarly, we include \Cref{chap:appendix} to collect various standard results in one place for the convenience of the reader. We conclude the thesis with \Cref{chap:last}, in which we summarize our findings and discuss potential applications of our work along with possible directions for future research. In closing this summary, we include a road map of the thesis to help clarify the hierarchy of our main results, and for the reader to refer to later on.

\medskip
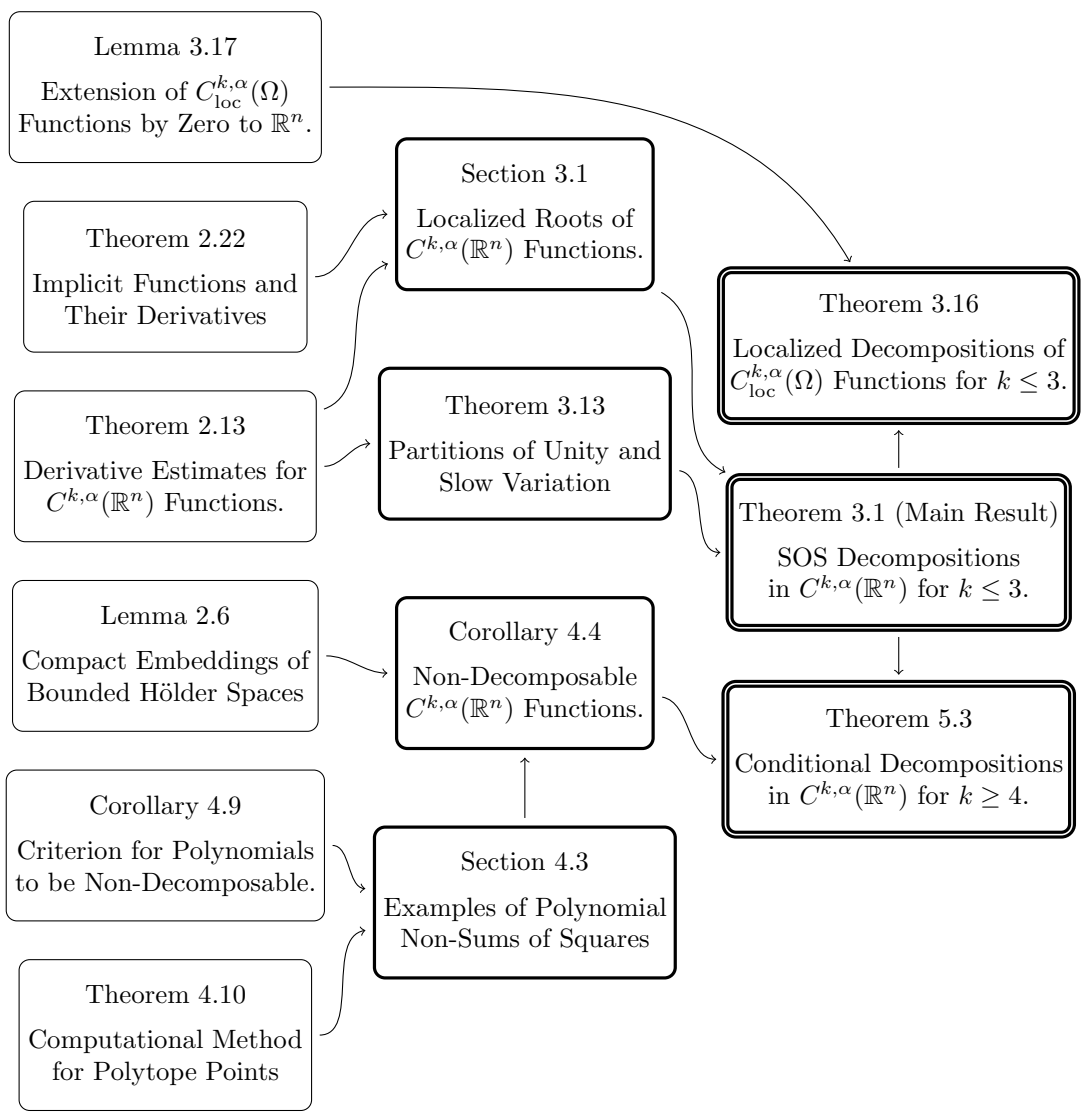
{\hspace{-1em}
\begin{tikzpicture}[block/.style={rounded corners, minimum width=3cm, minimum height=2cm, draw}]
\node[block,align=center] (1) {\Cref{thm:IFT}\\[0.5em]Implicit Functions and\\ Their Derivatives};
\node[block,very thick,above right=-1.2 and 1.15 of 1,align=center] (3) {\Cref{sec:reg}\\[0.5em]Localized Roots of\\ \(C^{k,\alpha}(\mathbb{R}^n)\) Functions.};
\node[block,double,very thick,below right=-0.3 and 0.85 of 3,align=center] (5) {\Cref{thm:locdecomp}\\[0.5em] Localized Decompositions of\\ \(C^{k,\alpha}_{\mathrm{loc}}(\Omega)\) Functions for \(k\leq 3\).};
\node[block,double,very thick,below=0.7of 5,align=center] (6) {\Cref{thm:c3main} (Main Result)\\[0.5em]SOS Decompositions\\ in \(C^{k,\alpha}(\mathbb{R}^n)\) for \(k\leq 3\).};
\node[block, below=0.5of 1,align=center] (2) {\Cref{thm:cauchylike}\\[0.5em]Derivative Estimates for\\ \(C^{k,\alpha}(\mathbb{R}^n)\) Functions.};
\node[block,very thick, below=1of 3,align=center] (4) {\Cref{thm:party}\\[0.5em]Partitions of Unity and\\ Slow Variation};
\node[block,double,very thick,below=0.7of 6,align=center] (7) {\Cref{thm:seconddmain}\\[0.5em] Conditional Decompositions\\ in \(C^{k,\alpha}(\mathbb{R}^n)\) for \(k\geq 4 \).}; 
\node[block, above=0.5of 1,align=center] (8) {\Cref{lem:extbyzero}\\[0.5em]Extension of \(C^{k,\alpha}_{\mathrm{loc}}(\Omega)\)\\ Functions by Zero to \(\mathbb{R}^n\).}; 
\node[block, below=0.5of 2,align=center] (11) {\Cref{lem:cptmb}\\[0.5em]Compact Embeddings of\\Bounded H\"older Spaces}; 
\node[block, below=0.5of 11,align=center] (9) {\Cref{cor:hullcond}\\[0.5em]Criterion for Polynomials\\to be Non-Decomposable.}; 
\node[block,below=0.5of 9,align=center] (10) {\Cref{thm:posweights}\\[0.5em]Computational Method \\for Polytope Points}; 
\node[block,very thick, below=1of 4,align=center] (13) {\Cref{cor:existence}\\[0.5em]Non-Decomposable\\\(C^{k,\alpha}(\mathbb{R}^n)\) Functions.};
\node[block,very thick, below=1of 13,align=center] (12) {\Cref{sec:gen}\\[0.5em] Examples of Polynomial\\Non-Sums of Squares}; 
\begin{scope}[->, shorten >=1mm, shorten <=1mm]
\draw (2) to[out=0,in=180] (4);
\draw (4) to[out=350,in=180] (6);
\draw (1) to[out=0,in=180] (3);
\draw (3) to[out=330,in=155] (6);
\draw (2) to[out=20,in=200] (3);
\draw (10) to[out=0,in=190] (12);
\draw (9) to[out=0,in=175] (12);
\draw (6) -- (5);
\draw (6) -- (7);
\draw (8) to[out=0,in=120] (5);
\draw (11) to[out=0,in=180] (13);
\draw (12) -- (13);
\draw (13) to[out=350,in=180] (7);
\end{scope}
\end{tikzpicture}
}
\chapter{Preliminary Results}\label{chap:holderchap}

The H\"older spaces play a central role in this work, so we begin by establishing some of their important properties and collecting results that will be useful in the subsequent chapters. The following expository section outlines the basic characteristics of H\"older spaces -- in Sections \ref{sec:nn} and \ref{sec:malg} we restrict attention to non-negative H\"older continuous functions, which enjoy some special pointwise properties.

\section{Elementary Properties of H\"older Spaces}

Given a connected set \(\Omega\subseteq\mathbb{R}^n\) and a positive number \(\alpha\leq 1\), a real-valued function \(f\) defined on \(\Omega\) is said to be uniformly \(\alpha\)-H\"older continuous if 
\begin{equation}\label{eq:snd}
    [f]_{\alpha,\Omega}=\sup_{\substack{x,y\in \Omega,\\x\neq y}}\frac{|f(x)-f(y)|}{|x-y|^\alpha}<\infty.
\end{equation}
The H\"older space \(C^\alpha(\Omega)\) is defined as the set of uniformly \(\alpha\)-H\"older continuous functions on \(\Omega\). Similarly, \(f\) is said to be locally \(\alpha\)-H\"older on \(\Omega\) if \([f]_{\alpha,K}<\infty\) for every compact subset \(K\) of \(\Omega\), and we denote the class of locally \(\alpha\)-H\"older functions by \(C^{\alpha}_\mathrm{loc}(\Omega)\).

Usually we write \([f]_\alpha\) if the underlying set is unambiguous. Given \(f,g\in C^\alpha(\Omega)\) it follows from \eqref{eq:snd} and the triangle inequality that \([f+g]_\alpha\leq [f]_\alpha+[g]_\alpha\), and for any \(c\in\mathbb{R}\) the homogeneity identity \([cf]_\alpha=|c|[f]_\alpha\) is also easy to verify. Therefore \eqref{eq:snd} defines a semi-norm and \(C^\alpha(\Omega)\) is a vector space under pointwise addition. Note however that this is not a proper norm, since \([f]_{\alpha,\Omega}=0\) whenever \(f\) is a constant function on \(\Omega\).

H\"older continuity is stronger than uniform continuity, and when \(\alpha=1\) inequality \eqref{eq:snd} holds for the uniformly Lipschitz functions on \(\Omega\). We denote the space of Lipschitz functions by \(C^{0,1}(\Omega)\) to avoid confusion with \(C^1(\Omega)\), the space of continuously differentiable functions on \(\Omega\). Analogously, we write \(C^{0,1}_\mathrm{loc}(\Omega)\) to denote the space of locally Lipschitz functions.

For any \(0<\alpha\leq 1\), a simple example of an \(\alpha\)-H\"older continuous function is \(f(x)=|x|^\alpha\). A more interesting example is the Cantor function, which belongs to \(C^{\log(2)/\log(3)}([0,1])\) as we show in \Cref{sec:holdexamples}. If \(\alpha>1\) then the semi-norm \eqref{eq:snd} is only finite if \(f\) is constant; to see this, observe that if \(\alpha>1\) and \(x\in\Omega\) then
\[
    \lim_{|h|\rightarrow0}\frac{|f(x+h)-f(x)|}{|h|}\leq [f]_\alpha\lim_{|h|\rightarrow0}\frac{|h|^\alpha}{|h|}=0.
\]
It follows that all directional derivatives of \(f\) exist and are identically zero, meaning that the space \(C^\alpha(\Omega)\) is comprised of constant functions. Due to this uninteresting behaviour we do not consider \(C^\alpha(\Omega)\) for \(\alpha>1\), and henceforth we restrict our attention to \(C^\alpha(\Omega)\) for \(0<\alpha\leq 1\).

In general, larger values of \(\alpha\) correspond to stronger continuity, with \(C^{0,1}(\Omega)\) functions being differentiable almost everywhere in \(\Omega\) by Rademacher's Theorem; see \Cref{thm:rad}. Further, if \(\Omega\) is bounded then \(C^{\beta}(\Omega)\subset C^{\alpha}(\Omega)\) whenever \(0<\alpha<\beta\). Indeed, if \(f\in C^{\beta}(\Omega)\) and \(d=\sup_{\Omega}|x-y|\) denotes the diameter of \(\Omega\), then for \(x,y\in\Omega\) we have
\[
    |f(x)-f(y)|\leq [f]_\beta|x-y|^\beta\leq [f]_\beta d^{\beta-\alpha}|x-y|^{\alpha}.
\]
Thus \([f]_\alpha\leq[f]_\beta d^{\beta-\alpha}\) and \(f\in C^\alpha(\Omega)\). Even when \(\Omega\) may be unbounded, if \(f\in C^{\alpha}(\Omega)\cap C^{\beta}(\Omega)\) for \(0<\alpha<\beta\) then \(f\in C^{\gamma}(\Omega)\) whenever \(\alpha<\gamma<\beta\). This is because for distinct points \(x,y\in\Omega\),
\begin{equation}\label{eq:sninterp}
    \bigg(\frac{|f(x)-f(y)|}{|x-y|^\gamma}\bigg)^{\beta-\alpha}=\bigg(\frac{|f(x)-f(y)|}{|x-y|^\alpha}\bigg)^{\beta-\gamma}\bigg(\frac{|f(x)-f(y)|}{|x-y|^\beta}\bigg)^{\gamma-\alpha}.
\end{equation}
From this estimate it follows after taking a supremum that \([f]_\gamma^{\beta-\alpha}\leq [f]_\alpha^{\beta-\gamma}[f]_\beta^{\gamma-\alpha}\).

For many applications it is useful to restrict one's attention to the set of bounded \(\alpha\)-H\"older continuous functions defined on \(\Omega\), and we denote the set of these functions by \(C^\alpha_b(\Omega)\). Clearly \(C^\alpha_b(\Omega)\subseteq C^\alpha(\Omega)\), and the two spaces are equal when \(\Omega\) is compact. The vector space of bounded \(\alpha\)-H\"older continuous functions, unlike \(C^\alpha(\Omega)\), can be equipped with the following norm
\[
    \|f\|_{C_b^{\alpha}(\Omega)}=\sup_\Omega|f|+[f]_{\alpha,\Omega}.
\]
Indeed, given this norm \(C_b^{\alpha}(\Omega)\) is a Banach space, as we discuss later in the chapter. 

For the purpose of our main results, a boundedness condition is often needlessly restrictive since it is not satisfied even by simple functions like \(f(x)=|x|^\alpha\) defined over \(\mathbb{R}^n\). Thus, most of the time we work with the simplest H\"older space \(C^\alpha(\Omega)\), and it is only for some results like \Cref{thm:almostthere} that we find it necessary to restrict attention to \(C^\alpha_b(\Omega)\). In passing however, we note that the norm on \(C^\alpha_b(\Omega)\) defined above enjoys the following interpolation property.

\begin{lem}
If \(f\in C_b^\alpha(\Omega)\cap C_b^\beta(\Omega)\) and \(0<\alpha<\gamma<\beta\) then \(\|f\|_{C_b^\gamma(\Omega)}^{\beta-\alpha}\leq \|f\|_{C_b^\alpha(\Omega)}^{\beta-\gamma}\|f\|_{C_b^\beta(\Omega)}^{\gamma-\alpha}\).
\end{lem}

\begin{proof} The semi-norm interpolation given by \eqref{eq:sninterp} shows that if \(s=\frac{\gamma-\alpha}{\beta-\alpha}\) and \(\lambda=\frac{[f]_\alpha}{\|f\|_{C_b^\alpha(\Omega)}}\), then
\[
    [f]_\gamma+\sup_\Omega|f|\leq [f]_\alpha^{1-s}[f]_\beta^s+\sup_\Omega|f|^{1-s}\sup_\Omega|f|^s=\|f\|_{C_b^\alpha(\Omega)}^{1-s}\bigg(\lambda\bigg(\frac{[f]_\beta}{\lambda}\bigg)^s+(1-\lambda)\bigg(\frac{\sup_\Omega|f|}{1-\lambda}\bigg)^s\bigg).    
\]
The mapping \(t\mapsto t^s\) is concave since \( 0<s< 1\), and as \(\lambda\leq 1\) it follows from \Cref{lem:convexity} that
\[
     \lambda\bigg(\frac{[f]_\beta}{\lambda}\bigg)^s+(1-\lambda)\bigg(\frac{\sup_\Omega|f|}{1-\lambda}\bigg)^s\leq \big([f]_\beta+\sup_\Omega|f|\big)^s=\|f\|_{C_b^\beta(\Omega)}^s.
\]
Therefore \(\|f\|_{C_b^\gamma(\Omega)}\leq \|f\|_{C_b^\alpha(\Omega)}^{1-s}\|f\|_{C_b^\beta(\Omega)}^{s}\), which yields the claimed interpolation.
\end{proof}

Some other norm inequalities for uniformly bounded H\"older functions are worth noting. If \(f,g\in C_b^\alpha(\Omega)\) for a set \(\Omega\subseteq\mathbb{R}^n\) not necessarily bounded, we have for \(x,y\in\Omega\) that
\[
    |f(x)g(x)-f(y)g(y)|\leq \sup_\Omega|f||g(x)-g(y)|+\sup_\Omega|g||f(x)-f(y)|.
\]
It follows that \([fg]_\alpha\leq \sup_\Omega|f|[g]_\alpha+\sup_\Omega|g|[f]_\alpha\) and \(\|fg\|_{C_b^\alpha(\Omega)}\leq \|f\|_{C_b^\alpha(\Omega)}\|g\|_{C_b^\alpha(\Omega)}\), meaning that the norm of \(C_b^\alpha(\Omega)\) is sub-multiplicative. Additionally, if \(\alpha<\beta\) then from Young's inequality (see \Cref{cor:young}) it follows that
\[
    \|g\|_{C_b^\alpha(\Omega)}=\sup_\Omega|g|+\sup_{\substack{x,y\in \Omega,\\x\neq y}}\frac{|g(x)-g(y)|}{|x-y|^\alpha}\leq \frac{\alpha}{\beta}\sup_{\substack{x,y\in \Omega,\\x\neq y}}\frac{|g(x)-g(y)|}{|x-y|^\beta}+\bigg(3-\frac{2\alpha}{\beta}\bigg)\sup_\Omega|g|\leq 3\|g\|_{C_b^\beta(\Omega)}.
\]
Thus if \(f\in C_b^{\alpha}(\Omega)\) and \(g\in C_b^{\beta}(\Omega)\), and if \(\gamma=\min\{\alpha,\beta\}\), then \(\|fg\|_{C_b^{\gamma}(\Omega)}\leq 3\|f\|_{C_b^{\alpha}(\Omega)}\|g\|_{C_b^{\beta}(\Omega)}\). 

Next we discuss higher-order H\"older spaces, denoting by \(C^{k,\alpha}(\Omega)\) the set of \(k\)-times continuously differentiable functions on \(\Omega\) whose \(k^\mathrm{th}\)-order derivatives are \(\alpha\)-H\"older continuous on \(\Omega\). To describe these spaces concisely, it is convenient to introduce multi-index notation. A multi-index of length \(n\) is an \(n\)-tuple of non-negative integers \(\beta=(\beta_1,\dots,\beta_n)\), and addition of multi-indices is defined entry-wise. The order of \(\beta\) is given by \(|\beta|=\beta_1+\cdots+\beta_n\) so that \(|\beta+\gamma|=|\beta|+|\gamma|\), and the \(\beta\)-derivative of a function \(f\) is then defined as the order-\(|\beta|\) derivative given by
\[
    \partial^\beta f=\frac{\partial^{|\beta| }f}{\partial^{\beta_1}x_1\cdots\partial^{\beta_n}x_n}.
\]
If \(|\beta|=0\) then \(\beta=(0,\dots,0)\), and to keep notation simple in this case we set \(\partial^\beta f=f\).

Equipped with multi-index notation, we say that a function \(f\) belongs to \(C^{k,\alpha}(\Omega)\) if \([\partial^\beta f]_{\alpha,\Omega}\) is finite for every multi-index \(\beta\) of order \(|\beta|=k\). Likewise, we define \(C^{k,\alpha}_\mathrm{loc}(\Omega)\) to be the vector space of functions \(f\) for which \([\partial^\beta f]_{\alpha,K}\) is finite whenever \(K\) is a compact subset of \(\Omega\) and \(|\beta|=k\). Finally, we denote by \(C_b^{k,\alpha}(\Omega)\) the subspace of \(C^{k,\alpha}(\Omega)\) comprised of uniformly bounded functions on \(\Omega\). It is easy to see that these three spaces coincide when \(\Omega\) is compact, but in general they are distinct.

By indexing over \(|\beta|\leq k\), we mean to reference every multi-index of a fixed length \(n\) whose order is at most \(k\). Adopting this convention we can define the following norm on \(C_b^{k,\alpha}(\Omega)\),
\begin{equation}\label{eq:holdernorm}
    \|f\|_{C^{k,\alpha}_b(\Omega)}=\sum_{|\beta|\leq k}\sup_\Omega|\partial^\beta f|+\sum_{|\beta|=k}[\partial^\beta f]_{\alpha,\Omega}.
\end{equation}
Once again, for the purposes of our later decomposition results it is unnecessary to restrict attention to bounded \(C^{k,\alpha}(\Omega)\) functions except in some special settings. Wherever a boundedness condition is needed, we state the requirement explicitly.

To express our main results more concisely, we explain exactly what we mean when we refer to `half-regular' spaces of functions. Fixing a H\"older space \(C^{k,\alpha}(\Omega)\), we define
\begin{equation}\label{eq:rootspaces}
    C^{\frac{k+\alpha}{2}}(\Omega)=\begin{cases}
    \hfil C^{\frac{k}{2},\frac{\alpha}{2}}(\Omega) & k\textrm{ even},\\
    C^{\frac{k-1}{2},\frac{1+\alpha}{2}}(\Omega) & k\textrm{ odd}.
    \end{cases}
\end{equation}
This space consists of functions which are essentially `half' as differentiable as those in \(C^{k,\alpha}(\Omega)\), and likewise the highest-order derivatives of said functions retain `half' of the H\"older continuity.

Our main result in the next chapter demonstrates that remarkably, every non-negative \(C^{k,\alpha}(\mathbb{R}^n)\) function can be written as a finite sums of squares of functions belonging to the half-regular space whenever \(k\leq 3\), and \(0<\alpha\leq 1\). This is true in any number of dimensions, and our results further show that the functions \(g_1,\dots,g_m\) in \eqref{eq:sos} and their derivatives can be controlled pointwise in a straightforward way -- see \Cref{thm:c3main} and the ensuing discussion. For \(k\geq 4\) matters are more complicated, as we discuss in \Cref{chap:higher}, but nevertheless we recover some partial results even in that more difficult setting.

In proving the aforementioned decomposition results we will frequently use a sub-product rule for H\"older semi-norms, like that employed by Sawyer \& Korobenko in \cite[Eq. (3.10)]{SOS_I} to prove a \(C^{4,\alpha}(\mathbb{R}^n)\) decomposition result. Before stating this rule, we must introduce some additional multi-index notation. On the collection of multi-indices of a fixed length \(n\), there exists a natural partial ordering given by writing \(\beta\leq \gamma\) for \(\beta=(\beta_1,\dots,\beta_n)\) and \(\gamma=(\gamma_1,\dots,\gamma_n)\) if \(\gamma_j\leq \beta_j\) for each \(j=1,\dots,n\). If \(\gamma\leq \beta\) then \(\beta-\gamma=(\beta_1-\gamma_1,\dots,\beta_n-\gamma_n)\) is also a multi-index, and we can define a generalized binomial coefficient by setting 
\begin{equation}\label{eq:binom}
    \binom{\beta}{\gamma}=\frac{\beta!}{\gamma!(\beta-\gamma)!},
\end{equation}
where \(\beta!=\beta_1!\cdots\beta_n!\) is a multi-index factorial. To indicate that we are summing over all multi-indices \(\gamma\) that are below \(\beta\) in the partial ordering described above, we use the subscript \(\gamma\leq\beta\). The utility of this notation is illustrated by the following well-known generalization of the product rule, whose proof we omit.

\begin{lem}[General Leibniz Rule]\label{lem:Leib}
Let \(f,g\in C^k(\mathbb{R}^n)\), and suppose that \(|\beta|\leq k\). Then
\[
    \partial^\beta(fg)=\sum_{\gamma\leq\beta}\binom{\beta}{\gamma}(\partial^\gamma f)(\partial^{\beta-\gamma}g).
\]
\end{lem}

Equipped with this result and the versatile notation afforded by multi-indices, we now state and prove the aforementioned sub-product rule of Sawyer \& Korobenko for H\"older semi-norms.

\begin{lem}\label{lem:subprod}
Let \(f,g\in C^{k,\alpha}(\Omega)\) for \(0<\alpha\leq 1\). If \(\beta\) is any multi-index for which \(|\beta|\leq k\) then 
\[
    [\partial^\beta(fg)]_{\alpha,\Omega}\leq\sum_{\gamma\leq\beta}\binom{\beta}{\gamma}([\partial^{\beta-\gamma}f]_{\alpha,\Omega}\sup_{\Omega}|\partial^\gamma g|+[\partial^{\beta-\gamma}g]_{\alpha,\Omega}\sup_{\Omega}|\partial^\gamma f|).
\]
\end{lem}

\begin{proof}
Using \Cref{lem:Leib}, we first expand the difference \(\partial^\beta(fg)(x)-\partial^\beta(fg)(y)\) as follows,
\begin{align*}
    |\partial^\beta(fg)(x)-&\partial^\beta(fg)(y)|=\bigg|\sum_{\gamma\leq\beta}\binom{\beta}{\gamma}(\partial^{\beta-\gamma}f(x)\partial^\gamma g(x)-\partial^{\beta-\gamma}f(y)\partial^\gamma g(y))\bigg|\\
    &=\bigg|\sum_{\gamma\leq\beta}\binom{\beta}{\gamma}((\partial^{\beta-\gamma}f(x)-\partial^{\beta-\gamma}f(y))\partial^\gamma g(x)+(\partial^\gamma g(x)-\partial^\gamma g(y))\partial^{\beta-\gamma} f(y))\bigg|.
\end{align*}
Using the triangle inequality, we see now that \(|\partial^\beta(fg)(x)-\partial^\beta(fg)(y)|\) is bounded from above by
\[
    \sum_{\gamma\leq\beta}\binom{\beta}{\gamma}(|\partial^{\beta-\gamma}f(x)-\partial^{\beta-\gamma}f(y)|\sup_\Omega|\partial^\gamma g|+|\partial^\gamma g(x)-\partial^\gamma g(y)|\sup_\Omega|\partial^{\beta-\gamma} f|).
\]
Dividing this sum by \(|x-y|^\alpha\) and taking a supremum over \(x,y\in\Omega\), the claimed semi-norm estimate for \(\partial^\beta(fg)\) follows immediately from sub-additivity of the supremum. 
\end{proof}

In \cite{Bony2}, Bony introduces a pointwise variant of the \(\alpha\)-H\"older semi-norm, which is also employed by Sawyer \& Korobenko in \cite{SOS_I}. This operator is defined for \(0<\alpha\leq 1\) by setting
\begin{equation}\label{eq:pwop}
    [f]_{\alpha}(x)=\limsup_{y,z\rightarrow x}\frac{|f(y)-f(z)|}{|y-z|^\alpha}.
\end{equation}
We use this pointwise operator, rather than the standard H\"older semi-norm, in order to extend pointwise H\"older estimates to global ones in \Cref{chap:decomps}, without fixing an underlying set. Like the usual H\"older semi-norm, this pointwise variant is subadditive, homogeneous, and it can be used to recover the usual H\"older semi-norm via the identity 
\[
    [f]_{\alpha,\Omega}=\sup_{x\in\Omega}[f]_{\alpha}(x).
\]
Moreover, \eqref{eq:pwop} satisfies a sub-product rule just like the usual H\"older semi-norm. The proof of the following variant of this inequality is identical to that of \Cref{lem:subprod}, so we omit it.

\begin{lem}\label{lem:sharpsubprod}
Let \(f,g\in C^{k,\alpha}(\Omega)\) for \(0<\alpha\leq 1\). If \(\beta\) is any multi-index for which \(|\beta|\leq k\) then
\[
    [\partial^\beta(fg)]_{\alpha}(x)\leq\sum_{\gamma\leq\beta}\binom{\beta}{\gamma}([\partial^{\beta-\gamma}f]_{\alpha}(x)|\partial^\gamma g(x)|+[\partial^{\beta-\gamma}g]_{\alpha}(x)|\partial^\gamma f(x)|).
\]
\end{lem}

Such objects as the local semi-norm allow us to argue locally in terms of the pointwise behaviour, before combining local estimates and taking a supremum over a larger domain to recover global H\"older estimates. The utility of this approach will become apparent in the proof of \Cref{thm:c3main}.

For completeness of this expository section, we also state some more advanced properties of the higher-order H\"older spaces. The first is a well-known and often useful property of the H\"older spaces of uniformly bounded functions; its proof is a standard exercise in graduate analysis, and we furnish a straightforward argument using the Fundamental Theorem of Calculus.

\begin{thm}
The space \(C_b^{k,\alpha}(\Omega)\) equipped with the norm \eqref{eq:holdernorm} is a Banach space.
\end{thm}

\begin{proof} Recall that a Banach space is a complete normed vector space. We have already established that \(C_b^{k,\alpha}(\Omega)\) is a normed vector space, leaving us to show completeness. If \(\{f_j\}\) is Cauchy in \(C_b^{k,\alpha}(\Omega)\) equipped with the norm \eqref{eq:holdernorm}, and if \(|\beta|\leq k\), then \(\{\partial^\beta f_j\}\) is Cauchy in \(C_b^0(\Omega)\), the space of uniformly continuous functions on \(\Omega\) equipped with the supremum norm. Since \(C^0_b(\Omega)\) is complete, we conclude that \(\{\partial^\beta f_j\}\) converges uniformly to a function \(f_\beta\in C_b^0(\Omega)\). For convenience, we define \(f=\displaystyle\lim_{j\rightarrow\infty}f_j\), noting that \(f\in C^0_b(\Omega)\).

Using induction on \(m\), we next show that \(\partial^\beta f=f_\beta\) for every \(\beta\) of order \(m\leq k\). In other words, the limits of the derivatives of the sequence \(\{f_j\}\) agree with the derivatives of its limits.  If \(m=0\) this follows from our definition of \(f\), giving a base case. Suppose now that \(\partial^\beta f=f_\beta\) whenever \(|\beta|=m\), and let \(\partial^\mu =\partial_i\partial^\beta \) be a derivative of order \(m+1\). Let \(e_i\) denote the \(i^\textrm{th}\) basis vector of \(\mathbb{R}^n\), and for \(x\) belonging to the interior of \(\Omega\) choose \(h\in\mathbb{R}\) small enough that \(x+he_i\in\Omega\). Proceeding with these selections, we see from the Fundamental Theorem of Calculus Part II that
\[
    \partial^\beta f(x+he_i)-\partial^\beta f(x)=\lim_{j\rightarrow\infty}\int_0^h\partial_{i}\partial^\beta f_j(x+te_i)dt
\]
Moreover, since \(\{\partial^\mu f_j\}\) is Cauchy in \(C^{0}_b(\Omega)\) we see that \(\sup_\Omega|\partial_i\partial^\beta f_j|\) is bounded independent of \(j\). It follows from the Dominated Convergence Theorem that
\[
    \partial^\beta f(x+he_i)-\partial^\beta f(x)=\int_0^h\lim_{j\rightarrow\infty}\partial_{i}\partial^\beta f_j(x+te_i)dt=\int_0^h f_\mu(x+te_i)dt.
\]
Taking a derivative of the identity above, the Fundamental Theorem of Calculus Part I shows that \(\partial^\mu f=f_\mu\). Hence \(\partial^\beta f=f_\beta\) whenever \(|\beta|\leq k\), meaning that \(\partial^\beta f_j\) converges uniformly to \(\partial^\beta f\) and the space \(C^k_b(\Omega)\) is complete.

It remains to show that \(\partial^\beta f\in C^\alpha_b(\Omega)\) whenever \(|\beta|=k\). For \(x,y\in\Omega\) we have
\[
        \frac{|\partial^\beta f(x)-\partial^\beta f(y)|}{|x-y|^\alpha}=\lim_{j\rightarrow\infty}\frac{|\partial^\beta f_j(x)-\partial^\beta f_j(y)|}{|x-y|^\alpha}\leq \limsup_{j\rightarrow\infty}\;[\partial^\beta f_j]_{\alpha,\Omega} \leq \lim_{j\rightarrow\infty}\|f_j\|_{C_b^{k,\alpha}(\Omega)}.
\]
The limit on the right converges since \(\{f_j\}\) is Cauchy in \(C_b^{k,\alpha}(\Omega)\), so by taking a supremum over \(x,y\in\Omega\) we find that \([\partial^\beta f]_\alpha<\infty\) whenever \(|\beta|=k\), showing that \(f\in C_b^{k,\alpha}(\Omega)\). Since \(\{f_j\}\) was any Cauchy sequence, it follows that \(C_b^{k,\alpha}(\Omega)\) is complete as claimed.
\end{proof}

Earlier we showed that if \(\alpha< \beta\), then the space \(C^{\beta}(\Omega)\) is embedded in \(C^{\alpha}(\Omega)\), provided that either \(\Omega\) is bounded or that we restrict to uniformly bounded functions. It is worth noting that the latter is a compact embedding, meaning that if \(\{f_j\}\) is a bounded sequence in the norm of \(C_b^\alpha(\Omega)\) then it has a subsequence convergent in the norm of \(C_b^\beta(\Omega)\). More generally we have the following result, which is a standard but more advanced property of H\"older spaces that is stated and proved in \cite[Lemma 6.33]{GilbargTrudinger} as well as in \cite[Theorem 24.14]{Driver}.

\begin{lem}\label{lem:cptmb}
Let \(0<\alpha,\beta\leq 1\), and let \(j\) and \(k\) be non-negative integers. If \(k+\alpha<j+\beta\) then \(C_b^{j,\beta}(\Omega)\) is compactly embedded in \(C_b^{k,\alpha}(\Omega)\).
\end{lem}

Indeed, \cite{GilbargTrudinger} and \cite{Driver} both contain a thorough development of the properties of H\"older spaces that go beyond the scope of this thesis, but which are relevant to the widespread applications of H\"older spaces in the study of partial differential equations. The reader wishing to explore these spaces further is encouraged to consult the works cited throughout this chapter, and also to see \Cref{sec:holdexamples} for some interesting examples of \(\alpha\)-H\"older continuous functions.

Next, we establish a useful property of the Taylor polynomials of H\"older continuous functions.

\begin{lem}\label{lem:Taylorest}
Let \(f\in C^{k,\alpha}(\Omega)\). If \(|\beta|<k\) then there exists a constant \(C\) such that for any \(x,y\in\Omega\) the following inequality holds,
\begin{equation}\label{eq:taylorest}
    |\partial^\beta f(x)-\partial^\beta f(y)|\leq \sum_{0\leq |\gamma|\leq k-|\beta|}\frac{1}{\gamma!}|x-y|^{|\gamma|}|\partial^{\beta+\gamma} f(x)|+C|x-y|^{k-|\beta|+\alpha}.    
\end{equation}
\end{lem}

\begin{proof}
Let \(\beta\) be any multi-index of length \(n\) and order \(|\beta|<k\). We begin by using \Cref{thm:taylorthm} to form a Taylor expansion for \(\partial^\beta f\) evaluated at \(y\in\Omega\) and centred at \(x\in\Omega\),
\[
    \partial^\beta f(y)=\sum_{0\leq |\gamma|<k-|\beta|}\frac{1}{\gamma!}\partial^{\beta+\gamma} f(x)(y-x)^\gamma+\sum_{|\gamma|=k-|\beta|}\frac{|\gamma|}{\gamma!}(y-x)^\gamma\int_0^1(1-t)^{|\gamma|-1}\partial^{\beta+\gamma}f(x+t(y-x))dt.
\]
To refine this, we observe that the integral on the right-hand side can be rewritten in the form
\[
    \frac{1}{|\gamma|}\partial^{\beta+\gamma}f(x)+\int_0^1(1-t)^{|\gamma|-1}(\partial^{\beta+\gamma}f(x+t(y-x))-\partial^{\beta+\gamma}f(x))dt.
\]
Consequently, if \(|\beta|+|\gamma|=k\) we may use that \(f\in C^{k,\alpha}(\Omega)\) by assumption to make the estimate 
\[
    |\gamma|\int_0^1(1-t)^{|\gamma|-1}\partial^{\beta+\gamma}f(x+t(y-x))dt\leq |\partial^{\beta+\gamma}f(x)|+[\partial^{\beta+\gamma}f]_{\alpha,\Omega}|\gamma||x-y|^\alpha\int_0^1(1-t)^{|\gamma|-1}t^\alpha dt.
\]
The integral on the right-hand side is bounded by \(1/|\gamma|\), meaning that the item on the left is bounded by \(\partial^{\beta+\gamma}f(x)+[\partial^{\beta+\gamma}f]_{\alpha,\Omega}|x-y|^\alpha\). Rearranging our initial Taylor expansion and applying this estimate affords \eqref{eq:taylorest} with \(C=\sum_{|\gamma|=k-|\beta|}[\partial^{\beta+\gamma}f]_{\alpha,\Omega}/\gamma!\).
\end{proof}

Much can be gleaned from the Taylor polynomials of functions belonging to H\"older spaces, and in the next section we exploit several useful properties of the Taylor polynomials of non-negative H\"older continuous functions. To close this section, we prove a useful `recombination' result which will be important in our later construction of sum of squares decompositions. Recall that two functions are said to have disjoint supports if they are not both nonzero at any point.

\begin{lem}\label{lem:recomb}
Let \(\{f_j\}\) be a countable collection of non-negative \(C^{k,\alpha}(\Omega)\) functions with pairwise disjoint supports. Assume that for each \(\beta\) with \(|\beta|\leq k \), the bound \([\partial^\beta f_j]_\alpha\leq C\) holds for a \(C\) independent of \(j\). Then the following function is also in \(C^{k,\alpha}(\Omega)\),
\[
    F=\sum_{j=1}^\infty f_j.
\]
\end{lem}

\begin{proof}

Fix a multi-index \(\beta\) of order \(k\) and let \(x,y\in\Omega\). Assume without loss of generality that \(F(x)=f_i(x)\) and \(F(y)=f_j(y)\) for some \(i,j\in\mathbb{N}\). If \(i=j\) then 
\[
    \frac{|\partial^\beta F(x)-\partial^\beta F(y)|}{|x-y|^\alpha}=\frac{|\partial^\beta f_i(x)-\partial^\beta f_i(x)|}{|x-y|^\alpha}\leq C.
\]
On the other hand, if \(i\neq j\) then \(\partial^\beta F(x)=\partial^\beta f_i(x)+\partial^\beta f_j(x)\) and \(\partial^\beta F(y)=\partial^\beta f_i(y)+\partial^\beta f_j(y)\), since the functions comprising \(f\) have pairwise disjoint support. Thus the triangle inequality,
\[
    \frac{|\partial^\beta F(x)-\partial^\beta F(y)|}{|x-y|^\alpha}\leq \frac{|\partial^\beta f_i(x)-\partial^\beta f_i(y)|}{|x-y|^\alpha}+\frac{|\partial^\beta f_j(x)-\partial^\beta f_j(y)|}{|x-y|^\alpha}\leq 2C.
\]
Since \(x\) and \(y\) were arbitrary points in \(\Omega\), we conclude that \([\partial^\beta F]_\alpha\leq 2C\) as required. 
\end{proof}

\section{Non-Negative Functions in H\"older Spaces}\label{sec:nn}

For the remainder of this thesis, we work to decompose non-negative functions in H\"older spaces -- accordingly, in this section we investigate the pointwise properties of such functions. Perhaps the most useful result we obtain is a set of inequalities which essentially show that the functions we study are controlled pointwise by their even-order derivatives. An early result of this type is an inequality of Malgrange, which is often attributed to Glaeser who published it in \cite{Glaeser}. Malgrange found that if \(f\in C^2(\mathbb{R})\) is non-negative and \(\sup_\mathbb{R}|f''|\leq M\), then
\begin{equation}\label{eq:ogmalg}
    |f'(x)|\leq \sqrt{2Mf(x)}    
\end{equation}
for every \(x\in\mathbb{R}\). By working in H\"older spaces we can refine this estimate and remove the boundedness condition. Moreover, we can obtain generalizations of \eqref{eq:ogmalg} to higher dimensions. These estimates will be critical to our study of \(C^{1,\alpha}(\mathbb{R}^n)\) functions.

Before moving on to more advanced estimates, we establish some basic properties of non-negative H\"older continuous functions. We begin with the observation that if \(f\in C^\alpha(\Omega)\) for an arbitrary domain \(\Omega\), then \(|f|\in C^\alpha(\Omega)\) with \([|f|]_\alpha\leq [f]_\alpha\). This is because
\[
    ||f(x)|-|f(y)||\leq |f(x)-f(y)|\leq [f]_\alpha|x-y|^\alpha.
\]
For \(k\geq 1\) however, it is not true in general that \(|f|\in C^{k,\alpha}(\Omega)\) whenever \(f\in C^{k,\alpha}(\Omega)\). Indeed if \(f\) is a function in \(C^{1,\alpha}(\Omega)\) which has a zero that is not also a local minimum point, then \(|f|\) will have a cusp. It follows that a first-order derivative of \(|f|\) may not even be continuous, let alone \(\alpha\)-H\"older continuous for some \(\alpha>0\).

We avoid this difficulty by restricting attention to non-negative functions for the remainder of this chapter.Our first result concerning the behaviour of the powers and roots of non-negative H\"older continuous functions is the following.

\begin{lem}\label{lem:smallpow}
Suppose that \(f\) is a non-negative function and let \(0\leq\beta\leq 1\). Then for any \(x,y\) in the domain of \(f\),
\begin{equation}\label{eq:concave}
    |f(x)^\beta-f(y)^\beta|\leq |f(x)-f(y)|^\beta.
\end{equation}
\end{lem}

\begin{proof}
First we show that if \(\beta\leq 1\) then \((a+b)^\beta\leq a^\beta+b^\beta\) for \(a,b\geq 0\). Set \(k=a^\beta/(a^\beta+b^\beta)\), and observe that \(k\leq k^\beta\) and \((1-k)\leq (1-k)^\beta\) since \(k\leq 1\). Therefore
\[
    \frac{(a+b)^\beta}{a^\beta+b^\beta}=(k^\frac{1}{\beta}+(1-k)^\frac{1}{\beta})^\beta\leq (k+1-k)^\beta=1,
\]
giving the claimed estimate. Now we take \(a=|f(x)-f(y)|\) and \(b=f(y)\) to obtain the bound
\[
    f(x)^\beta \leq (|f(x)-f(y)|+|f(y)|)^\beta\leq |f(x)-f(y)|^\beta+f(y)^\beta.
\]
Thus \(f(x)^\beta-f(y)^\beta\leq |f(x)-f(y)|^\beta\), and interchanging \(x\) and \(y\) gives inequality \eqref{eq:concave}.
\end{proof}

\begin{cor}\label{cor:easycase}
If \(f\in C^\alpha(\Omega)\) is non-negative and \(\beta\leq 1\), then \(f^\beta\in C^{\alpha\beta}(\Omega)\) and \([f^\beta]_{\alpha\beta}\leq [f]_\alpha^\beta\). In particular, if \(f\in C^\alpha(\Omega)\) then \(\sqrt{f}\in C^\frac{\alpha}{2}(\Omega)\).
\end{cor}

We can obtain a similar estimate, albeit a slightly more complicated one, when \(\beta>1\).

\begin{lem}\label{lem:almostLipschitz}
Suppose that \(f\in C^{\alpha}(\Omega)\) is non-negative and let \(\beta>1\) and \(\varepsilon>0\). Then for any two points \(x,y\in\Omega\),
\[
    |f(x)^\beta-f(y)^\beta|\leq \frac{(1+\varepsilon)[f]_\alpha^\beta}{((1+\varepsilon)^\frac{1}{\beta-1}-1)^{\beta-1}}|x-y|^{\alpha\beta}+\varepsilon\max\{f(x)^\beta,f(y)^\beta\}.
\]
\end{lem}

\begin{proof}
For a number \(0<\lambda<1\) to be chosen momentarily, observe that Jensen's inequality gives
\[
    f(x)^\beta\leq\frac{1}{\lambda^\beta}(\lambda|f(x)-f(y)|+\lambda f(y))^\beta\leq \frac{1}{\lambda^{\beta-1}}|f(x)-f(y)|^\beta+\frac{1}{(1-\lambda)^{\beta-1}} f(y)^\beta.
\]
Rearranging and applying the H\"older estimate \(|f(x)-f(y)|^\beta\leq [f]_\alpha^\beta|x-y|^{\alpha\beta}\) thus shows that
\[
    f(x)^\beta-f(y)^\beta\leq \frac{[f]_\alpha^\beta}{\lambda^{\beta-1}}|x-y|^{\alpha\beta}+\bigg(\frac{1}{(1-\lambda)^{\beta-1}}-1\bigg)f(y)^\beta.
\]
Next select \(\lambda=1-(1+\varepsilon)^{-\frac{1}{\beta-1}}\), so that the coefficient on the last term is exactly \(\varepsilon\) and
\[
    f(x)^\beta-f(y)^\beta\leq \frac{(1+\varepsilon)[f]_\alpha^\beta}{((1+\varepsilon)^\frac{1}{\beta-1}-1)^{\beta-1}}|x-y|^{\alpha\beta}+\varepsilon f(y)^\beta.
\]
Interchanging \(x\) and \(y\) and then taking a maximum, we obtain the desired result.
\end{proof}

By taking \(\beta=\frac{1}{\alpha}\) we obtain the following special case that will be useful to us in \Cref{chap:decomps}.

\begin{cor}
If \(f\in C^\alpha(\Omega)\) is non-negative, then \(f^\frac{1}{\alpha}\) is `almost' Lipschitz, in the sense that for every \(\varepsilon>0\) there exists a constant \(C_\varepsilon\) such that for all \(x,y\) in the domain of \(f\), 
\[
    |f(x)^\frac{1}{\alpha}-f(y)^\frac{1}{\alpha}|\leq C_\varepsilon|x-y|+\varepsilon\max\{f(x),f(y)\}^\frac{1}{\alpha}.
\]
\end{cor}

A critical property of non-negative H\"older functions is that they satisfy some pointwise derivative estimates. Before making this idea precise in \Cref{thm:cauchylike}, we introduce some helpful notation that will be employed throughout the rest of this work. 

By \(\nabla^kf\) we mean a vector comprised of all \(k^\mathrm{th}\)-order derivatives of \(f\), and for our purposes their ordering within the vector is not important. Given \(\xi\in\mathbb{R}^n\) with \(|\xi|=1\), we denote by \(\partial^k_\xi f\) the \(k^{th}\)-order directional derivative of \(f\) in the direction of \(\xi\). Explicitly this directional operator can be written using in the form
\begin{equation}\label{eq:directionalop}
    \partial^k_\xi=(\xi\cdot\nabla)^k=\sum_{|\beta|=k}\frac{k!}{\beta!}\xi^\beta\partial^\beta,
\end{equation}
where \(\nabla\) is the usual gradient operator and \(\xi^\beta=\xi_1^{\beta_1}\cdots\xi_n^{\beta_n}\). For a real-valued function \(f\) we define the positive part of \(f\) as \([f]_+=\max\{f,0\}\). Using this notation we can compactly express the maximal positive \(k^\mathrm{th}\)-order directional derivative of \(f\) at a point \(x\in\mathbb{R}^n\) by writing
\[
    \sup_{|\xi|=1}[\partial^j_\xi f(x)]_+.
\]
This quantity appears frequently in the next chapter. It is also worth noting that the functions \(|\nabla^k f(x)|\) and \(\displaystyle\sup_{|\xi|=1}|\partial^j_\xi f(x)|\) are equivalent, in the sense that
\[
    c|\nabla^k f(x)|\leq \sup_{|\xi|=1}|\partial^k_\xi f(x)|\leq C|\nabla^k f(x)|
\]
for \(x\) in the domain of \(f\) and for constants \(c\) and \(C\) which depend only on \(n\) and \(k\).

Following a standard convention in analysis, we henceforth use \(C\) to denote a positive finite constant which depends only on fixed parameters such as \(n\), \(k\) and \(\beta\). The value of such constants may change from line-to-line, and our only concern is that they are finite. This is because the constants we encounter are often complicated, unwieldy, and far from sharp -- for our purposes, their values are also usually unimportant.

Momentarily we state the main result of this section, which allows us to control the derivatives of non-negative functions in higher-order H\"older spaces by their even-order derivatives. This idea is not new; Fefferman \& Phong point out in \cite{Fefferman-Phong} that the even-order terms that appear in the Taylor series of non-negative function control the odd-order terms, and it is a straightforward consequence of this fact that similar control estimates hold for the Taylor polynomials of H\"older continuous functions. There are results to this effect in \cite{Tataru}, \cite{SOS_I}, and \cite{Bony2}, and in many other works.

Our next result makes this control explicit for non-negative \(C^{k,\alpha}(\mathbb{R}^n)\) functions and it generalizes several results in the aforementioned works. The main difficulty in its proof is the selective cancellation of odd-order terms from Taylor polynomials, which we overcome by employing the very useful weighted summation identities that we establish in \Cref{thm:specialodds}.

\begin{thm}\label{thm:cauchylike}
If \(f\in C^{k,\alpha}(\mathbb{R}^n)\) is non-negative, then for every \(\ell\leq k\) there exists a constant \(C\) for which the following estimate holds on \(\mathbb{R}^n\),
\[
    |\nabla^\ell f(x)|\leq C\max_{\substack{0\leq j\leq k\\j\;\mathrm{even}}}\bigg\{\sup_{|\xi|=1}[\partial^j_\xi f(x)]_+^\frac{k-\ell+\alpha}{k-j+\alpha}\bigg\}.
\]
\end{thm}

\textit{Remark}: Recall that for notational convenience we regard \(f\) as a zero-order derivative of itself. If \(k=1\) the inequality above reduces to the Malgrange inequality for \(C^{1,\alpha}(\mathbb{R}^n)\), a sharper form of which is proved in \Cref{sec:malg}. The bound above is somewhat complicated, however we actually apply this result using the considerably simpler inequality of \Cref{cor:controlbound}.

\begin{proof}
For an even number \(\ell\leq k\), we begin by bounding the negative part of the \(\ell^\mathrm{th}\)-order directional derivatives of \(f\). To do this we use non-negativity of the Taylor polynomial of \(f\) to obtain the following inequality for \(\lambda\in\mathbb{R}\) and \(\mathbb{\xi}\in\mathbb{R}^n\) satisfying \(|\xi|=1\),
\[
    0\leq f(x+\lambda\xi)\leq  \sum_{j=0}^k\frac{\lambda^j}{j!}\partial^j_\xi f(x)+\frac{1}{k!}[\nabla^k f]_\alpha|\lambda|^{k+\alpha}.
\]
Additionally we can use the fact that \(0\leq f(x-\lambda\xi)\), to see that the same estimate holds when \(\lambda\) is replaced with \(-\lambda\), and summing the Taylor polynomials of \(f\) at \(x+\lambda\xi\) and \(x-\lambda\xi\), we see for any \(\lambda>0\) that the following estimate holds,
\[
    0\leq \sum_{\substack{0\leq j\leq k,\\ j\;\textrm{even}}}\frac{\lambda^j}{j!}\partial^j_\xi f(x)+\frac{2}{k!}[\nabla^k f]_\alpha\lambda^{k+\alpha}\leq \sum_{\substack{0\leq j\leq k,\\ j\;\textrm{even}}}\frac{\lambda^j}{j!}[\partial^j_\xi f(x)]_++\frac{2}{k!}[\nabla^k f]_\alpha\lambda^{k+\alpha}.
\]
The second inequality above holds since \(\partial^j_\xi f(x)\leq [\partial^j_\xi f(x)]_+\) for all \(j\) an \(x\in\mathbb{R}^n\). Rearranging the inequality above gives
\[
    -\partial_\xi^\ell f(x)\leq\sum_{\substack{0\leq j\leq k,\\ j\;\textrm{even},\;j\neq\ell}}\frac{\ell!}{j!}\lambda^{j-\ell}[\partial^j_\xi f(x)]_++\frac{2\ell!}{k!}[\nabla^k f]_\alpha\lambda^{k-\ell+\alpha}
\]
for even \(\ell\leq k\).
Taking \(\lambda=\displaystyle\max_{\substack{0\leq j\leq k,\\ j\;\textrm{even},\;j\neq\ell}}[\partial^j_\xi f(x)]_+^\frac{1}{k-j+\alpha}\), we have \(\lambda^{j-\ell}[\partial^j_\xi f(x)]_+\leq \lambda^{k-\ell+\alpha}\) and so
\[
    -\partial_\xi^\ell f(x)\leq C\lambda^{k-\ell+\alpha}=C\max_{\substack{0\leq j\leq k,\\ j\;\textrm{even},\;j\neq\ell}}\bigg\{[\partial^j_\xi f(x)]_+^\frac{k-\ell+\alpha}{k-j+\alpha}\bigg\}.
\]
Taking a supremum over all \(\xi\in\mathbb{R}^n\) such that \(|\xi|=1\), and using equivalence of the norms \(|\nabla^kf(x)|\) and \(\sup_{|\xi|=1}|\partial^k_\xi f(x)|\), we see now that for even values of \(\ell\)
\[
    |\nabla^\ell f(x)|\leq C\sup_{|\xi|=1}[\partial_\xi^\ell f(x)]_++C\max_{\substack{0\leq j\leq k,\\ j\;\textrm{even},\;j\neq\ell}}\bigg\{\sup_{|\xi|=1}[\partial^j_\xi f(x)]_+^\frac{k-\ell+\alpha}{k-j+\alpha}\bigg\}.
\]
The the first term on the right-hand side above can be absorbed into the second, at the expense of making \(C\) larger. It follows that the claimed estimate holds for derivatives of even order.

For derivatives of odd order, matters are slightly more complicated since we cannot begin by cancelling out all odd terms simultaneously. Once again we begin by using a Taylor expansion and non-negativity of \(f\) to see that for any \(\lambda\in\mathbb{R}\), the following holds when \(|\xi|=1\),
\[
    0\leq \sum_{j=0}^k\frac{\lambda^j}{j!}\partial^j_\xi f(x)+\frac{1}{k!}[\nabla^kf]_\alpha|\lambda|^{k+\alpha}.
\]
This inequality continues to hold if we replace \(\lambda\) with \(\eta\lambda\) for any \(\eta\in\mathbb{R}\), since \(f\) is non-negative everywhere. First let \(\ell\) be the largest odd number that is less than or equal to \(k\) and set \(s=\frac{\ell+1}{2}\). Replacing  \(\lambda\) with \(\eta_1\lambda\) through \(\eta_s\lambda\) for some real numbers \(\eta_1,\dots,\eta_s\) to be chosen momentarily, and adding the resulting inequalities scaled by positive constants \(q_1,\dots,q_s\) which we also select below, we find that
\[
    0\leq \sum_{j=0}^k\bigg(\sum_{i=1}^sq_i\eta_i^j\bigg)\frac{\lambda^j}{j!}\partial^j_\xi f(x)+C\lambda^{k+\alpha}.
\]
The preceding inequality continues to hold if we replace \(\lambda\) with \(-\lambda\), allowing us to write
\[
    \bigg|\sum_{\substack{1\leq j\leq \ell,\\ j\;\textrm{odd}}}\bigg(\sum_{i=1}^sq_i\eta_i^j\bigg)\frac{\lambda^j}{j!}\partial^j_\xi f(x)\bigg|\leq C\bigg(\sum_{\substack{0\leq j\leq k,\\ j\;\textrm{even}}}\lambda^j\partial^j_\xi f(x)+\lambda^{k+\alpha}\bigg)
\]
for every \(\lambda>0\). For brevity we denote by \(F_\lambda(x)\) the function on the right-hand side above, and we next wish to choose the numbers \(q_i\) and \(\eta_i\) so that all but one term in the sum on the left vanishes. This affords a pointwise bound on an odd-order derivative by \(F_\lambda\). Using the technical \Cref{thm:specialodds}, we choose the numbers \(\eta_1,\dots,\eta_s\in\mathbb{R}\) and \(q_1,\dots,q_s\geq 0\) so that \(\sum_{i=1}^sq_i\eta_i^j=0\) for every odd \(j<\ell\), and so that \(\sum_{i=1}^sq_i\eta_i^\ell=1\). With these selections, get \(|\lambda^\ell\partial^\ell_\xi f(x)|\leq C F_\lambda(x)\).

Equipped with this estimate, we can repeat the argument above using different constants \(\tilde{q}_1,\dots,\tilde{q}_{s-1}\geq 0\) and \(\tilde{\eta}_1,\dots\tilde{\eta}_{s-1}\in\mathbb{R}\), and by combining non-negative Taylor polynomials we get
\[
    \bigg|\sum_{\substack{1\leq j\leq \ell-2,\\ j\;\textrm{odd}}}\bigg(\sum_{i=1}^{s-1}\tilde{q}_i\tilde{\eta}_i^j\bigg)\frac{\lambda^j}{j!}\partial^j_\xi f(x)\bigg|\leq C|\lambda^\ell\partial^\ell_\xi f(x)|+CF_\lambda(x)\leq CF_\lambda(x).\\
\]
Using \Cref{thm:specialodds} once more, we can choose the constants \(\tilde{q}_i\) and \(\tilde{\eta}_i\) so that the right-hand side above reduces to \(|\lambda^{\ell-2}\partial^{\ell-2}_\xi f(x)|\). This gives \(|\lambda^{\ell-2}\partial^{\ell-2}_\xi f(x)|\leq CF_\lambda(x)\), and we see by repeating this argument that for each odd \(\ell\leq k\) and \(\lambda>0\),
\[
    |\partial_\xi^\ell f(x)|\leq  CF_\lambda(x)\leq  C\bigg(\sum_{\substack{0\leq j\leq k,\\ j\;\textrm{even}}}\lambda^{j-\ell}[\partial^j_\xi f(x)]_++\lambda^{k-\ell+\alpha}\bigg).
\]
Taking \(\lambda=\displaystyle\max_{\substack{0\leq j\leq k,\\ j\;\textrm{even}}}\{|\partial_\xi^jf(x)|^\frac{1}{k-j+\alpha}\}\) in this estimate gives \(|\partial_\xi^\ell f(x)|\leq C\lambda^{k-\ell+\alpha}\), and it follows that
\[
    |\nabla^\ell f(x)|\leq C\sup_{|\xi|=1}|\partial^\ell_\xi f(x)|\leq C\sup_{|\xi|=1}\lambda^{k-\ell+\alpha}\leq C\max_{\substack{0\leq j\leq k\\j\;\mathrm{even}}}\bigg\{\sup_{|\xi|=1}[\partial^j_\xi f(x)]_+^\frac{k-\ell+\alpha}{k-j+\alpha}\bigg\}.
\]
Thus the claimed derivative estimates hold for every order \(\ell\leq k\), and we are finished.
\end{proof}

We note in passing that for the small values of \(k\), it is possible to avoid invoking \Cref{thm:specialodds} in the preceding proof, and instead to combine inequalities methodically and thereby achieve the appropriate cancellation of odd-order terms. We chose to prove the result above in greater generality than is needed in \Cref{chap:holderchap}, since \Cref{thm:cauchylike} is useful when generalizing sum of squares theorems to \(C^{k,\alpha}(\mathbb{R}^n)\). The result has applications elsewhere, for instance in studying the coefficients of non-negative polynomials.

\section{Generalized Malgrange Inequalities}\label{sec:malg}

In this section we study pointwise control of derivatives of non-negative \(C^{1,\alpha}(\mathbb{R}^n)\) functions. The classical Malgrange inequality in \(\mathbb{R}\) states that \(|f'(x)|\leq C\sqrt{f(x)}\) provided that \(f''(x)\) exists and is bounded on the real line. As we noted above, this inequality is often attributed to Glaeser in the literature, however Glaeser credits it to Malgrange in \cite{Glaeser}, so we refer to the similar pointwise estimates we prove in this section as generalized Malgrange inequalities.

The inequality we now prove was also found by Ghisi \& Gobbino in \cite{Gobbino}, and their work was made known to the author after it was discovered independently in this work. Our result for functions defined on \(\mathbb{R}\) is not as strong as that in \cite{Gobbino}, however our proof is considerably shorter and our subsequent generalizations do appear to be new to the literature.

\begin{thm}\label{thm:malgrangeineq}
Suppose that \(f\in C^{1,\alpha}(\mathbb{R})\) is non-negative. Then for each \(x\in\mathbb{R}\), 
\begin{equation}\label{eq:malg}
    |f'(x)|\leq \frac{\alpha+1}{\alpha^\frac{\alpha}{1+\alpha}}[f']_\alpha^\frac{1}{1+\alpha}f(x)^\frac{\alpha}{1+\alpha}.    
\end{equation}
\end{thm}

\begin{proof} 
For any \(x,h\in\mathbb{R}\), we first have by Taylor's Theorem and non-negativity of \(f\) that
\[
    0\leq f(x+h)\leq f(x)+hf'(x)+[f']_\alpha|h|^{1+\alpha}.
\]
Since \(0\leq f(x-h)\), we can replace \(h\) with \(-h\) to get \(|hf'(x)|\leq f(x)+|h|^{1+\alpha}[f']_\alpha\). Taking \(h=(f(x)/[f']_\alpha)^\frac{1}{1+\alpha}\alpha^{-\frac{1}{1+\alpha}}\), we find that
\[
    |f'(x)|\leq \frac{f(x)}{|h|}+|h|^\alpha[f']_\alpha=[f']_\alpha^\frac{1}{1+\alpha}f(x)^\frac{\alpha}{1+\alpha}(\alpha^{\frac{1}{1+\alpha}}+\alpha^{-\frac{\alpha}{1+\alpha}}).
\]
Simplifying the right-hand side gives inequality \eqref{eq:malg}.
\end{proof}

In fact, a simple modification of the preceding proof shows that for any \(\eta>0\) and \(x,\xi\in\mathbb{R}\) the following slightly stronger estimate actually holds:
\[
    |f'(\xi)|\leq \frac{1}{\eta}f(x)^\frac{\alpha}{1+\alpha}+[f']_\alpha(\eta f(x)^\frac{1}{1+\alpha}+|x-\xi|)^\alpha.
\]
\Cref{thm:malgrangeineq} follows as a special case. Taking \(\alpha=1\) we get \(|f'(x)|\leq 2\sqrt{[f']_1f(x)}\) whenever \(f\in C^{1,1}(\Omega)\), which resembles the classical Malgrange inequality. Later, we will find conditional extensions of this inequality to higher dimensions as well as to higher-order derivatives. We note in passing that we cannot simply apply the preceding estimate to \(f''(x)\) and iterate \eqref{eq:malg} for higher-order derivative estimates, since this would require \(f\) to be monotone.

This generalization of Malgrange's inequality yields several interesting consequences. An application of Young's inequality to \eqref{eq:malg} affords the pointwise bound \(|f'(x)|\leq [f']_\alpha+f(x)\), from which it follows that \(\|f'\|_\infty\leq [f']_\alpha+\|f\|_\infty\). Hence the norm on \(C^{1,\alpha}_b(\mathbb{R})\) defined in \eqref{eq:holdernorm} is equivalent to \(\|f\|_\infty+[f']_\alpha\) on the cone of non-negative \(C^{1,\alpha}(\mathbb{R})\) functions. Using the Malgrange inequality, we are also able to discuss the regularity (i.e. the degree of smoothness) of square roots of \(C^{1,\alpha}(\mathbb{R})\) functions.

\begin{thm}\label{thm:rootsin}
Let \(f\in C^{1,\alpha}(\mathbb{R})\) be non-negative. Then \(\sqrt{f}\in C^{\frac{1+\alpha}{2}}(\mathbb{R})\).
\end{thm}

\begin{proof}
We first use the Malgrange inequality to show that \(f^\frac{1}{1+\alpha}\in C^{0,1}(\mathbb{R})\). To this end, we apply \Cref{lem:smallpow} to get
\[
    |\sqrt{f(x)}-\sqrt{f(y)}|^\frac{2}{1+\alpha}\leq |f(x)^\frac{1}{1+\alpha}-f(y)^\frac{1}{1+\alpha}|.
\]
Applying the mean value theorem on the right-hand side and then applying inequality \eqref{eq:malg} at some intermediate point \(\xi\) between \(x\) and \(y\) gives
\[
    |f(x)^\frac{1}{1+\alpha}-f(y)^\frac{1}{1+\alpha}|= \frac{1}{1+\alpha}\frac{|f'(\xi)|}{f(\xi)^\frac{\alpha}{1+\alpha}}|x-y|\leq \frac{[f']_\alpha^\frac{1}{1+\alpha}}{\alpha^\frac{1}{1+\alpha}}|x-y|.
\]
Combining the preceding inequalities we get \(|\sqrt{f(x)}-\sqrt{f(y)}|^\frac{2}{1+\alpha}\leq[f']_\alpha^\frac{1}{1+\alpha}|x-y|/\alpha^\frac{1}{1+\alpha}\), which gives the required estimate after taking a root.
\end{proof}

In addition to showing that \(\sqrt{f}\) is half as regular as \(f\) in the sense of \eqref{eq:rootspaces}, the proof above also yields the semi-norm estimate 
\[
    \big[\sqrt{f}\big]_{\frac{1+\alpha}{2}}\leq \sqrt{\frac{[f']_\alpha}{\alpha}}.
\]

As a special case of the preceding theorem, we also have the following result.

\begin{cor}
If \(f\) is non-negative and \(f'\) is Lipschitz then \(\sqrt{f}\) is also Lipschitz.
\end{cor}

In summary, if \(f\in C^{\alpha}(\mathbb{R})\) then \(\sqrt{f}\in C^\frac{\alpha}{2}(\mathbb{R})\), while if \(f\in C^{1,\alpha}(\mathbb{R})\) then \(\sqrt{f}\in C^\frac{1+\alpha}{2}(\mathbb{R})\). This result is essentially sharp, in the sense that it does not hold in higher-order H\"older spaces. Recall from \Cref{chap:intro} that in \cite{Bony}, Bony gives an example of a \(C^\infty(\mathbb{R})\) function whose root is not \(C^{1,\alpha}(\mathbb{R})\) for any \(\alpha>0\).

Indeed, if \(f\in C^\infty(\mathbb{R})\) them \(\sqrt{f}\) need not even belong to \(C^1(\mathbb{R})\). It follows that if \(f\in C^{k,\alpha}(\mathbb{R})\) for \(k\geq 2\) then we cannot even expect continuity of the derivative of \(\sqrt{f}\), let alone H\"older continuity. For instance, the function 
\begin{equation}\label{eq:badroot}
    f(x)=\begin{cases}
    x^2e^\frac{1}{x^2-1} & |x|<1,\\
    0 & |x|\geq 1
    \end{cases}
\end{equation}
is a bounded and non-negative \(C^\infty(\mathbb{R})\) function, however \(\sqrt{f}\) is not continuously differentiable at the origin. In \(\mathbb{R}^n\) we can similarly define \(F(x)=f(x_1)\cdots f(x_n)\) for \(f\) as above to obtain a bounded and non-negative \(C^\infty(\mathbb{R}^n)\) function whose square root is not in \(C^1(\mathbb{R}^n)\).

In any number of dimensions, a sum of squares decomposition of \(f\in C^{k,\alpha}(\mathbb{R}^n)\) for \(k\geq 2\) will in general require more than one root function if those roots are to retain half of the regularity of \(f\). Before moving on to the task of constructing such decompositions, we first verify that the \(C^{1,\alpha}(\mathbb{R})\) results proved in this section hold also in \(\mathbb{R}^n\) for \(n\geq 2\).

\begin{thm}
Suppose that \(f\in C^{1,\alpha}(\mathbb{R}^n)\) is non-negative everywhere. Then \(f\) satisfies 
\[
    |\nabla f(x)|\leq \frac{\alpha+1}{\alpha^\frac{\alpha}{1+\alpha}}[\nabla f]_\alpha^\frac{1}{1+\alpha}f(x)^\frac{\alpha}{1+\alpha}.
\]
\end{thm}

\begin{proof}
Using a Taylor expansion in \(\mathbb{R}^n\) for non-negative \(f\), we have for any \(x,h\in\mathbb{R}^n\) and some \(\xi\) on the line segment between \(x\) and \(x+h\) that
\[
    0\leq f(x+h)=f(x)+h\cdot\nabla f(\xi)=f(x)+h\cdot\nabla f(x)+h\cdot(\nabla f(\xi)-\nabla f(x)).
\]
Applying the Cauchy-Schwarz inequality on the last term above and also using the fact that \(|\nabla f|\in C^{\alpha}(\mathbb{R}^n)\), we get
\[
    0\leq f(x)+h\cdot\nabla f(x)+|h||\nabla f(x)-\nabla f(\xi)|\leq f(x)+h\cdot\nabla f(x) +[\nabla f]_\alpha|h|^{1+\alpha}.
\]
Since \(f\) is non-negative everywhere, an identical estimate holds if we replace \(h\) with \(-h\), showing that \(|h\cdot\nabla f(x)|\leq f(x) +[\nabla f]_\alpha|h|^{1+\alpha}\). Taking \(h=\lambda\nabla f(x)\) for \(\lambda>0\) gives \(h\cdot\nabla f(x)=\lambda|\nabla f(x)|^2\), and we find that
\[
    \lambda|\nabla f(x)|^2\leq f(x) +\lambda^{1+\alpha}[\nabla f]_\alpha|\nabla f(x)|^{1+\alpha}.
\]
Now choosing \(\lambda=(f(x)/\alpha[\nabla f]_\alpha)^\frac{1}{1+\alpha}/|\nabla f(x)|\) gives \(|\nabla f(x)|\leq [\nabla f]_\alpha^\frac{1}{1+\alpha}f(x)^\frac{\alpha}{1+\alpha}(\alpha^\frac{1}{1+\alpha}+\alpha^{-\frac{\alpha}{1+\alpha}})\). Hence \eqref{eq:malg} continues to hold in \(\mathbb{R}^n\) as claimed.
\end{proof} 

As in the one-dimensional case, this pointwise estimate gives us information about the square roots of non-negative \(C^{1,\alpha}(\mathbb{R}^n)\) functions.

\begin{cor}
Let \(f\in C^{1,\alpha}(\mathbb{R}^n)\). If \(f\) is non-negative then \(\sqrt{f}\in C^{\frac{1+\alpha}{2}}(\mathbb{R}^n)\).
\end{cor}

The proof is identical to that of \Cref{thm:rootsin}, so we do not repeat it. Further generalizations of the Malgrange inequality, in particular to a general modulus of continuity, have been obtained by the author. However, such estimates go beyond the scope of this thesis, so we do not discuss them here.

We close this section by noting that we may weaken the hypothesis of non-negativity in the Malgrange inequality, and assume simply that \(f\) is bounded below. Specifically, given \(f\in C^{1,\alpha}(\mathbb{R}^n)\), not necessarily non-negative but bounded below, we can replace \(f\) with \(f-\inf_{\mathbb{R}^n}f\), and this is a non-negative \(C^{1,\alpha}(\mathbb{R}^n)\) function. An application of inequality thus \eqref{eq:malg} gives
\[
    |\nabla f(x)|\leq C\big(f(x)-\inf_{\mathbb{R}^n} f\big)^\frac{\alpha}{1+\alpha}
\]
for every \(x\in\mathbb{R}^n\). For non-negative functions this is actually stronger than \eqref{eq:malg}, and a similar generalization can easily be obtained for functions that are bounded above.

In summary, we have shown that if \(f\in C^{k,\alpha}(\mathbb{R}^n)\) is non-negative then \(\sqrt{f}\in C^\frac{k+\alpha}{2}(\mathbb{R}^n)\) for \(k=0\) and \(k=1\). Now we are free to focus on the cases \(k=2\) and \(k=3\), in which taking a single root does not suffice for a decomposition to preserve regularity.

\section{Multivariate Calculus Identities}

For the remainder of this section we generalize the chain rule from elementary differential calculus, and adapt it to arbitrary derivatives of functions defined on \(\mathbb{R}^n\). The identities we prove are critical for showing that the functions \(g_1,\dots,g_m\) appearing in \Cref{thm:c3main} of \Cref{chap:decomps} have the desired regularity properties.

The main technique we use in this section involves expressing derivatives of compositions of functions by summing over multi-index partitions. This idea was introduced by Hardy in \cite{hardy} in a study of the combinatorics induced by higher-order derivatives. A similar and equally effective result is given in \cite[Eq. (3.4)]{SOS_I} and proved in \cite{masur}, and we emphasize that many expressions for these higher-order derivatives can be found in the literature and are fairly well-known. Our contribution is to make the computation of the corresponding coefficients explicit, rather than defining them recursively or via a difficult combinatorial exercise.

Further, we generalize our result to compositions of two functions of multiple variables, whereas many well-known versions of the chain rule consider compositions of multivariate functions with a function of a single variable. In particular, we require this generalization to study implicitly defined functions and their derivatives.

To state these general variants of the chain rule in closed form and prove them using appropriate combinatorial arguments, we introduce some notation. Recall that a multi-index \(\beta\) of length \(n\) is an \(n\)-tuple of non-negative integers. A multi-set \(\Gamma\) that is comprised of multi-indices is called a partition of \(\beta\) if each \(\gamma\in\Gamma\) is a proper multi-index (i.e. each entry of \(\gamma\) is a non-negative integer and at least one is non-zero), and if
\[
    \sum_{\gamma\in\Gamma}\gamma=\beta.
\]
Multi-sets permit repetition, so several multi-indices in the sum above can be identical. 

We identify the support of a partition \(\Gamma\), denoted \(\mathrm{supp}(\Gamma)\), as the set obtained by omitting repetition from \(\Gamma\). We let \(m(\Gamma,\gamma)\) be the number of times that a multi-index \(\gamma\) appears in \(\Gamma\), and counting repetition, we let \(|\Gamma|\) denote the cardinality of the multi-set \(\Gamma\) so that
\[
    \sum_{\gamma\in\mathrm{supp}(\Gamma)}m(\Gamma,\gamma)=\sum_{\gamma\in\Gamma}1=|\Gamma|.
\]
Finally we denote by \(P(\beta)\) the collection of all unordered partitions of \(\beta\). For instance, if \(\beta=(1,2)\) then this partition set is given explicitly by
\[
    P(\beta)=\big\{\{(1,2)\},\{(1,1),(0,1)\},\{(1,0),(0,2)\}\\
    \{(1,0),(0,1),(0,1)\}
    \big\}.
\]
Thankfully, for our purposes it is unnecessary to actually compute the partition set of a given multi-index; we simply use their intrinsic properties and the concise formulas they afford.

Briefly we recall some notation introduced in the previous section, which will be employed heavily to prove the subsequent results. Given a multi-index \(\beta=(\beta_1,\dots,\beta_n)\) of length \(n\) we write \(\beta!=\beta_1!\cdots\beta_n!\), and if \(\gamma\) is another multi-index we say that \(\gamma\leq\beta\) if \(\gamma_j\leq\beta_j\) for each \(j=1,\dots,n\), and \(\gamma<\beta\) if \(\gamma\leq\beta\) and \(\gamma\neq\beta\). Given \(\gamma\leq\beta\) we can define a generalized binomial coefficient as in \eqref{eq:binom}, and with this notation we express the following result.

\begin{lem}[Generalized Chain Rule]\label{lem:genchain}
Let \(f:\mathbb{R}^{n+1}\rightarrow\mathbb{R}\) and \(g:\mathbb{R}^{n}\rightarrow\mathbb{R}\) both be \(k\) times differentiable, and define \(h:\mathbb{R}^n\rightarrow\mathbb{R}\) by setting \(h(x)=f(x,g(x))\). For any multi-index \(\beta\) of order \(|\beta|\leq k\) and length \(n\), there exist constants \(C_{\beta,\Gamma}\) for which
\begin{equation}\label{eq:holycow}
    \partial^\beta h =\sum_{0\leq\eta\leq\beta}\sum_{\Gamma\in P(\eta)}C_{\beta,\Gamma}(\partial^{\beta-\eta}\partial^{|\Gamma|}_{n+1}f)\prod_{\gamma\in\Gamma}\partial^\gamma g,
\end{equation}
where \(\partial^\beta h\) is evaluated at \(x\) and the functions on the right-hand side are evaluated at \((x,g(x))\). Moreover, the constant \(C_{\beta,\Gamma}\) in \eqref{eq:holycow} is given explicitly by the formula
\begin{equation}\label{eq:cm}
    C_{\beta,\Gamma}=\eta!\binom{\beta}{\eta}\bigg(\prod_{\gamma\in\Gamma}\gamma!\bigg)^{-1}\bigg(\prod_{\gamma\in\mathrm{supp}(\Gamma)}m(\Gamma,\gamma)!\bigg)^{-1}.
\end{equation}
\end{lem}

\begin{proof}
We use induction on \(k\) to prove that \eqref{eq:holycow} holds with the appropriate constants. The case \(k=1\) follows from the usual chain rule. For an inductive hypothesis assume that the claimed formula holds for derivatives up to order \(k\) and let \(\mu\) be a derivative of order \(k+1\), so that we can write \(\partial^\mu=\partial^\nu\partial^\beta\) for \(|\nu|=1\) and \(|\beta|=k\). By the inductive hypothesis and linearity we have
\[
    \partial^\mu h=\partial^\nu\partial^\beta h=\sum_{0\leq \eta\leq\beta}\sum_{\Gamma\in P(\eta)}C_{\beta,\Gamma}\partial^\nu\bigg((\partial^{\beta-\eta}\partial^{|\Gamma|}_n f)\prod_{\gamma\in\Gamma}\partial^\gamma g\bigg),
\]
where each function on the right-hand side is evaluated at \((x,g(x))\). The derivative with respect to \(\nu\) on the right-hand side is first-order, so applying the standard form of the product rule gives
\[
    \partial^\mu h=\sum_{0\leq\eta\leq\beta}\sum_{\Gamma\in P(\eta)}C_{\beta,\Gamma}\bigg(\partial^\nu(\partial^{\beta-\eta}\partial^{|\Gamma|}_nf)\prod_{\gamma\in\Gamma}\partial^\gamma g+(\partial^{\beta-\eta}\partial^{|\Gamma|}_n f)\prod_{\gamma\in\Gamma}\partial^\gamma g\sum_{\gamma\in\Gamma}\frac{\partial^{\gamma+\nu }g}{\partial^\gamma g}\bigg).
\]
To simplify this, we first observe that by the standard chain rule we have the identity
\[
    \partial^\nu(\partial^{\beta-\eta}\partial^{|\Gamma|}_n f(x,g(x)))=\partial^{\mu-\eta}\partial^{|\Gamma|}_n f(x,g(x))+\partial^{\beta-\eta}\partial^{|\Gamma|+1}_nf(x,g(x))\partial^\nu g(x).
\]
Using this identity together with the calculation preceding it, we obtain the expansion
\begin{equation}\label{eq:lots}
\begin{split}
    \partial^\mu h&=\sum_{0\leq\eta\leq\beta}\sum_{\Gamma\in P(\eta)}C_{\beta,\Gamma}(\partial^{\mu-\eta}\partial^{|\Gamma|}_n f)\prod_{\gamma\in\Gamma}\partial^\gamma g\\
    &\qquad+\sum_{0\leq\eta\leq\beta}\sum_{\Gamma\in P(\eta)}C_{\beta,\Gamma}(\partial^{\mu-(\eta+\nu)}\partial^{|\Gamma|+1}_nf)\partial^\nu g\prod_{\gamma\in\Gamma}\partial^\gamma g\\
    &\qquad+\sum_{0\leq\eta\leq\beta}\sum_{\Gamma\in P(\eta)}C_{\beta,\Gamma}(\partial^{\mu-(\eta+\nu)}\partial^{|\Gamma|}_n f)\prod_{\gamma\in\Gamma}\partial^\gamma g\sum_{\gamma\in\Gamma}\frac{\partial^{\gamma+\nu }g}{\partial^\gamma g}.
\end{split}
\end{equation}

Each term in \eqref{eq:lots} takes the form \(\partial^{\mu-\eta}\partial^{|\Gamma|}_nf\prod_{\gamma\in\Gamma}\partial^\gamma g\) for some \(\eta\leq \mu\) and \(\Gamma\in P(\eta)\), meaning that after grouping common terms we can write \(\partial^\mu h\) in the form of identity \eqref{eq:holycow} for some new constants which we denote by \(C_{\mu,\Gamma}\). Now we compute these coefficients and show that they satisfy \eqref{eq:cm}. To this end we fix \(\eta\leq\mu\) and \(\Gamma\in P(\eta)\), and for brevity, we also define 
\[
    G=\prod_{\gamma\in\Gamma}\gamma!\qquad\textrm{and}\qquad M=\prod_{\gamma\in\mathrm{supp}(\Gamma)}m(\Gamma,\gamma)!.
\]
Consider the term of the form \(\partial^{\mu-\eta}\partial^{|\Gamma|}_nf\prod_{\gamma\in\Gamma}\partial^\gamma g\) in \eqref{eq:lots}, and note that for each \(\gamma\in\mathrm{supp}(\Gamma)\) such that \(\gamma>\nu\), the constant \(C_{\beta,\Gamma\setminus\{\gamma\}\cup\{\gamma-\nu\}}\) appears as a coefficient to this term in the third sum in \eqref{eq:lots} once for every time \(\gamma-\nu\) appears in the partition \(\Gamma\setminus\{\gamma\}\cup\{\gamma-\nu\}\) of \(\beta\). There are exactly \(m(\Gamma\setminus\{\gamma\}\cup\{\gamma-\nu\},\gamma-\eta)\) of these terms. Additionally, \(C_{\Gamma\setminus\{\nu\}}\) appears once as a coefficient from the second sum, and \(C_{\beta,\Gamma}\) once as a coefficient in the first sum of \eqref{eq:lots}. Altogether then, we have the following recursive identity for \(C_{\mu,\Gamma}\),
\begin{equation}\label{eq:murecur}
    C_{\mu,\Gamma}=\sum_{\substack{\gamma\in\mathrm{supp}(\Gamma)\\\gamma>\nu}}m(\Gamma\setminus\{\gamma\}\cup\{\gamma-\nu\},\gamma-\eta)C_{\beta,\Gamma\setminus\{\gamma\}\cup\{\gamma-\nu\}}+C_{\beta,\Gamma\setminus\{\nu\}}+C_{\beta,\Gamma}.
\end{equation}
Fix \(\tilde{\gamma}\in\mathrm{supp}(\Gamma)\) and write \(\tilde{\Gamma}=\Gamma\setminus\{\tilde{\gamma}\}\cup\{\tilde{\gamma}-\nu\}\). By our inductive hypothesis we have
\[
    C_{\beta,\tilde{\Gamma}}=\frac{\beta!}{(\mu-\eta)!}\bigg(\prod_{\gamma\in\tilde{\Gamma}}\gamma!\bigg)^{-1}\bigg(\prod_{\gamma\in\mathrm{supp}(\tilde{\Gamma})}m(\tilde{\Gamma},\gamma)!\bigg)^{-1}.
\]
To simplify \eqref{eq:murecur}, we set \(\nu=(0,\dots,1,\dots,0)\), with the one in position \(\ell\), so that \(\prod_{\gamma\in\tilde{\Gamma}}\gamma!= \frac{G}{\tilde{\gamma}_\ell}\). Similarly, by counting multiplicities of multi-indices in \(\Gamma\) and \(\tilde{\Gamma}\) respectively, we find that
\[
    m(\tilde{\Gamma},\gamma)=\begin{cases}
    m(\Gamma,\gamma)+1 & \gamma=\tilde{\gamma}-\nu,\\
    m(\Gamma,\gamma)-1 & \gamma=\tilde{\gamma},\\
    m(\Gamma,\gamma) & \mathrm{otherwise}.
    \end{cases}
\]
Therefore we have the identity \(\prod_{\gamma\in\mathrm{supp}(\tilde{\Gamma})}m(\tilde{\Gamma},\gamma)!=\frac{m(\Gamma,\tilde{\gamma}-\nu)+1}{m(\Gamma,\tilde{\gamma})}M\). Equipped with these formulas, we see that the coefficients of the sum in the recursive identity \eqref{eq:murecur} can be written as \(m(\tilde{\Gamma},\tilde{\gamma}-\nu)C_{\beta,\tilde{\Gamma}}=\frac{\beta!\tilde{\gamma}_\ell m(\Gamma,\tilde{\gamma})}{(\mu-\eta)!GM}\). Using the inductive hypothesis with \eqref{eq:cm} we also find that
\[
     C_{\beta,\Gamma\setminus\{\nu\}}=\frac{\beta!}{(\mu-\eta)!}\bigg(\prod_{\gamma\in\Gamma\setminus\{\nu\}}\gamma!\bigg)^{-1}\bigg(\prod_{\gamma\in\mathrm{supp}(\Gamma\setminus\{\nu\})}m(\Gamma\setminus\{\nu\},\gamma)!\bigg)^{-1}.
\]
Since \(\nu!=1\) and \(\prod_{\gamma\in\mathrm{supp}(\Gamma\setminus\{\nu\})}m(\Gamma\setminus\{\nu\},\gamma)!=\frac{M}{m(\Gamma,\nu)}\), we have \(C_{\beta,\Gamma\setminus\{\nu\}}=\frac{\beta!m(\Gamma,\nu)}{(\mu-\eta)!GM}\).

Similarly, the inductive hypothesis and \eqref{eq:cm} give \(C_{\beta,\Gamma}=\frac{\beta!(\mu_\ell-\eta_\ell)}{(\mu-\eta)!GM}\), and from \eqref{eq:murecur} we get
\[
    C_{\mu,\Gamma}=\frac{\beta!}{(\mu-\eta)!GM}\bigg(\sum_{\gamma\in\mathrm{supp}(\Gamma)}\tilde{\gamma}_\ell m(\Gamma,\tilde{\gamma})+\mu_\ell-\eta_\ell\bigg)=\frac{\mu!}{(\mu-\eta)!GM}.
\]
The right-hand side takes the form of \eqref{eq:cm}, and since \(\Gamma\) and \(\eta\) were arbitrary we see that \eqref{eq:holycow} holds with \(\mu\) in place of \(\beta\). Moreover, \(\mu\) was any multi-index of order \(k+1\), and it follows by induction that \eqref{eq:holycow} holds for all derivatives of \(h\).
\end{proof}

For the remainder of this section, we explore some direct consequences of \Cref{lem:genchain}. The first of these is the following somewhat less general form of the chain rule.

\begin{cor}[Chain Rule]\label{cor:chain2}
Let \(f:\mathbb{R}\rightarrow\mathbb{R}\) and  \(g:\mathbb{R}^n\rightarrow\mathbb{R}\) both be \(k\) times differentiable, and define \(h:\mathbb{R}^n\rightarrow\mathbb{R}\) by setting \(h(x)=f(g(x))\). If \(|\beta|\leq k\), then 
\begin{equation}\label{eq:orderedcr}
    \partial^\beta h =\sum_{\Gamma\in P(\beta)}C_{\beta,\Gamma}f^{(|\Gamma|)}(g)\prod_{\gamma\in\Gamma}\partial^\gamma g,
\end{equation}
where the constants \(C_{\beta,\Gamma}\) are given by \eqref{eq:cm}.
\end{cor}

By taking the function \(f\) in this corollary to be a square root, we get the following.

\begin{cor}\label{cor:rootscor}
Let \(g:\mathbb{R}^n\rightarrow\mathbb{R}\) be non-negative and \(k\) times differentiable. For \(|\beta|\leq k\),
\[
    \partial^\beta \sqrt{g} =\sum_{\Gamma\in P(\beta)}C_{\beta,\Gamma}g^{\frac{1}{2}-|\Gamma|}\prod_{\gamma\in\Gamma}\partial^\gamma g,
\]
where the constants \(C_{\beta,\Gamma}\) are given by
\[
    C_{\beta,\Gamma}=\frac{(-1)^{1+|\Gamma|}\beta!\left(2|\Gamma|-2\right)!}{2^{2|\Gamma|-1}(|\Gamma|-1)!}\bigg(\prod_{\gamma\in\Gamma}\gamma!\bigg)^{-1}\bigg(\prod_{\gamma\in\mathrm{supp}(\Gamma)}m(\Gamma,\gamma)!\bigg)^{-1}.
\]
\end{cor}

Ultimately the form of the constants above is unimportant for our work; we only include explicit formulas for the convenience of the reader interested in using them for other calculations. The supplementary calculus results established in this section are sufficient for the remainder of the work, and we now move on to establishing results needed to prove out main theorems.

One such result is a recursive formula for the derivatives an implicitly defined function. For completeness, we state the well-known Implicit Function Theorem before providing this recursive formula. In the subsequent proof and henceforth in this thesis, we use the notation \(x=(x',x_n)\) for \(x\in\mathbb{R}^n\).

\begin{thm}[Implicit Function Theorem]\label{thm:IFT}
Let \(G\in C^k(\mathbb{R}^n)\) and let \(x_0\in\mathbb{R}^n\) be a point at which \(G(x_0)=0\) and \(\frac{\partial G}{\partial x_n}(x_0)>0\). There exists a unique \(g\in C^k(U)\), for some neighbourhood \(U\subset\mathbb{R}^{n-1}\) of \(x_0'\), such that \(G(x',g(x'))=0\) for every \(x'\in U\). Further, the derivatives of \(g\) are given recursively by
\[
    \partial^\beta g=-\frac{1}{\partial_{n}G}\sum_{0\leq\eta\leq\beta}\sum_{ \substack{\Gamma\in P(\eta),\\\Gamma\neq \{\beta\}}}C_{\beta,\Gamma}(\partial^{\beta-\eta}\partial^{|\Gamma|}_nG)\prod_{\gamma\in\Gamma}\partial^\gamma g,
\]
where for functions on the right are evaluated at \((x',g(x'))\) and \(C_{\beta,\Gamma}\) is given by \eqref{eq:cm}.
\end{thm}

\begin{proof}
The first part is the standard Implicit Function Theorem, a proof of which can be found in \cite[Theorem 3.2.1]{IFTref}. Since \(G(x',g(x'))\) is identically zero on \(U\), so too are all if its derivatives, meaning that \(\partial^\beta G(x',g(x'))=0\). On the other hand, using \Cref{lem:genchain} we can write 
\[
    \partial^\beta G(x',g(x'))=\sum_{\eta\leq\beta}\sum_{\Gamma\in P(\eta)}C_{\beta,\Gamma}(\partial^{\beta-\eta}\partial^{|\Gamma|}_nG(x',g(x')))\prod_{\gamma\in\Gamma}\partial^\gamma g(x').
\]
The derivative \(\partial^\beta g\) appears in only one term on the right-hand side, namely when \(\eta=\beta\) and \(\Gamma=\{\beta\}\), meaning we can combine the two identities for \(\partial^\beta G(x',g(x'))\) and rearrange to get the claimed derivative formula.
\end{proof}

Due to the recursive nature of the formula above it is a straightforward exercise, albeit a computationally difficult one, to solve for the derivatives of \(g\) in terms of \(G\) alone. For our purposes, such a closed form is unnecessary and the recursive formula given above suffices. 
\chapter{Sum of Squares Decompositions}\label{chap:decomps}

Given a non-negative function \(f\in C^{k,\alpha}(\mathbb{R}^n)\) for \(k\leq 3\) and \(0<\alpha\leq 1\), we now begin working to show that \(f\) can be decomposed as a sum of squares of functions in the H\"older space of functions with `half' the regularity of \(f\). By this, we mean that we can write \(f=g_1^2+\cdots+g_m^2\) for functions \(g_1,\dots,g_m\) which belong to the H\"older space
\begin{equation}\label{eq:halfreg}
    C^\frac{k+\alpha}{2}(\mathbb{R}^n)=
    \begin{cases}
    \hfil C^{\frac{k}{2},\frac{\alpha}{2}}(\mathbb{R}^n) & k\;\textrm{even,}\\
    C^{\frac{k-1}{2},\frac{1+\alpha}{2}}(\mathbb{R}^n) & k\;\textrm{odd}.
    \end{cases}    
\end{equation}
Earlier, we showed that if \(f\in C^{k,\alpha}(\mathbb{R}^n)\) for \(k=0\) or \(k=1\) and \(0<\alpha\leq 1\) then \(\sqrt{f}\in C^\frac{k+\alpha}{2}(\mathbb{R}^n)\). Thus, when \(k\leq 1\) the desired decomposition holds with just one square given by \(g_1=\sqrt{f}\).

Moving on to \(k\geq 2\), the techniques required to prove the existence of a sum of squares decomposition are altogether different and considerably more elaborate. This is because taking a single root does not suffice to preserve regularity for \(k\geq 2\), as we showed with \eqref{eq:badroot}. Our main result in this thesis, which we prove using these more advanced techniques, is the following.

\begin{thm}\label{thm:c3main}
If \(0<\alpha\leq1\) and \(f\in C^{2,\alpha}(\mathbb{R}^n)\) is non-negative, then there exist functions \(g_1,\dots, g_{m_n}\in C^{1,\frac{\alpha}{2}}(\mathbb{R}^n)\) which satisfy
\begin{equation}\label{eq:decompident}
    f=\sum_{j=1}^{m_n}g_j^2.
\end{equation}
Similarly, if \(f\in C^{3,\alpha}(\mathbb{R}^n)\) then \eqref{eq:decompident} holds for functions \(g_1,\dots, g_m\in C^{1,\frac{1+\alpha}{2}}(\mathbb{R}^n)\). Moreover, the number of squares \(m_n\) depends only on the dimension \(n\). 
\end{thm}

Our proof of this result actually goes further, to show that the functions \(g_1,\dots,g_m\) and their derivatives are controlled pointwise in the same way that the derivatives of \(f\) are by \Cref{thm:cauchylike}. Specifically, we show that when \(f\in C^{2,\alpha}(\mathbb{R}^n)\), each \(g_j\) in \eqref{eq:decompident} satisfies the pointwise estimates
\[
    |g_j(x)|\leq C\max\bigg\{\sqrt{f(x)},\sup_{|\xi|=1}[\partial^2_\xi f(x)]_+^\frac{2+\alpha}{2\alpha}\bigg\}\textrm{ and }|\nabla g_j(x)|\leq C\max\bigg\{f(x)^\frac{\alpha}{4+2\alpha},\sup_{|\xi|=1}[\partial^2_\xi f(x)]_+^\frac{1}{2}\bigg\}.
\]
Similar estimates hold when \(f\in C^{3,\alpha}(\mathbb{R}^n)\), and these bounds show that the functions \(g_1,\dots,g_m\) inherit some of the pointwise structure of \(f\). In \Cref{sec:ext} we extend \Cref{thm:c3main} to the local spaces \(C^{2,\alpha}_\mathrm{loc}(\Omega)\) and \(C^{3,\alpha}_\mathrm{loc}(\Omega)\) for any open set \(\Omega\subseteq\mathbb{R}^n\), see \Cref{thm:locdecomp}.

The classical result of Fefferman \& Phong appearing in \cite{Fefferman-Phong}, which we record here for completeness, follows as a special case of \Cref{thm:c3main}.

\begin{cor}[Fefferman \& Phong]\label{cor:FP}
If \(f\in C^{3,1}(\mathbb{R}^n)\) is non-negative, then it can be written as a finite sum of squares of \(C^{1,1}(\mathbb{R}^n)\) functions.
\end{cor}

There is a self-contained proof of this result in \cite{Guan} which Guan attributes to Fefferman. The key idea underlying the argument communicated by Guan is that the domain of \(f\) can be partitioned into regions on which \(f\) is well-behaved, and others where \(f\) does not have a half-regular root. The same approach is used to prove the generalizations of \Cref{cor:FP} in \cite[Proposition 1.1]{Tataru} and \cite[Theorem 4.5]{SOS_I}.  In each partition region, separate techniques are employed to recover an appropriate local decomposition of \(f\). Then, using an inductive argument, these local estimates are combined to obtain a global sum of squares decomposition.

Following suit, we begin by using the properties of non-negative \(C^{k,\alpha}(\mathbb{R}^n)\) functions established in \Cref{chap:holderchap} to partition the domain of \(f\) and argue locally. Our main instrument to this end is a partition of unity that we construct in \Cref{sec:pou}, which takes the form
\begin{equation}\label{eq:pou}
    \sum_{j=1}^\infty\psi_j^2=\begin{cases}
    1 & \textrm{if }f\textrm{ or one of its derivatives is nonzero},\\
    0 & \textrm{where }f\textrm{ and all of its derivatives  vanish}.
    \end{cases}
\end{equation}
The functions \(\psi_j\) in \eqref{eq:pou} are smooth, and they have special properties relating to \(f\) which we summarize in \Cref{thm:party}. Combining this partition with the localized results that we establish in the following section, we go on to prove \Cref{thm:c3main} in \Cref{sec:c3proof}.

\section{Local Regularity of Roots}\label{sec:reg}

Now we show that the localization of a non-negative function \(f\in C^{k,\alpha}(\mathbb{R}^n)\) to a sufficiently small ball either has a square root that belongs to the half-regular H\"older space \eqref{eq:halfreg}, or it can be decomposed into a sum of two functions; one which has a half-regular root, and another which depends on fewer variables than \(f\). These local decompositions are performed in the supports of the partition functions \(\psi_j\) written in \eqref{eq:pou}, and they always exist when \(k=2\) and \(k=3\). If \(k\geq4\) then \(f\) can be decomposed locally in the same way, but doing so requires additional hypotheses which are discussed in \Cref{chap:higher} and in \cite{SOS_I}.

The results of this section are established in full generality, and we do not restrict to the cases \(k\leq 3\), since we also use these results to study \(C^{k,\alpha}(\mathbb{R}^n)\) for \(k\geq 4\) later in the work. This generality sometimes comes at the expense of cumbersome notation, but our efforts are rewarded by a much clearer picture of the behaviour of \(C^{k,\alpha}(\mathbb{R}^n)\) functions. Unless stated  otherwise, throughout this section it is understood that \(k\) is any non-negative integer and \(0<\alpha\leq 1\).

Our primary instrument in this chapter is a function \(r\) that controls the derivatives of a function \(f\) pointwise. It was originally identified by Fefferman \& Phong in \cite{Fefferman-Phong} in the case \(k=3\) and \(\alpha=1\), and modified by Sawyer \& Korobenko in \cite{SOS_I} to work for \(k=4\) and \(0<\alpha\leq1\). For our purposes, we require a generalized form of this control function which is suitable for any pairing of \(k\) and \(\alpha\). This object lets us identify balls on which \(f\) is well-behaved in a manner that permits decomposition. Fixing a non-negative function \(f\in C^{k,\alpha}(\mathbb{R}^n)\), we define
\begin{equation}\label{eq:controlfunc}
    r(x)=\max_{\substack{0\leq j\leq k,\\j\;\mathrm{even}}}\bigg\{\sup_{|\xi|=1}[\partial^j_\xi f(x)]_+^\frac{1}{k-j+\alpha}\bigg\}.
\end{equation}
From this definition and \Cref{thm:cauchylike}, we see that \(r\) controls \(f\) pointwise in the following way.
\begin{cor}\label{cor:controlbound} 
For \(r\)  as in \eqref{eq:controlfunc} and for \(\ell\leq k\), the following pointwise derivative bounds hold:
\[
    |\nabla^\ell f(x)|\leq Cr(x)^{k-\ell+\alpha}.
\]
\end{cor}
Following \cite{Tataru} and \cite{SOS_I}, we show that \(r\) is essentially constant on sufficiently small balls. This property, which we call slow variation, is critical to our ensuing arguments.

\begin{lem}\label{lem:slowvar}
There exists a constant \(\nu>0\) such that \(|r(x)-r(y)|\leq \frac{1}{4}r(x)\) when \(|x-y|\leq \nu r(x)\).
\end{lem}

\textit{Remark}: Throughout this chapter we will continue to put size restrictions on the parameter \(\nu\). Our conditions ensure that in the end, \(\nu\) is a very small positive constant. The actual value of \(\nu\) turns out to be inconsequential for our final construction, thanks to \Cref{thm:party}, so we encounter no issue in occasionally asking that \(\nu\) be smaller than previously assumed.

\begin{proof}
Fix a non-negative function \(f\in C^{k,\alpha}(\mathbb{R}^n)\) and let \(r\) be as in \eqref{eq:controlfunc}. Given any two points \(x,y\in\mathbb{R}^n\) for which \(|x-y|\leq \nu r(x)\), we employ Lemmas \ref{lem:maxproplem} and \ref{lem:supprop} to first make the estimate
\begin{equation}\label{eq:intermed4}
    |r(x)-r(y)|\leq \max_{\substack{0\leq j\leq k,\\j\;\mathrm{even}}}\big\{\sup_{|\xi|=1}\big|[\partial^j_\xi f(x)]_+^\frac{1}{k-j+\alpha}-[\partial^j_\xi f(y)]_+^\frac{1}{k-j+\alpha}\big|\big\}.
\end{equation}
To proceed, first assume that \(j<k\), so that \(k-j+\alpha> 1\). Then \Cref{lem:smallpow} applies, and we get
\[
    \big|[\partial^j_\xi f(x)]_+^\frac{1}{k-j+\alpha}-[\partial^j_\xi f(y)]_+^\frac{1}{k-j+\alpha}\big|\leq |\partial^j_\xi f(x)-\partial^j_\xi f(y)|^\frac{1}{k-j+\alpha}.
\]
To estimate the right-hand side above, we expand the directional operator \(\partial_\xi^j\) from \eqref{eq:directionalop} to write
\[
    |\partial^j_\xi f(x)-\partial^j_\xi f(y)|=\bigg|\sum_{|\beta|=j}\frac{\beta!}{j!}\xi^\beta(\partial^\beta f(x)-\partial^\beta f(y))\bigg|\leq \sum_{|\beta|=j}\frac{\beta!}{j!}|\partial^\beta f(x)-\partial^\beta f(y)|.
\]
Then, \Cref{lem:Taylorest} together with the bounds \(|\nabla^\ell f(x)|\leq Cr(x)^{k-\ell+\alpha}\) and \(|x-y|\leq \nu r(x)\) gives us
\[
    |\partial^\beta f(x)-\partial^\beta f(y)|\leq C\sum_{0\leq |\gamma|\leq k-|\beta|}\nu^{|\gamma|}r(x)^{k-|\beta|+\alpha}+C\nu^{k-|\beta|+\alpha}r(x)^{k-|\beta|+\alpha}.
\]
Every power of \(\nu\) above is larger than one, since \(|\beta|=j<k\). Thus, assuming that \(\nu\leq 1\), we find there exists a constant \(C\) which depends only on \(k\), \(\alpha\) and \(|\beta|\) for which
\begin{equation}\label{eq:omegaeqn}
    |\partial^\beta f(x)-\partial^\beta f(y)|\leq C\nu r(x)^{k-|\beta|+\alpha}.
\end{equation}
Therefore \(|\partial^j_\xi f(x)-\partial^j_\xi f(y)|\leq C\nu r(x)^{k-|\beta|+\alpha}\), and since this bound is independent of the direction \(\xi\) we can choose a small number \(\nu\) which is independent of \(x\) such that the following holds,
\[
    \sup_{|\xi|=1}\big|[\partial^j_\xi f(x)]_+^\frac{1}{k-j+\alpha}-[\partial^j_\xi f(y)]_+^\frac{1}{k-j+\alpha}\big|\leq  C\nu^\frac{1}{k-j+\alpha}r(x)\leq \frac{1}{4}r(x).
\]

It remains to consider the case \(j=k\), which arises when \(k\) is even. This time, we can employ \Cref{lem:almostLipschitz}, taking \(\beta=\frac{1}{\alpha}\) and \(\varepsilon=\nu^\alpha\). Combining with the inequality \(|x-y|\leq \nu r(x)\), we get
\[
    \sup_{|\xi|=1}\big|[\partial^k_\xi f(x)]_+^\frac{1}{\alpha}-[\partial^k_\xi f(y)]_+^\frac{1}{\alpha}\big|\leq\frac{C\nu (1+\nu^\alpha)r(x)}{((1+\nu^\alpha)^{\frac{\alpha}{1-\alpha}}-1)^\frac{1-\alpha}{\alpha}}+\nu^\alpha\max\{r(x),r(y)\}.
\]
If \(\nu\leq 1\) the quotient above is bounded by \(C\nu^\alpha r(x)\), for a constant \(C\) that depends only on \(\alpha\) and the H\"older semi-norms of \(f\). Since \(\max\{r(x),r(y)\}\leq r(x)+|r(x)-r(y)|\) by non-negativity of \(r\), when \(\nu\) is small enough we  have
\[
    \sup_{|\xi|=1}\big|[\partial^k_\xi f(x)]_+^\frac{1}{\alpha}-[\partial^k_\xi f(y)]_+^\frac{1}{\alpha}\big|\leq C\nu^\alpha r(x)+\nu^\alpha|r(x)-r(y)|\leq \frac{1}{8}r(x)+\frac{1}{2}|r(x)-r(y)|.
\]
It follows from these bounds and \eqref{eq:intermed4} that \(|r(x)-r(y)|\leq \frac{1}{4}r(x)\), as we wished to show.
\end{proof}

\textit{Remark}: The choice of the constant \(\frac{1}{4}\) in the preceding lemma is largely arbitrary; any constant between \(0\) and \(1\) will suffice, and we only choose \(\frac{1}{4}\) to simplify some later calculations and avoid introducing a new parameter unnecessarily.

Our main analysis takes place in the balls \(B(x,\nu r(x))\) for \(\nu\) as above, and the preceding lemma shows that \(r\) is essentially constant in these sets. Explicitly, for every \(y\in B(x,\nu r(x))\) the preceding lemma shows that \(\frac{3}{4}r(x)\leq r(y)\leq \frac{5}{4}r(x)\). The following result goes further, ensuring that \(f\) and its second derivatives also vary slowly on suitably small balls.

\begin{cor}\label{cor:supplementary}
Let \(f\in C^{k,\alpha}(\mathbb{R}^n)\) be non-negative for \(k\geq 2\), and let \(r\) be as in \eqref{eq:controlfunc}. If \(\nu\) is a sufficiently small positive constant, then there exists another constant \(\omega\) for which
\begin{itemize}
    \item[(1)] \(|f(x)-f(y)|\leq \frac{1}{2}\omega \nu r(x)^{k+\alpha}\) whenever \(|x-y|\leq \nu r(x)\), and
    \item[(2)] \(|\partial^\beta f(x)-\partial^\beta f(y)|\leq \frac{1}{2}r(x)^{k-2+\alpha}\) whenever \(|x-y|\leq \sqrt{\nu^2+3\nu\omega}r(x)\) and \(|\beta|=2\).
\end{itemize}
\end{cor}

\begin{proof}
Estimate \textit{(1)} follows by taking \(\omega=2C\) for \(C\) as in \eqref{eq:omegaeqn} when \(|\beta|=0\), while \textit{(2)} is given by replacing \(\nu\) with \(\sqrt{\nu^2+3\nu\omega}\) in \eqref{eq:omegaeqn} when \(|\beta|=2\), and once again choosing \(\nu\) small enough.
\end{proof}

For the remainder of this section, we fix \(x\in\mathbb{R}^n\) and the corresponding ball \(B=B(x,\nu r(x))\). We study the behaviour of \(f\) on \(B\) in two exhaustive cases: when \(f(x)\geq \omega\nu r(x)^{k+\alpha}\), and when \(f(x)<\omega\nu r(x)^{k+\alpha}\). In the first case, \(f\) has a half-regular square root on \(B\), as we now show.

\begin{lem}\label{lem:local1}
Let \(f\), \(r\), and \(\nu\) be as above and assume that \(f(x)\geq \omega\nu r(x)^{k+\alpha}\). Then for any multi-index \(\beta\) of order \(|\beta|\leq \frac{k+\alpha}{2}\), and for any \(y\in B\), the following pointwise estimates hold:
\begin{itemize}
    \item[(1)] \(|\partial^\beta\sqrt{f(y)}|\leq Cr(y)^{\frac{k+\alpha}{2}-|\beta|}\),
    \item[(2)] \([\partial^\beta\sqrt{f}]_{\frac{\alpha}{2}}(y)\leq Cr(y)^{\frac{k}{2}-|\beta|}\) if \(k\) is even,
    \item[(3)] \([\partial^\beta\sqrt{f}]_{\frac{1+\alpha}{2}}(y)\leq Cr(y)^{\frac{k-1}{2}-|\beta|}\) if \(k\) is odd.
\end{itemize}
Recall that \([f]_\alpha(y)\) denotes the pointwise variant of the \(\alpha\)-H\"older semi-norm, defined in \eqref{eq:pwop}. 
\end{lem}

\begin{proof}
Using \Cref{cor:rootscor} we bound derivatives of \(f\) at \(y\in B\), estimating pointwise to first get
\[
    |\partial^\beta\sqrt{f(y)}|=\bigg|\sum_{\Gamma\in P(\beta)}C_{\beta,\Gamma}f(y)^{\frac{1}{2}-|\Gamma|}\prod_{\gamma\in\Gamma}\partial^\gamma f(y)\bigg|\leq C\sum_{\Gamma\in P(\beta)}f(y)^{\frac{1}{2}-|\Gamma|}\prod_{\gamma\in\Gamma}|\partial^\gamma f(y)|.
\]
Observe that \(f(y)\geq \frac{1}{2}\omega\nu r(x)^{k+\alpha}\) for any \(y\in B\) by estimate \textit{(1)} of \Cref{cor:supplementary}. Additionally, \Cref{cor:controlbound} and slow variation of \(r\) on \(B\) together show that
\[
    |\partial^\beta\sqrt{f(y)}|\leq C\sum_{\Gamma\in P(\beta)}r(y)^{(k+\alpha)(\frac{1}{2}-|\Gamma|)}\prod_{\gamma\in\Gamma}r(y)^{k+\alpha-|\gamma|}= Cr(y)^{\frac{k+\alpha}{2}-|\beta|}.
\]
Since this holds pointwise on \(B\) we conclude that item \textit{(1)} holds.

For the semi-norm estimates \textit{(2)} and \textit{(3)}, we treat the odd and even cases simultaneously by letting \(\Lambda\) denote the integer part of \(\frac{k}{2}\) and setting \(\lambda=\frac{k+\alpha}{2}-\Lambda\). To prove the claimed results it suffices to show that \([\partial^\beta\sqrt{f}]_{\lambda}(y)\leq Cr(y)^{\Lambda-|\beta|}\) whenever \(|\beta|\leq \Lambda\) and \(y\in B\). Let \(\beta\) be given and observe that by \Cref{cor:rootscor} together with the triangle inequality, we have for \(w,z\in B\) that
\begin{equation}\label{eq:ctrl}
\begin{split}
    |\partial^\beta\sqrt{f(w)}-\partial^\beta\sqrt{f(z)}|&\leq C\sum_{\Gamma\in P(\beta)} f(w)^{\frac{1}{2}-|\Gamma|}\bigg|\prod_{\gamma\in\Gamma}\partial^\gamma f(w)-\prod_{\gamma\in\Gamma}\partial^\gamma f(z)\bigg|\\
    &\qquad+C\sum_{\Gamma\in P(\beta)} |f(w)^{\frac{1}{2}-|\Gamma|}-f(z)^{\frac{1}{2}-|\Gamma|}|\prod_{\gamma\in\Gamma}|\partial^\gamma f(z)|.    
\end{split}
\end{equation}
We treat the terms on the right-hand side above separately, first using the Mean Value Theorem and slow variation of \(r\) on \(B\) to obtain the bound
\[
    |f(w)^{\frac{1}{2}-|\Gamma|}-f(z)^{\frac{1}{2}-|\Gamma|}|\leq Cr(y)^{\frac{k+\alpha}{2}-1-|\Gamma|(k+\alpha)}|w-z|\leq Cr(y)^{\frac{k+\alpha}{2}-|\Gamma|(k+\alpha)-\lambda}|w-z|^\lambda.
\]
Explicitly, we have used here that \(f\) is bounded below by a multiple of \(r(y)^{k+\alpha}\) on \(B\), followed by the fact that \(|w-z|\leq Cr(y)^{1-\lambda}|w-z|^\lambda\) for \(w,z\in B\). Bounding the other derivatives of \(f\) on \(B\) in the same fashion, we see that
\[
    |f(w)^{\frac{1}{2}-|\Gamma|}-f(z)^{\frac{1}{2}-|\Gamma|}|\prod_{\gamma\in\Gamma}|\partial^\gamma f(z)|\leq Cr(y)^{\frac{k+\alpha}{2}-|\Gamma|(k+\alpha)-\lambda}|w-z|^\lambda\prod_{\gamma\in\Gamma}r(y)^{k+\alpha-|\gamma|}.
\]
Since \(\Gamma\in P(\beta)\), we can simplify the right-hand side above to \(Cr(y)^{\Lambda-|\beta|}|w-z|^\lambda\). To estimate the remaining term of \eqref{eq:ctrl}, we use \Cref{lem:proddiff} to produce the bound
\begin{equation}\label{eq:sub}
    \bigg|\prod_{\gamma\in\Gamma}\partial^\gamma f(w)-\prod_{\gamma\in\Gamma}\partial^\gamma f(z)\bigg|\leq \sum_{\gamma\in\Gamma}\bigg(\prod_{\mu\in\Gamma\setminus\{\gamma\}}\sup_{B}|\partial^\mu f|\bigg)|\partial^\gamma f(w)-\partial^\gamma f(z)|.    
\end{equation}

Arguing as above, we can control the product in \eqref{eq:sub} by \(Cr(y)^{(k+\alpha)(|\Gamma|-1)-|\beta|+|\gamma|}\). Further, since \(|\beta|\leq\Lambda< \frac{k+\alpha}{2}\), whenever \(\gamma\in P(\beta)\) we must have \(|\gamma|<k\), meaning that we can use the Mean Value Theorem and \Cref{cor:controlbound} to get for \(w,z\in B\) that
\[
    |\partial^\gamma f(w)-\partial^\gamma f(z)|\leq Cr(y)^{k+\alpha-|\gamma|-\lambda}|w-z|^\lambda.
\]
Therefore the left-hand side of \eqref{eq:sub} is bounded by \(Cr(y)^{(k+\alpha)|\Gamma|-|\beta|-\lambda}|w-z|^\lambda\). Consequently,
\[
    f(w)^{\frac{1}{2}-|\Gamma|}\bigg|\prod_{\gamma\in\Gamma}\partial^\gamma f(w)-\prod_{\gamma\in\Gamma}\partial^\gamma f(z)\bigg|\leq Cr(y)^{\frac{k+\alpha}{2}-|\beta|-\lambda}|y-z|^\lambda=Cr(y)^{\Lambda-|\beta|}|y-z|^\lambda.
\]
Now from \eqref{eq:ctrl} and definition \eqref{eq:pwop} of the pointwise semi-norm, we observe for any \(y\in B\) that
\[
    [\partial^\beta\sqrt{f}]_\lambda(y)=\limsup_{w,z\rightarrow y}\frac{|\partial^\beta\sqrt{f(w)}-\partial^\beta\sqrt{f(z)}|}{|w-z|^\lambda}\leq Cr(y)^{\Lambda-|\beta|},
\]
where critically, the constant \(C\) is independent of \(B\). Hence items \textit{(2)} and \textit{(3)} hold as claimed.
\end{proof}

In the proof above we constructed pointwise \(\alpha\)-H\"older semi-norm estimates, rather than global semi-norm estimates, since the former are needed to prove \Cref{thm:c3main}. These pointwise estimates actually imply local variants, since \textit{(2)} and \textit{(3)} together with slow variation of \(r\) on the ball \(B=B(x,\nu r(x))\) give \([\partial^\beta\sqrt{f}]_{\frac{\alpha}{2}}(y)\leq Cr(x)^{\frac{k}{2}-|\beta|}\) and \([\partial^\beta\sqrt{f}]_{\frac{1+\alpha}{2}}(y)\leq Cr(x)^{\frac{k-1}{2}-|\beta|}\), respectively when \(k\) is even and odd. Taking a supremum over \(y\in B\) yields the local bounds
\[
    [\partial^\beta\sqrt{f}]_{\frac{\alpha}{2},B}\leq Cr(x)^{\frac{k}{2}-|\beta|}\qquad\mathrm{and}\qquad[\partial^\beta\sqrt{f}]_{\frac{1+\alpha}{2},B}\leq Cr(x)^{\frac{k-1}{2}-|\beta|},
\]
again when \(k\) is even and odd respectively. Since \(x\) and \(B\) are fixed, we can conclude the following.

\begin{cor}
Let \(f\in C^{k,\alpha}(\mathbb{R}^n)\) be non-negative and let \(r\) be as in \eqref{eq:controlfunc}. If \(\nu\) is a sufficiently small positive constant and \(f(x)\geq \omega\nu r(x)^{k+\alpha}\), then \(\sqrt{f}\in C^\frac{k+\alpha}{2}(B(x,\nu r(x)))\).
\end{cor}

On the other hand, if \(f(x)<\omega\nu r(x)^{k+\alpha}\) then we can no longer conclude that the localization of \(f\) to \(B\) has a half-regular square root. Rather, the best we can do is form a local decomposition of \(f\) on \(B\) that takes the form
\begin{equation}\label{eq:loooooocal}
    f=g^2+F,
\end{equation}
where \(g\) is half as regular as \(f\) and the remainder \(F\) can be handled using other techniques. Specifically, \(F\) can be made to depend on only \(n-1\) variables and it inherits the regularity of \(f\).

Constructing the local decomposition described above is a laborious task, which comprises the remainder of this section. Lemmas \ref{lem:minia}, \ref{lem:minia2}, \ref{lem:nice}, and \ref{lem:implicitest} are devoted to constructing \(F\) and verifying that it has the desired properties -- our sequence of results is patterned after the decomposition arguments in \cite{Tataru} and \cite{SOS_I}, but we go to greater lengths to obtain estimates for every \(k\geq 0\). As such, the pages which follow are the most technically difficult of this thesis.

In what follows, we partition a variable \(x\in\mathbb{R}^n\) by writing \(x=(x',x_n)\) for \(x'\in\mathbb{R}^{n-1}\) and \(x_n\in\mathbb{R}\). Given a ball \(B\) we also write \(B'=\{x':x\in B\}\) to denote the projection of \(B\) onto \(\mathbb{R}^{n-1}\). Equipped with notation we have the following result.

\begin{lem}\label{lem:minia}
Assume that a non-negative function \(f\in C^{k,\alpha}(\mathbb{R}^n)\) satisfies \(f(x)<\omega \nu r(x)^{k+\alpha}\), and let \(r\) satisfy the following pointwise bound on \(B=B(x,\nu r(x))\),
\vspace{-0.25em}
\begin{equation}\label{eq:needed}
    r(y)\leq \max\bigg\{f(y)^\frac{1}{k+\alpha},\sup_{|\xi|=1}[\partial^2_\xi f(y)]_+^\frac{1}{k-2+\alpha}\bigg\}.
\end{equation}
Then after a suitable coordinate rotation centred at \(x\), there exists a function \(X\in C^{k-1}(B')\) which enjoys the following properties for \(y\in B\),
\begin{itemize}
    \item[(1)] \(|x_n-X(y')|\leq r(x)\),
    \item[(2)] \(f(y',X(y'))\leq f(y)\),
    \item[(3)] \(\partial_{x_n}f(y',X(y'))=0\).
\end{itemize}
\end{lem}

\medskip

\textit{Remark}: Put simply, this result states that \(f\) has a unique minimum along each vertical ray in a cylinder of height \(2r(x)\) that contains \(B\). Moreover, the collection of these minima, viewed as a function on the base \(B'\), actually defines a function belonging to \(C^{k-1}(B')\). The required pointwise bound \eqref{eq:needed} holds automatically when \(k=2\) and \(k=3\), while if \(k\geq 4\) then additional conditions must be placed on \(f\) for \eqref{eq:needed} to hold; see \Cref{chap:higher}.

\begin{proof}
Let \(y'\) satisfy \(|y'-x'|\leq \nu r(x)\) so that \(y'\in B'\) and set \(h=\sqrt{3\omega\nu}\), where \(\nu\) is chosen small enough that \(\omega\nu\leq \frac{1}{3}\). Our aim is to show that the mapping \(t\mapsto f(y',t)\) has a unique minimum in the open interval \((x_n-h,x_n+h)\), so that items \textit{(1)-(3)} above follow as direct consequences. Observe that if \(f(x)<\omega\nu r(x)^{k+\alpha}< r(x)^{k+\alpha}\) then \(r(x)\) is strictly positive, and by \eqref{eq:needed} we have
\[
    r(x)=\sup_{|\xi|=1}[\partial^2_\xi f(x)]_+^\frac{1}{k-2+\alpha}.
\]
After a rotation of coordinates centred at \(x\), we can assume without loss of generality that the supremum above is attained by the \(x_n\) directional derivative of \(f\), giving \(r(x)^{k-2+\alpha}=\partial^2_{x_n}f(x)\). We note that non-negativity of \(f\) as well as its inclusion in \(C^{k,\alpha}(\mathbb{R}^n)\) are preserved by such coordinate changes.

For each \(t\in [x_n-h,x_n+h]\) we have \(\partial^2_{x_n}f(y',t)\geq\frac{1}{2}r(x)^{k-2+\alpha}\) by item \textit{(2)} of \Cref{cor:supplementary}, meaning that \(t\mapsto f(x',t)\) has a unique minimum on the compact interval \([x_n-h ,x_n+h]\).  Assume toward a contradiction that the minimum on this interval occurs at the left endpoint \(x_n-h\), so that \(\partial_{x_n}f(x',x_n-h)\geq 0\). Additionally, note that \(f(y',x_n)<\frac{3}{2}\omega\nu r(x)^{k+\alpha}\) by item \textit{(1)} of \Cref{cor:supplementary}. So for some \(\xi\in (x_n-h,x_n)\) we have from our choice of \(h\) that
\[
     \frac{3}{2}\omega\nu r(x)^{k+\alpha}>f(y',x_n)=f(y',x_n-h)+h\partial_{x_n}f(x',x_n-h)+\frac{1}{2}h^2\partial^2_{x_n}f(y',\xi)\geq \frac{3}{2}\omega\nu r(x)^{k+\alpha}.
\]
This is a contradiction since \(r(x)>0\), and an identical contradiction arises if we assume that the minimum occurs at \(x_n+h\). It follows that \(t\mapsto f(y',t)\) has a unique minimum in \((x_n-h,x_n+h)\) which we call \(X(y')\). Properties \textit{(1)-(3)} follow from our construction, and as \(y'\) was any point satisfying \(|y'-x'|\leq \nu r(x)\), we see that \(X\) is well-defined as a function on all of \( B'\).

To see that \(X\in C^{k-1}(B')\), we observe that for each \(y'\) we have \(\partial_{x_n}f(y',X(y'))=0\) and that \(\partial^2_{x_n}f(y',X(y'))>0\) by item \textit{(2)} of \Cref{cor:supplementary}. Applying the Implicit Function Theorem (\Cref{thm:IFT}) to \(\partial_{x_n}f\) at the point \((y',X(y'))\), we see that \(X\) is \(C^{k-1}\) in some neighbourhood of \(y'\) since \(\partial_{x_n}f\in C^{k-1}(\mathbb{R}^n)\). Further, as \(y'\) was any point in \(B'\) we can find such a neighbourhood around every point in the domain of \(X\). By the uniqueness statement of \Cref{thm:IFT}, these \(C^{k-1}\) functions must agree wherever their neighbourhoods overlap, meaning that \(X\in C^{k-1}(B')\).
\end{proof}

\begin{lem}\label{lem:minia2}
The function \(X\) defined above belongs to \(C^{k-1,\alpha}(B')\). In particular, it satisfies the following pointwise estimates for every \(y\in B\) and for all derivatives of order \(|\beta|\leq k-1\),
\begin{itemize}
    \item[(1)] \(|\partial^\beta X(y')|\leq Cr(y)^{1-|\beta|}\),
    \item[(2)] \([\partial^\beta X]_{\alpha}(y')\leq Cr(y)^{1-\alpha-|\beta|}\).
\end{itemize}
\end{lem}

\textit{Remark}: Our proof of this result is challenging, as it employs the multivariate calculus identities established in the previous chapter to their full extent. The reader may wish to pass over this argument on their first read, and return later when the significance of \Cref{lem:minia2} in the proof of \Cref{thm:c3main} is understood.

\begin{proof}
To establish derivative estimates of \(X\) for item \textit{(1)}, we first use \Cref{thm:IFT} to write
\[
    |\partial^\beta X(y')|=\bigg|\frac{1}{\partial^2_{x_n}f(y',X(y'))}\sum_{0\leq\eta\leq\beta}\sum_{ \substack{\Gamma\in P(\eta),\\\Gamma\neq \{\beta\}}}C_{\beta,\Gamma}(\partial^{\beta-\eta}\partial^{|\Gamma|+1}_{x_n}f(y',X(y')))\prod_{\gamma\in\Gamma}\partial^\gamma X(y')\bigg|.
\]
We simplify this by observing that \(\partial^2_{x_n}f(y',X(y'))\geq Cr(y)^{k-2+\alpha}\) on \(B'\) by item \textit{(2)} of \Cref{cor:supplementary} and slow variation of \(r\). Using this fact with the estimates of \Cref{cor:controlbound}, we get
\[
    |\partial^\beta X(y')|\leq Cr(y)^{2-k-\alpha}\sum_{0\leq\eta\leq\beta}\sum_{ \substack{\Gamma\in P(\eta),\\\Gamma\neq \{\beta\}}}r(y',X(y'))^{k+\alpha+|\eta|-|\beta|-|\Gamma|-1}\prod_{\gamma\in\Gamma}|\partial^\gamma X(y')|.
\]
Taking \(\nu\) small enough that \Cref{lem:slowvar} holds when \(\nu\) is replaced by \(\sqrt{\nu^2+3\omega\nu}\), it follows from the definition of \(X\) that \(r(y',X(y'))\leq Cr(y)\) whenever \(y'\in B'\).

Now we argue by strong induction that \(|\partial^\beta X(y')|\leq Cr(y)^{1-|\beta|}\), observing for a base case that if \(|\beta|=1\) then for some \(j=1,\dots,n-1\), and \(y'\) in the domain of \(X\) we can write 
\[
    |\partial^\beta X(y')|=|\partial_{x_j}X(y')|=\bigg|\frac{\partial_{x_j}\partial_{x_n}f(y',X(y'))}{\partial^2_{x_n}f(y',X(y'))}\bigg|\leq \frac{Cr(y)^{k+\alpha-2}}{r(y)^{k+\alpha-2}}=C.
\]
For an inductive hypothesis we assume that \(|\partial^\gamma X(y')|\leq Cr(y)^{1-|\gamma|}\) whenever \(|\gamma|<|\beta|\), so that
\[
    |\partial^\beta X(y')|\leq C\sum_{0\leq\eta\leq\beta}\sum_{ \substack{\Gamma\in P(\eta),\\\Gamma\neq \{\beta\}}}r(y)^{1-|\beta|+|\eta|-|\Gamma|}\prod_{\gamma\in\Gamma}r(y)^{1-|\gamma|}=Cr(y)^{1-|\beta|}.
\]
It follows that the claimed estimate holds for every \(\beta\) with \(|\beta|\leq k-1\), giving item \textit{(1)}.

It remains to demonstrate that estimate \textit{(2)} holds, and once again we prove this using strong induction. For a base case we observe that \(|X(w')-X(z')|\leq C|w'-z'|\) by the Mean Value Theorem and the derivative bound on \(X\) shown above. Further, if \(w',z'\in B'\) then we have that \(|w'-z'|\leq Cr(y)^{1-\alpha}|w'-z'|^\alpha\). Therefore 
\[
    |X(w')-X(z')|\leq Cr(y)^{1-\alpha}|w'-z'|^\alpha,
\]
from which it follows by taking a limit supremum as \(w',z'\rightarrow y'\) that \([X]_{\alpha}(y')\leq Cr(y)^{1-\alpha}\) for \(y\in B\). Thus the required semi-norm estimate holds in the base case.

For the inductive step we begin by using \Cref{thm:IFT} to expand the pointwise difference \(\partial^\beta X(w')-\partial^\beta X(z')\) as follows,
\[
    \sum_{\eta\leq\beta}\sum_{ \substack{\Gamma\in P(\eta),\\\Gamma\neq \{\beta\}}}C_{\beta,\Gamma}\bigg(\frac{\partial^{\beta-\eta}\partial^{|\Gamma|+1}_{x_n}f(w',X(w'))}{\partial^2_{x_n}f(w',X(w'))}\prod_{\gamma\in\Gamma}\partial^\gamma X(w')-\frac{\partial^{\beta-\eta}\partial^{|\Gamma|+1}_{x_n}f(z',X(z'))}{\partial^2_{x_n}f(z',X(z'))}\prod_{\gamma\in\Gamma}\partial^\gamma X(z')\bigg).
\]
For brevity, we will henceforth write \(\Tilde{w}=(w',X(w'))\) and note that \(|\tilde{w}-\Tilde{z}|\leq C|w'-z'|\). Using the triangle inequality we can bound each of the terms in the sum above by the quantity
\[
    \frac{|\partial^{\beta-\eta}\partial^{|\Gamma|+1}_{x_n}f(\tilde{w})|}{\partial^2_{x_n}f(\tilde{w})}\bigg|\prod_{\gamma\in\Gamma}\partial^\gamma X(w)-\prod_{\gamma\in\Gamma}\partial^\gamma X(z)\bigg|+\bigg|\frac{\partial^{\beta-\eta}\partial^{|\Gamma|+1}_{x_n}f(\tilde{w})}{\partial^2_{x_n}f(\tilde{w})}-\frac{\partial^{\beta-\eta}\partial^{|\Gamma|+1}_{x_n}f(\Tilde{z})}{\partial^2_{x_n}f(\Tilde{z})}\bigg|\prod_{\gamma\in\Gamma}|\partial^\gamma X(w)|
\]
meaning that the required semi-norm estimate on \(\partial^\beta X\) will follow if we can bound each of the four factors above in an appropriate fashion.

To this end we first observe that by lower control of \(\partial^2_{x_n}f\) on \(B\), together with slow variation of \(r\) and the derivative estimates of \Cref{cor:controlbound}, we have
\[
    \frac{|\partial^{\beta-\eta}\partial^{|\Gamma|+1}_{x_n}f(\tilde{w})|}{\partial^2_{x_n}f(\tilde{w})}\leq \frac{Cr(y)^{k+\alpha-|\beta|+|\eta|-|\Gamma|-1}}{r(y)^{k+\alpha-2}}=Cr(y)^{1-|\beta|+|\eta|-|\Gamma|}.
\]
Next, we employ the derivative bounds on \(X\) proved above for item \textit{(1)}. Note that since \(\Gamma\in P(\eta)\) for \(\eta\) satisfying \(|\eta|\leq |\beta|\leq k\), and since \(\Gamma\neq\{\beta\}\), for each \(\gamma\in \Gamma \) we must have that \(|\gamma|\leq k-1\). Therefore \(|\partial^\gamma X(w')|\leq Cr(y)^{1-|\gamma|}\) uniformly on \(B'\) by \textit{(1)}, and it follows that
\[
    \bigg|\prod_{\gamma\in\Gamma}\partial^\gamma X(w')\bigg|\leq C\prod_{\gamma\in\Gamma}r(y)^{1-|\gamma|}=Cr(y)^{|\Gamma|-|\eta|}.
\]
Using lower control of \(\partial^2_{x_n}f\) by \(r(y)^{k-2+\alpha}\) once again, we can also make the the following estimate,
\[
    \bigg|\frac{\partial^{\beta-\eta}\partial^{|\Gamma|+1}_{x_n}f(\tilde{w})}{\partial^2_{x_n}f(\tilde{w})}-\frac{\partial^{\beta-\eta}\partial^{|\Gamma|+1}_{x_n}f(\Tilde{z})}{\partial^2_{x_n}f(\Tilde{z})}\bigg|\leq \frac{C|\partial^2_{x_n}f(\Tilde{z})\partial^{\beta-\eta}\partial^{|\Gamma|+1}_{x_n}f(\tilde{w})-\partial^2_{x_n}f(\tilde{w})\partial^{\beta-\eta}\partial^{|\Gamma|+1}_{x_n}f(\Tilde{z})|}{r(y)^{2k+2\alpha-4}}.
\]
Employing the triangle inequality and our local control of derivatives of \(f\), we can bound the numerator on the right-hand side above by a constant multiple of the following expression,
\[
    r(y)^{k-2+\alpha}|\partial^{\beta-\eta}\partial^{|\Gamma|+1}_{x_n}f(\tilde{w})-\partial^{\beta-\eta}\partial^{|\Gamma|+1}_{x_n}f(\Tilde{z})|+r(y)^{k+\alpha-|\beta|+|\eta|-|\Gamma|-1}|\partial^2_{x_n}f(\Tilde{z})-\partial^2_{x_n}f(\tilde{w})|.
\]
If \(|\beta|-|\eta|+|\Gamma|+1<k\) then the Mean Value Theorem, together with our pointwise derivative estimates on \(f\), and the fact that \(|\tilde{w}-\Tilde{z}|\leq Cr(y)\) for \(w',z'\in B'\), all give
\[
    |\partial^{\beta-\eta}\partial^{|\Gamma|+1}_{x_n}f(\tilde{w})-\partial^{\beta-\eta}\partial^{|\Gamma|+1}_{x_n}f(\Tilde{z})|\leq Cr(y)^{k-|\beta|+|\eta|-|\Gamma|-1}|w'-z'|^\alpha.
\]
On the other hand, if \(|\beta|-|\eta|+|\Gamma|+1=k\) this estimate holds since \(\partial^{\beta-\eta}\partial^{|\Gamma|+1}_{x_n}f\in C^\alpha(\mathbb{R}^n)\) and 
\[
    |\partial^{\beta-\eta}\partial^{|\Gamma|}_{x_n}f(\tilde{w})-\partial^{\beta-\eta}\partial^{|\Gamma|}_{x_n}f(\Tilde{z})|\leq C|\tilde{w}-\Tilde{z}|^\alpha\leq C|w'-z'|^\alpha=Cr(y)^{k-|\beta|+|\eta|-|\Gamma|-1}|w'-z'|^\alpha.
\]
An identical argument shows that \(|\partial^2_{x_n}f(z',X(z'))-\partial^2_{x_n}f(w',X(w'))|\leq Cr(y)^{k-2}|w'-z'|^\alpha\) when \(k\geq 2\), and altogether we find now that
\[
    \bigg|\frac{\partial^{\beta-\eta}\partial^{|\Gamma|+1}_{x_n}f(\tilde{w})}{\partial^2_{x_n}f(\tilde{w})}-\frac{\partial^{\beta-\eta}\partial^{|\Gamma|+1}_{x_n}f(\Tilde{z})}{\partial^2_{x_n}f(\Tilde{z})}\bigg|\leq Cr(x)^{1-\alpha-|\beta|+|\eta|-|\Gamma|}|w'-z'|^\alpha.
\]

One more term estimate remains before we can invoke an inductive hypothesis on the H\"older semi-norms of \(X\) to simplify our bounds. Iterating the triangle inequality using \Cref{lem:proddiff} gives
\[
    \bigg|\prod_{\gamma\in\Gamma}\partial^\gamma X(w')-\prod_{\gamma\in\Gamma}\partial^\gamma X(z')\bigg|\leq \sum_{\gamma\in\Gamma}\bigg(\prod_{\mu\in\Gamma\setminus\{\gamma\}}\sup_{B'}|\partial^\mu X|\bigg)|\partial^\gamma X(w')-\partial^\gamma X(z')|,
\]
and using the uniform estimates established above for item \textit{(1)}, we can bound this further by
\[
    \bigg|\prod_{\gamma\in\Gamma}\partial^\gamma X(w')-\prod_{\gamma\in\Gamma}\partial^\gamma X(z')\bigg|\leq C\sum_{\gamma\in\Gamma}r(y)^{|\Gamma|-1-|\eta|+|\gamma|}|\partial^\gamma X(w')-\partial^\gamma X(z')|. 
\]
Consequently, we can bound each term in our earlier expansion of \(\partial^\beta X(w')-\partial^\beta X(z')\) to get
\[
    |\partial^\beta X(w')-\partial^\beta X(z')|\leq C\sum_{\eta\leq\beta}\sum_{ \substack{\Gamma\in P(\eta),\\\Gamma\neq \{\beta\}}}\sum_{\gamma\in\Gamma}(r(y)^{-|\beta|+|\gamma|}|\partial^\gamma X(w')-\partial^\gamma X(z')|+r(y)^{1-\alpha-|\beta|}|w'-z'|^\alpha).
\]

Dividing the expression above by \(|w'-z'|^\alpha\) and taking a limit supremum as \(w',z'\rightarrow y'\), it follows from definition \eqref{eq:pwop} of the pointwise H\"older semi-norm that
\[
    [\partial^\beta X]_\alpha(y')\leq C\sum_{\eta\leq\beta}\sum_{ \substack{\Gamma\in P(\eta),\\\Gamma\neq \{\beta\}}}\sum_{\gamma\in\Gamma}r(y)^{-|\beta|+|\gamma|}[\partial^\gamma X]_\alpha(y')+Cr(y)^{1-\alpha-|\beta|}.
\]
Finally, assume for a strong inductive hypothesis that \([\partial^\gamma X]_{\alpha}(y')\leq Cr(y)^{1-\alpha-|\gamma|}\) for each \(y\in B\) whenever \(|\gamma|<|\beta|\). It follows immediately that
\[
    [\partial^\beta X]_\alpha(y')\leq C\sum_{\eta\leq\beta}\sum_{ \substack{\Gamma\in P(\eta),\\\Gamma\neq \{\beta\}}}\sum_{\gamma\in\Gamma}r(y)^{-|\beta|+|\gamma|}r(y)^{1-\alpha-|\gamma|}+Cr(y)^{1-\alpha-|\beta|}=Cr(y)^{1-\alpha-|\beta|}.
\]
By induction this holds for every \(\beta\) with \(|\beta|\leq k-1\), showing that \textit{(2)} holds on \(B'\).
\end{proof}

Henceforth, we assume that \(\nu\) has been chosen small enough that the \(\nu\)-dependant estimates of Lemmas \ref{lem:slowvar}, \ref{cor:supplementary}, \ref{lem:minia} and \ref{lem:minia2} all hold. Moreover, we assume that \eqref{eq:needed} also holds and we recall that this is automatic when \(k=2\) and \(k=3\). Our reward for undergoing the hard work of defining \(X\) on \(B\) and proving its various properties is the ability to unambiguously define
\[
    F(y)=f(y',X(y')).
\]
It follows from the construction above that \(f-F\geq 0\) on \(B\), with equality at local minima of \(f\).

Additionally, we will show that \(f-F\) has a root which is half as regular as \(f\) in the sense of \eqref{eq:rootspaces}, meaning that we can form the decomposition \(f=(\sqrt{f-F})^2+F\) on \(B\). This was the key insight underlying Fefferman \& Phong's result from \cite{Fefferman-Phong} in the setting of \(C^{3,1}(\mathbb{R}^n)\). To prove that this decomposition has the desired properties, we begin by establishing some estimates.

\begin{lem}\label{lem:nice}
Assume that the hypotheses of \Cref{lem:minia} hold, let \(X\) and be as above, and define \(F(y)=f(y',X(y'))\). Then for \(\beta\) with \(|\beta|<\frac{k+\alpha}{2}\) and every \(y\in B\) the following estimates hold,
\begin{itemize}
    \item[(1)] \(|\partial^\beta\sqrt{f(y)-F(y)}|\leq Cr(y)^{\frac{k+\alpha}{2}-|\beta|}\),
    \item[(2)] \([\partial^\beta\sqrt{f-F}]_{\frac{\alpha}{2}}(y)\leq Cr(y)^{\frac{k}{2}-|\beta|}\) if \(k\) is even,
    \item[(3)] \([\partial^\beta\sqrt{f-F}]_{\frac{1+\alpha}{2}}(y)\leq Cr(y)^{\frac{k-1}{2}-|\beta|}\) if \(k\) is odd.
\end{itemize}
\end{lem}

\begin{proof}
By the fundamental theorem of calculus we have \(f(y)-F(y)=(y_n-X(y'))^2H(y)\), where
\[
    H(y)=\int_0^1 (1-t)\partial^2_{x_n}f(y',ty_n+(1-t)X(y'))dt.
\]
Arguing as in the proof of \Cref{lem:minia} we have \(\frac{1}{2}r(x)^{k-2+\alpha}\leq \partial^2_{x_n}f(y',ty_n+(1-t)X(y'))\) for every \(t\in[0,1]\), meaning that \(H(y)\geq Cr(y)^{k-2+\alpha}\) on \(B\). Additionally, we can bound the derivatives of \(H\) on \(B\). To this end we first observe that
\[
    \partial^\beta H(y)=\int_0^1 (1-t)\partial^\beta[\partial^2_{x_n}f(y',ty_n+(1-t)X(y'))]dt.
\]
To evaluate the derivative inside the integral, we use the shorthand \(L(y,t)=ty_n+(1-t)X(y')\) and write \(\beta'=(\beta_1,\dots,\beta_{n-1})\), so that an application of \Cref{lem:genchain} gives
\[
    \partial^\beta[\partial^2_{x_n}f(y',L(y,t))]=t^{\beta_n}\sum_{\mu\leq\beta'}\sum_{\Gamma\in P(\mu)}C_{\beta',\Gamma}(1-t)^{|\Gamma|}\partial^{\beta'-\mu}\partial_{x_n}^{2+\beta_n+|\Gamma|}f(y',L(y,t))\prod_{\gamma\in\Gamma}\partial^\gamma X(y'),
\]
where the constants \(C_{\beta',\Gamma}\) are given by \eqref{eq:cm}. It follows from the triangle inequality that
\[
    |\partial^\beta H(y)|\leq  C\sum_{\mu\leq\beta'}\sum_{\Gamma\in P(\mu)}\int_0^1t^{\beta_n}(1-t)^{|\Gamma|+1}|\partial^{\beta'-\mu}\partial_{x_n}^{2+\beta_n+|\Gamma|}f(y',L(y,t))|dt\prod_{\gamma\in\Gamma}|\partial^\gamma X(y')|.
\]
Using estimate \textit{(1)} of \Cref{lem:minia2} we get \(|\partial^\gamma X(y')|\leq Cr(y)^{1-|\gamma|}\). Additionally, \Cref{cor:controlbound} and \Cref{lem:slowvar} give \(|\partial^{\beta'-\mu}\partial_{x_n}^{2+\beta_n+|\Gamma|}f(y',L(y,t))|\leq Cr(y)^{k+\alpha-2-|\Gamma|-|\beta|+|\mu|}\), so
\[
    |\partial^\beta H(y)|\leq  C\sum_{\mu\leq\beta'}\sum_{\Gamma\in P(\mu)}r(y)^{k+\alpha-2-|\Gamma|-|\beta|+|\mu|}\prod_{\gamma\in\Gamma}r(y)^{1-|\gamma|}=Cr(y)^{k+\alpha-2-|\beta|}.
\]

Next, using these bounds we can estimate derivatives of \(\sqrt{H}\) on \(B\). To this end we employ \Cref{cor:rootscor}, the lower bound \(H(y)\geq Cr(y)^{k-2+\alpha}\), and the derivative bound above to estimate
\[
    |\partial^\beta \sqrt{H(y)}| =\bigg|\sum_{\Gamma\in P(\beta)}C_{\beta,\Gamma}H(y)^{\frac{1}{2}-|\Gamma|}\prod_{\gamma\in\Gamma}\partial^\gamma H(y)\bigg|\leq Cr(y)^{\frac{k+\alpha}{2}-1-|\beta|}.
\]
Further, estimate \textit{(1)} of \Cref{lem:minia} and estimate \textit{(1)} of \Cref{lem:minia2} together show us that \(|\partial^\beta(y_n-X(y'))|\leq Cr(y)^{1-|\beta|}\) on \(B\), respectively when \(|\beta|=0\) and when \(|\beta|>0\). Using the Leibniz rule we can thus bound derivatives of \(\sqrt{f(y)-F(y)}=(y_n-X(y'))\sqrt{H(y)}\) by writing
\[
    \partial^\beta\sqrt{f(y)-F(y)}=\sum_{\gamma\leq\beta}\binom{\beta}{\gamma}\partial^{\beta-\gamma}(y_n-X(y'))\partial^\gamma\sqrt{H(y)}.
\]
Employing the triangle inequality and the derivative bounds computed above, we find now that
\[
    |\partial^\beta\sqrt{f(y)-F(y)}|\leq C\sum_{\gamma\leq\beta}|\partial^{\beta-\gamma}(y_n-X(y'))||\partial^\gamma\sqrt{H(y)}|\leq Cr(y)^{\frac{k+\alpha}{2}-|\beta|}.
\]
Since \(y\in B\) was arbitrary, estimate \textit{(1)} of \Cref{lem:nice} follows.

For estimates \textit{(2)} and \textit{(3)} we once again denote by \(\Lambda\) the integer part of \(\frac{k}{2}\), and set \(\lambda=\frac{k+\alpha}{2}-\Lambda\) so that the claimed estimates will follow if we can show that \([\partial^\beta\sqrt{f-F}]_{\lambda}(y)\leq Cr(y)^{\Lambda-|\beta|}\). To this end we observe that if \(w,z\in B\), then the Mean Value Theorem and the derivative estimate above together give
\[
    |\partial^\beta\sqrt{f(w)-F(w)}-\partial^\beta\sqrt{f(z)-F(z)}|\leq Cr(y)^{\frac{k+\alpha}{2}-|\beta|-1}|w-z|\leq Cr(y)^{\Lambda-|\beta|}|w-z|^\lambda.
\]
It follows by taking a limit supremum as \(w,z\rightarrow y\) that \([\partial^\beta\sqrt{f-F}]_{\lambda}(y)\leq Cr(y)^{\Lambda-|\beta|}\).
\end{proof}

There remains one more critical property of the remainder \(F\) to establish before we can proceed to extend the local results of this section to global ones. These are derivative and local semi-norm estimates, which follow from the properties of \(X\) established in \Cref{lem:minia2}.

\begin{lem}\label{lem:implicitest}
Assume that the hypotheses of \Cref{lem:minia} hold on \(B\) for some fixed \(x\), and define \(X\) as above on \(B'=\{x':x\in B\}\). If \(F(y')=f(y',X(y'))\) then for every \(y\in B\) the following estimates hold, 
\begin{itemize}
    \item[(1)] \(|\partial^\beta F(y)|\leq Cr(y)^{k+\alpha-|\beta|}\),
    \item[(2)] \([\partial^\beta F]_{\alpha}(y)\leq Cr(y)^{k-|\beta|}\).
\end{itemize}
\end{lem}

\begin{proof}
Given a multi-index \(\beta\) of order \(|\beta|\leq k\), we compute \(\partial^\beta F\) by first writing \(\beta=\mu+\rho\), for a multi-index \(\rho\leq \beta\) with \(|\rho|=1\), and with \(\mu=\beta-\rho\). It follows from the definition of \(X\) that
\[
    \partial^\beta F(y')=\partial^\mu[\partial^\rho f(y',X(y'))+\partial_{x_n}f(y',X(y'))\partial^\rho X(y')]=\partial^\mu[\partial^\rho f(y',X(y'))].
\]
To compute the remaining \(\mu\)-derivative on the right-hand side, we employ \Cref{lem:genchain} to get
\[
    \partial^\beta F(y')=\partial^\mu[\partial^\rho f(y',X(y'))]=\sum_{\eta\leq\mu }\sum_{\Gamma\in P(\eta)}C_{\mu,\Gamma}(\partial^{\beta-\eta}\partial^{|\Gamma|}_{x_n}f(y',X(y')))\prod_{\gamma\in\Gamma}\partial^\gamma X(y').
\]

Although the right-hand side of the formula above for \(\partial^\beta F\) appears to depend on our choice of \(\mu\), it turns out that this formula yields the same derivative regardless of which \(\mu\) is chosen. This is due to commutation relations which follow directly from the identity \(\partial_{x_n}f(y',X(y'))=0\). For instance, with second-order derivatives we have everywhere on the domain of \(F\) that
\[
    \partial_{x_i}\partial_{x_n}f(y',X(y'))\partial_{x_j}X(y')=\partial_{x_j}\partial_{x_n}f(y',X(y'))\partial_{x_i}X(y').
\]
Thus no generality is lost by fixing \(\rho\) and \(\mu\) as above to compute \(\partial^\beta F\). With this selection we observe by the triangle inequality that
\[
    |\partial^\beta F(y')|\leq C\sum_{\eta\leq\mu }\sum_{\Gamma\in P(\eta)}|\partial^{\beta-\eta}\partial^{|\Gamma|}_{x_n}f(y',X(y'))|\prod_{\gamma\in\Gamma}|\partial^\gamma X(y')|.
\]
Since \(|\mu|=|\beta|-1\leq k-1\), for each \(\gamma\) in the product above, we have that \(|\partial^\gamma X(y')|\leq Cr(y)^{1-|\gamma|}\) on \(B'\) by estimate \textit{(1)} of \Cref{lem:minia2}. Similarly, \Cref{cor:controlbound} and slow variation of \(r\) allow us to bound the derivatives of \(F\) above uniformly by powers of \(r(y)\) to get item \textit{(1)},
\[
    |\partial^\beta F(y')|\leq C\sum_{\eta\leq\mu }\sum_{\Gamma\in P(\eta)}r(y)^{k+\alpha-|\beta|+|\eta|-|\Gamma|}\prod_{\gamma\in\Gamma}r(y)^{1-|\gamma|}=Cr(y)^{k+\alpha-|\beta|}.
\]

For the estimates of \textit{(2)}, we take \(w',z'\in B'\) and for brevity we write \(\tilde{w}=(w',X(w'))\). As above, for some multi-index \(\mu< \beta\) of order \(|\beta|-1\) we can expand \(\partial^\beta F\) as follows:
\[
    \partial^\beta F(w')-\partial^\beta F(z')= \sum_{0\leq\eta\leq\mu}\sum_{\Gamma\in P(\eta)}C_{\beta,\Gamma}\bigg(\partial^{\beta-\eta}\partial^{|\Gamma|}_{x_n}f(\tilde{w})\prod_{\gamma\in\Gamma}\partial^\gamma X(w)-\partial^{\beta-\eta}\partial^{|\Gamma|}_{x_n}f(\Tilde{z})\prod_{\gamma\in\Gamma}\partial^\gamma X(z)\bigg).
\]
From here, an argument identical to that used in our proof of item \textit{(2)} of \Cref{lem:minia} shows that \([\partial^\beta F]_{\alpha}(y')\leq Cr(y)^{k-|\beta|}\) for every derivative of order \(|\beta|\leq k\) and every \(y'\in B'\), meaning that item \textit{(2)} holds as claimed.
\end{proof}

For later convenience, we now replace \(\nu\) with \(\frac{1}{2}\nu\) to ensure that the functions \(X\) and \(F\) are well-defined and satisfy the estimates established above on the dilated ball \(2B'\). This substitution comes only at the expense of making several constant coefficients in the estimates above larger. Extension of \(F\) to \(2B'\) is critical, as it allows us to show in the following lemma that there exists a \(C^{k,\alpha}(\mathbb{R}^{n-1})\) function which agrees with \(F\) on \(B'\). This function is required in the inductive argument that we employ to prove \Cref{thm:c3main}.

\begin{lem}\label{lem:crudeext}
There exists a non-negative function in \( C^{k,\alpha}(\mathbb{R}^{n-1})\) which agrees with \(F\) on \(B'\).
\end{lem}

\begin{proof}
Let \(\psi\) be a smooth non-negative function with compact support contained in the ball \(B(2,0)\subset\mathbb{R}^{n-1}\), such that \(\psi=1\) on the ball of half radius \(B(1,0)\subset\mathbb{R}^{n-1}\). Then set
\[
    \varphi(y')=\psi\bigg(\frac{y'-x'}{\nu r(x)}\bigg),
\]
so that \(\varphi\) is supported in \(2B'\) and \(\varphi=1\) on \(B'\). It remains to show that \(\varphi F\in C^{k,\alpha}(\mathbb{R}^n)\). To this end we first construct derivative estimates on \(\varphi\), observing that repeat differentiation gives
\[
    |\partial^\beta\varphi(y')|=\frac{1}{\nu^{|\beta|}r(x)^{|\beta|}}\bigg|\partial^\beta\psi\bigg(\frac{y'-x'}{\nu r(x)}\bigg)\bigg|\leq Cr(x)^{-|\beta|},
\]
since \(\psi\) and all of its derivatives are uniformly bounded independent of \(r\) on \(\mathbb{R}^{n-1}\). Additionally, we can observe for any multi-index \(\beta\) and \(0<\alpha\leq 1\) that \([\partial^\beta\varphi]_{\alpha,\mathbb{R}^n}=[\partial^\beta\varphi]_{\alpha,B'}\). Therefore by the Mean Value Theorem and the estimate on \(\partial^\beta\psi\) above, it follows that \([\partial^\beta\varphi]_{\alpha,\mathbb{R}^n}\leq Cr(x)^{-\alpha-|\beta|}\).

Combining these estimates on \(\varphi\) with the inequalities of \Cref{lem:implicitest}, a calculation using the Leibniz rule shows that for any multi-index \(\beta\) and \(y'\in B'\),
\[
    |\partial^\beta\varphi(y') F(y')|\leq C\sum_{\gamma\leq\beta}|\partial^\gamma\varphi (y')||\partial^{\beta-\gamma}F(y')|\leq C\sum_{\gamma\leq\beta}r(x)^{-|\gamma|}r(x)^{k+\alpha-|\beta|+|\gamma|}=Cr(x)^{k+\alpha-|\beta|}.
\]
Taking a supremum over \(y'\in 2B'\) yields \(\sup_{2B'}|\partial^\beta(\varphi F)|\leq Cr(x)^{k+\alpha-\beta}\), and since \(\varphi F\) is identically zero outside of \(2B'\) it follows that this bound holds over \(\mathbb{R}^n\) and \(\varphi F\in C^{k}(\mathbb{R}^{n-1})\). 

It remains to verify that \(\partial^\beta(\varphi F)\in C^{\alpha}(\mathbb{R}^{n-1})\) whenever \(|\beta|=k\). To this end we first use the sub-product rule of \Cref{lem:subprod} to write
\[
    [\partial^\beta(\varphi F)]_{\alpha,\mathbb{R}^{n-1}}=[\partial^\beta(\varphi F)]_{\alpha,2B'}\leq\sum_{\gamma\leq\beta}\binom{\beta}{\gamma}([\partial^{\beta-\gamma}\varphi]_{\alpha}\sup_{2B'}|\partial^\gamma F|+[\partial^{\beta-\gamma}F]_{\alpha,B'}\sup_{\mathbb{R}^{n-1}}|\partial^\gamma \varphi|).
\]
Employing the estimates from above, we see that
\[
    [\partial^\beta(\varphi F)]_{\alpha,\mathbb{R}^{n-1}}\leq C\sum_{\gamma\leq\beta}(r(x)^{-\alpha -|\beta|+|\gamma|}r(x)^{k+\alpha-|\gamma|}+r(x)^{k-|\beta|+|\gamma|}r(x)^{-|\gamma|})=Cr(x)^{k-|\beta|}
\]
In particular, we see that if \(|\beta|=k\) then \([\partial^\beta(\varphi F)]_{\alpha,\mathbb{R}^{n-1}}\leq C\), giving \(\varphi F\in C^{k,\alpha}(\mathbb{R}^{n-1})\).
\end{proof}

\textit{Remark}: In passing, we observe that by iterating the preceding argument for \(\varphi\) as above, we can show that in fact \(\varphi^mF\) is a non-negative \(C^{k,\alpha}(\mathbb{R}^{n-1})\) function for any \(m\in\mathbb{N}\). The most important consequence of the preceding result is that the restriction of \(F\) to a small enough ball can be extended to a function defined on all of \(\mathbb{R}^{n-1}\), and this extension can be achieved without sacrificing regularity.

To summarize the results of this section, if \(f\in C^{k,\alpha}(\mathbb{R}^n)\) is non-negative and bounded below by a constant multiple of \(r(x)^{k+\alpha}\) at the center of the ball \(B=B(x,\nu r(x))\), then \(\sqrt{f}\) is half as regular as \(f\) on \(B\). On the other hand, if \(f\) is bounded above at the center of \(B\) then we can construct a local decomposition of the form
\[
    f=g^2+F,
\]
where \(g=\sqrt{f-F}\in C^\frac{k+\alpha}{2}(B)\) and \(F\in C^{k,\alpha}(B)\) depends on \(n-1\) variables. Moreover, there exists a non-negative function in \(C^{k,\alpha}(\mathbb{R}^n)\) whose restriction to \(B'\) agrees with \(F\). In any case, we see that \(f\) can be locally decomposed.

\section{Partitions of Unity}\label{sec:pou}

To extend our local decompositions to global ones, we construct functions which sum to a constant on the set where \(f\) and its derivatives are nonzero, and which are compactly supported on the balls identified in the previous section. Our main instrument for extending local estimates is the following theorem.

\begin{thm}\label{thm:party}
Let \(r\) be slowly-varying on balls of the form \(B(x,\nu r(x))\). Then there exists a countable collection of balls \(\{B(x_j,\nu r_j)\}\), where \(r_j=r(x_j)\), and a partition of unity
\begin{equation}\label{eq:pou2}
    \sum_{j=1}^\infty\psi_j(x)^2=
    \begin{cases}
    1 & \textrm{if }\;r(x)>0,\\
    0 & \textrm{if }\;r(x)=0,
    \end{cases}
\end{equation}
which has the following properties:
\begin{itemize}
    \item[(1)] The support of each \(\psi_j\) is contained in the ball \(B(x_j,\nu r_j)\),
    \item[(2)] Every \(x\in\mathbb{R}^n\) belongs to at most \(N_n\) of the balls for a constant \(N_n\) that depends only on \(n\),
    \item[(3)] For every multi-index \(\beta\), \(\psi_j\) satisfies the estimate \(\sup_{\mathbb{R}^n}|\partial^\beta\psi_j|\leq C r_j^{-|\beta|}\),
    \item[(4)] For every multi-index \(\beta\), \(x\in\mathbb{R}^n\) and \(\alpha\in(0,1]\), \(\psi_j\) satisfies \([\partial^\beta\psi_j]_{\alpha}(x)\leq Cr_j^{-\alpha-|\beta|}\),
    \item[(5)] There exist \(N_n\) sub-collections of \(\{\psi_j\}\) that are each comprised of functions which have pairwise disjoint support (i.e. no two functions in any of the \(N_n\) sub-collections are nonzero simultaneously at any given point in \(\mathbb{R}^n\)). 
\end{itemize}
\end{thm}

\medskip

\textit{Remark}: Such a partition can be constructed for any function with the slow-variation property, and it is not exclusive to \(r\) as given by \eqref{eq:controlfunc}. A similar result to \Cref{thm:party} is proved by De Guzman in \cite[Lemma 1.2]{guzman}, however the version we require differs in several respects. As such, we include a complete and independent proof here. Our construction is motivated in spirit by the argument in \cite{Tataru}, but it differs considerably at several points.

\begin{proof}
We first set \(S=\{x\in\mathbb{R}^n:r(x)>0\}\) and recursively construct a cover of \(S\) by balls of the form \(B(x_j,\frac{1}{2}\nu r_j)\). Having chosen the first \(\ell\) balls, we add to that collection the ball \(B(x_{\ell+1},\frac{1}{2}\nu r_{\ell+1})\) centred at any point \(x_{\ell+1}\in S\setminus(\bigcup_{j=1}^\ell B(x_j,\frac{1}{2}\nu r_j))\). If \(x\in B(x_j,\frac{1}{2}\nu r_j)\) for some \(j\) then \(2r(x)>r_j>0\) by slow variation, meaning \(x\in S\), and from this construction we thus get
\begin{equation}\label{eq:seteq}
    S=\bigcup_{j=1}^\infty B\bigg(x_j,\frac{1}{2}\nu r_j\bigg).
\end{equation}

A useful property of these balls is that their dilates \(B(x_j,\frac{1}{8}\nu r_j)\) are pairwise disjoint. To see this, assume toward a contradiction that \(x\in B(x_i,\frac{1}{8}\nu r_i)\cap B(x_j,\frac{1}{8}\nu r_j)\) for some \(i>j\). Then \(r_i\leq \frac{4}{3}r(z)\leq \frac{5}{3}r_j\) by slow variation and
\[
    |x_i-x_j|\leq |x_i-x|+|x_j-x|<\frac{1}{8}\nu(r_i+r_j)\leq \frac{1}{8}\nu\bigg(\frac{5}{3}r_j+r_j\bigg)=\frac{1}{3}\nu r_j<\frac{1}{2}\nu r_j.
\]
Thus \(x_i\in B(x_j,\frac{1}{2}\nu r_j)\), contrary to our construction. Using the pairwise disjoint property we can show that any given point belongs to only finitely many balls of the form \(B(x_j,\nu r_j)\). Let \(m\) denote the number of balls which contain a fixed point \(x\in S\). If \(B(x_j,\nu r_j)\) is one of these balls then \(\frac{4}{5}r(x)\leq r_j \leq \frac{4}{3}r(x)\) by slow variation, meaning that \(B(x_j,\frac{1}{10}r(x))\subseteq B(x_j,\frac{1}{8}r_j)\). Additionally, if \(y\in B(x_j,\frac{1}{8}\nu r_j)\) then \(|x-y|\leq |x-x_j|+|y-x_j|< \frac{9}{8}\nu r_j\leq \frac{3}{2}r(x)\), meaning that \(B(x_j,\frac{1}{8}\nu r_j)\subseteq B(x,\frac{3}{2}r(x))\). Taking a union of the \(m\) disjoint balls centred at \(x_{j_1},\dots,x_{j_m}\) whose dilates contain \(x\), we see from these estimates that
\[
    \bigcup_{k=1}^m B\bigg(x_{j_k},\frac{1}{10}\nu r(x)\bigg)\subseteq \bigcup_{k=1}^m B\bigg(x_{j_k},\frac{1}{8}\nu r_j\bigg)\subseteq B\bigg(x,\frac{3}{2}\nu r(x)\bigg).
\]
Since the \(m\) balls on the left are pairwise disjoint and they all have the same size, it follows that
\[
     \frac{mV \nu^nr(x)^n}{10^n}\leq \frac{3^nV \nu^nr(x)^n}{2^n},
\]
where \(V\) denotes the volume of the unit ball in \(\mathbb{R}^n\). This simplifies to give \(m \leq 15^n\), meaning that each \(x\in S\) belongs to at most \(N_n=15^n\) balls of the form \(B(x_j,\nu r_j)\) from our cover of \(S\).

Next let \(\psi\) be a non-negative smooth function that is supported in the unit ball in \(\mathbb{R}^n\), which satisfies \(\psi\leq 1\) everywhere and \(\psi=1\) on \(B(0,\frac{1}{2})\). For \(x\in S\) we then define
\begin{equation}\label{eq:partfunc}
    \psi_j(x)=\psi\bigg(\frac{x-x_j}{\nu r_j}\bigg)\bigg(\sum_{k=1}^\infty\psi\bigg(\frac{x-x_k}{\nu r_k})\bigg)^2\bigg)^{-\frac{1}{2}},
\end{equation}
so that \(\sup|\psi_j|\leq 1\) and \(\psi_j\) is supported in \(B(x_j,\nu r_j)\). Each \(x\) belongs to some ball of the form \(B(x_j,\frac{1}{2}\nu r_j)\) and at most \(N_n\) balls of the form \(B(x_j,\nu r_j)\), meaning that 
\begin{equation}\label{eq:lowercontrol}
    1\leq \sum_{k=1}^\infty\psi\bigg(\frac{x-x_k}{\nu r_k}\bigg)^2\leq N_n.
\end{equation}
Consequently the sum in \eqref{eq:partfunc} is nonzero and finite, so there is no issue of convergence.

Now we establish derivative and semi-norm estimates for \(\psi_j\). Since \(\psi\) is a smooth function with compact support, \(\sup_{\mathbb{R}^n}|\partial^\beta\psi|\leq C \) for a constant that depends on \(\beta\). By the general Leibniz rule \Cref{lem:Leib}, we may write
\begin{equation}\label{eq:leib}
    \partial^\beta\psi_j(x)=\sum_{\gamma\leq\beta}\binom{\beta}{\gamma}\bigg(\partial^\gamma\psi\bigg(\frac{x-x_j}{\nu r_j}\bigg)\bigg)\bigg(\partial^{\beta-\gamma}\bigg(\sum_{k=1}^\infty\psi\bigg(\frac{x-x_{k}}{\nu r_k}\bigg)^2\bigg)^{-\frac{1}{2}}\bigg).
\end{equation}
The derivatives of \(\psi_j\) are bounded, so \(\sup|\partial^\gamma\psi(\frac{x-x_j}{\nu r_j})|\leq C r_j^{-|\gamma|}\). Additionally, we can use the first inequality of \eqref{eq:lowercontrol} along with the fact that \(cr_j\leq r_k\leq Cr_j\) when \(x\in B(x_j,\nu r_j)\cap B(x_k,\nu r_k) \), to get from the chain rule \Cref{cor:chain2} that 
\[
    \bigg|\partial^{\beta-\gamma}\bigg(\sum_{k=1}^\infty\psi\bigg(\frac{x-x_{k}}{\nu r_k}\bigg)^2\bigg)^{-\frac{1}{2}}\bigg|\leq Cr_j^{|\gamma|-|\beta|}.
\]
Combining these estimates with \eqref{eq:leib} we get \(\sup|\partial^\beta \psi_j|\leq Cr_j^{-|\beta|}\). For semi-norm estimates we let \(y,z\in B(x_j,\nu r_j)\) and use the Mean Value Theorem together with derivative bound \textit{(3)} to get
\[
    |\partial^\beta\psi_j(y)-\partial^\beta\psi_j(z)|\leq Cr_j^{-|\beta|-1}\leq  C\nu^{1-\alpha} r_j^{-\alpha-|\beta|}|y-z|^\alpha.
\]
The semi-norm estimate \textit{(4)} then follows from taking a limit supremum as \(y,z\rightarrow x\).

Finally, to prove item \textit{(5)} we construct an infinite graph \(G\) as follows: to each ball in the collection \(\{B(x_j,\nu r_j)\}\) we assign a vertex, and we add an edge between two given vertices if their corresponding balls intersect. Since our collection has bounded overlap, each vertex of \(G\) has degree at most \(N_n\), and by \cite[Theorem 3]{behzad} it follows that the chromatic number of \(G\) is bounded above by \(N_n\). Consequently we can find \(N_n\) sub-collections of \(\{B(x_j,\nu r_j)\}\), each comprised of pairwise disjoint balls. The functions \(\psi_j\) corresponding to the balls in any given sub-collection of \(\{B(x_j,\nu r_j)\}\) form the desired subset of functions having pairwise disjoint supports, giving item \textit{(5)} and completing the proof.
\end{proof}

\textit{Remark}: Our use of squares in \eqref{eq:pou2} is helpful in the proof of \Cref{thm:c3main}, however from a general standpoint it is insignificant. Given a countable collection of integers \(\{m_j\}\) we can replace \eqref{eq:partfunc} with the alternative definition
\[
    \psi_j(x)=\psi\bigg(\frac{x-x_j}{\nu r_j}\bigg)\bigg(\sum_{k=1}^\infty\psi\bigg(\frac{x-x_k}{\nu r_k})\bigg)^{m_k}\bigg)^{-\frac{1}{m_j}}.
\]
This allows us to replace the sum of squares in \eqref{eq:pou2} with \(\sum_{j=1}^\infty\psi_j^{m_j}\), and it is easy to check that \(\psi_j\) defined in this way satisfies similar differential inequalities to those found above. For the purposes of this work, our continued use of squares is sufficient. However, this observation may be useful if one wishes to decompose into sums of arbitrary integer powers instead.

Combining \Cref{thm:party} with the results of the preceding section, we are now able to extend our local decompositions of \(f\) to a global one. Our first result to this end essentially states that the local roots identified in the previous section extend to \(C^{\frac{k+\alpha}{2}}(\mathbb{R}^n)\) functions when multiplied by the partition functions constructed in Theorem \ref{thm:party}.

\begin{lem}\label{lem:rootscors}
Let \(f\) be a non-negative \(C^{k,\alpha}(\mathbb{R}^n)\) function, and let \(\psi_j\) be one of the partition functions in \eqref{eq:pou2} supported in \(B(x_j,\nu r_j)\). If \(f(x_j)\geq \omega r_j^{k+\alpha}\) then the following estimates hold pointwise in \(\mathbb{R}^n\),
\begin{itemize}
    \item[(1)] \(|\partial^\beta[\psi_j\sqrt{f}](x)|\leq Cr(x)^{\frac{k+\alpha}{2}-|\beta|}\),
    \item[(2)] \([\partial^\beta(\psi_j\sqrt{f})]_{\frac{\alpha}{2}}(x)\leq Cr(x)^{\frac{k}{2}-|\beta|}\) if \(k\) is even,
    \item[(3)] \([\partial^\beta(\psi_j\sqrt{f})]_{\frac{1+\alpha}{2}}(x)\leq Cr(x)^{\frac{k-1}{2}-|\beta|}\) if \(k\) is odd.
\end{itemize}
Similarly, if \(f(x_j)<\omega r_j^{k+\alpha}\) and the other hypotheses of \Cref{lem:minia} are satisfied, then for \(F\) as in \Cref{lem:nice} we also have
\begin{itemize}
    \item[(4)] \(|\partial^\beta[\psi_j\sqrt{f-F}](x)|\leq Cr(x)^{\frac{k+\alpha}{2}-|\beta|}\),
    \item[(5)] \([\partial^\beta(\psi_j\sqrt{f-F})]_{\frac{\alpha}{2}}(x)\leq Cr(x)^{\frac{k}{2}-|\beta|}\) if \(k\) is even,
    \item[(6)] \([\partial^\beta(\psi_j\sqrt{f-F})]_{\frac{1+\alpha}{2}}(x)\leq Cr(x)^{\frac{k-1}{2}-|\beta|}\) if \(k\) is odd.
\end{itemize}
\end{lem}

\begin{proof}
First we bound the derivatives of \(\psi_j\sqrt{f}\) pointwise on the ball \(B=B(x_j,\nu r_j)\), employing \Cref{lem:Leib} and the triangle inequality to achieve the following pointwise bound
\[
    |\partial^\beta[\psi_j\sqrt{f}](x)|\leq \sum_{\gamma\leq\beta}\binom{\beta}{\gamma}|\partial
    ^\gamma \sqrt{f(x)}||\partial^{\beta-\gamma}\psi_j(x)|.
\]
Using the derivative bounds on \(\psi_j\) from \Cref{thm:party}, the local derivative bounds on \(\sqrt{f}\) from \Cref{lem:local1}, and the fact that \(r(x)\leq Cr_j\) for \(x\in B\) by slow variation of \(r\), it follows that
\[
    |\partial^\beta[\psi_j\sqrt{f}](x)|\leq C\sum_{\gamma\leq\beta}r(x)^{\frac{k+\alpha}{2}-|\gamma|}r_j^{|\gamma|-|\beta|}\leq Cr(x)^{\frac{k+\alpha}{2}-|\beta|}.
\]
If \(x\not\in B\) then \(\psi_j\) and all of its derivatives are identically zero, so \textit{(1)} holds trivially there and we conclude that \textit{(1)} holds on all of \(\mathbb{R}^n\).

For \textit{(2)} and \textit{(3)} we let \(\Lambda\) denote the integer part of \(\frac{k}{2}\) and set \(\lambda=\frac{k+\alpha}{2}-\Lambda\) so that it suffices to show that \([\partial^\beta(\psi_j\sqrt{f})]_\lambda(x)\leq Cr(x)^{\Lambda-|\beta|}\) on \(\mathbb{R}^n\). To this end we employ the sub-product rule of \Cref{lem:subprod} to see for \(x\in B\) that
\[
    [\partial^\beta(\psi_j\sqrt{f})]_\lambda(x)\leq \sum_{\gamma\leq\beta}\binom{\beta}{\gamma}([\partial^{\beta-\gamma}\sqrt{f}]_{\lambda}(x)\sup_{B}|\partial^\gamma \psi_j|+[\partial^{\beta-\gamma}\psi_j]_\lambda(x)\sup_{B}|\partial^\gamma \sqrt{f}|).
\]
Employing the derivative estimates on \(\psi_j\) and \(\sqrt{f}\) on \(B\) from \Cref{thm:party} and \Cref{lem:local1} as above, and using slow variation of \(r\) on \(B\) to see that \(cr_j\leq r(x)\leq Cr_j\), we get
\[
    [\partial^\beta(\psi_j\sqrt{f})]_\lambda(x)\leq C\sum_{\gamma\leq\beta}(r(x)^{\Lambda-|\beta|+|\gamma|}r_j^{-|\gamma|}+r_j^{|\gamma|-|\beta|-\lambda}r(x)^{\frac{k+\alpha}{2}-|\gamma|})\leq Cr(x)^{\Lambda-|\beta|}.
\]
On the other hand, if \(x\not\in B\) then \([\partial^\beta(\psi_j\sqrt{f})]_\lambda(x)\) is identically zero, and the estimate holds trivially. Therefore \textit{(2)} and \textit{(3)} hold throughout \(\mathbb{R}^n\). The proofs of items \textit{(4)} through \textit{(6)} are virtually identical, using the estimates from \Cref{lem:nice} in place of those from \Cref{lem:local1}.
\end{proof}

The pointwise semi-norm estimates given in \Cref{lem:rootscors} imply some useful inclusions.

\begin{cor}\label{cor:theniceone}
The functions \(\psi_j\sqrt{f}\) and \(\psi_j\sqrt{f-F}\) defined above belong to \(C^\frac{k+\alpha}{2}(\mathbb{R}^n)\).
\end{cor}

Finally, with the preceding constructions in place, we are equipped to prove the main result.

\section{Proof of \texorpdfstring{\Cref{thm:c3main}}{}}\label{sec:c3proof}

Arguing by induction on \(n\), we show that if \(k=2\) or \(k=3\) then identity \eqref{eq:decompident} holds for every non-negative \(f\in C^{k,\alpha}(\mathbb{R}^{n})\). Indeed, we also show that the functions \(g_1,\dots,g_{m_n}\) satisfy the following set of pointwise estimates:
\begin{itemize}
    \item[\textit{(1)}] \(|\partial^\beta g_j(x)|\leq Cr_f(x)^{\frac{k+\alpha}{2}-|\beta|}\),
    \item[\textit{(2)}] \([\partial^\beta g_j]_{\frac{\alpha}{2},\mathbb{R}^{n-1}}(x)\leq Cr_f(x)^{\frac{k}{2}-|\beta|}\) if \(k\) is even,
    \item[\textit{(3)}] \([\partial^\beta g_j]_{\frac{1+\alpha}{2},\mathbb{R}^{n-1}}(x)\leq Cr_f(x)^{\frac{k-1}{2}-|\beta|}\) if \(k\) is odd.
\end{itemize}
The function \(r_f\) is as in \eqref{eq:controlfunc}, and we now include the subscript to emphasize dependence on \(f\).

Like we did in \Cref{sec:reg}, we prove all estimates for arbitrary integer-valued \(k\geq 0\) and we do not restrict to \(k=2\) and \(k=3\), since the argument that follows is valid whenever inequality \eqref{eq:needed} holds pointwise. The generality of this argument saves us from repeating a virtually identical proof for the separate but similar cases when \(k=2\) and \(k=3\).

\subsection*{\normalsize Base Case in One Dimension}

For a base case of \Cref{thm:c3main} in one dimension, we show that every non-negative \(C^{k,\alpha}(\mathbb{R})\) function can be decomposed as a sum of half-regular squares for every \(k\geq 0\) and \(0<\alpha\leq 1\). Such a result is already proved in \cite{Bony2}, and for completeness we include Bony's result as \Cref{thm:onedim}. However, the claim appearing in \cite{Bony2} omits the derivative and semi-norm estimates that we require for our high dimensional inductive argument, so we furnish a complete proof together with the required estimates. We emphasize that this one-dimensional result is not new, and we do not show that \(m_1=2\) as Bony does, however our proof avoids some of the difficult technical algebraic arguments in \cite{Bony2} and thus it is considerably shorter.

Fixing a non-negative function \(f\in C^{k,\alpha}(\mathbb{R})\) and using \Cref{thm:party}, we construct a partition of unity using the control function \(r_f\) define from \eqref{eq:controlfunc}, so that we may write \(f\) as follows,
\[
    f=\sum_{j=1}^\infty\psi_j^2f.
\]
For a sufficiently small parameter \(\nu\), each function \(\psi_j\) in the sum above is supported on the an interval of the form \(B(x_j,\nu r_j)=(x_j-\nu r_j,x_j+\nu r_j)\). Fixing \(j\in\mathbb{N}\), we observe that at the center of this interval, either \(f(x_j)\geq \omega\nu r_j^{k+\alpha}\) or the converse \(f(x_j)<\omega\nu r_j^{k+\alpha}\) is true.

In the first case, \Cref{lem:rootscors} shows that \(\psi_j\sqrt{f}\) is a \(C^\frac{k+\alpha}{2}(\mathbb{R})\) function that satisfies the pointwise inequalities \textit{(1)} through \textit{(3)} above. On the other hand, if \(f(x_j)<\omega\nu r_j^{k\alpha}\) then \Cref{lem:minia} gives a unique local minimum \(X_j\) of \(f\) near \(x_j\) for which the number \(F_j=f(X_j)\) satisfies 
\[
    \psi_j\sqrt{f-F_j}\in C^\frac{k+\alpha}{2}(\mathbb{R}),
\] 
once again by \Cref{lem:rootscors}. Moreover, \Cref{lem:rootscors} shows that \(\psi_j\sqrt{f-F_j}\) satisfies the required pointwise inequalities \textit{(1)} through \textit{(3)}. In the case that \(f(x_j)<\omega\nu r_j^{k\alpha}\) then, we see that \(f\) can be decomposed into a sum of two squares on the interval \(B(x_j,\nu r_j)\),
\[
    \psi_j^2f=(\psi_j\sqrt{f-F_j})^2+(\psi_j\sqrt{F_j})^2.
\]

To proceed with our decomposition in one dimension, we must show that \(\psi_j\sqrt{F_j}\) is a \(C^\frac{k+\alpha}{2}(\mathbb{R})\) function which satisfies the required pointwise estimates. This turns out to be rather straightforward, since \(F_j\) is a non-negative constant. By slow variation of \(r_f\) we have \(F_j\leq Cr_j^{k+\alpha}\), and it follows by item \textit{(3)} of \Cref{thm:party} that for \(x\in B(x_j,\nu r(x_j))\) the following estimate holds,
\[
    |\partial^\beta[\psi_j\sqrt{F_j}](x)|=\sqrt{F_j}|\partial^\beta\psi_j(x)|\leq Cr_j^{\frac{k+\alpha}{2}-|\beta|}\leq Cr(x)^{\frac{k+\alpha}{2}-|\beta|}.
\]
This bound trivially holds outside the support of \(\psi_j\), meaning it holds everywhere as required.

Next we prove that \(\psi_j\sqrt{F_j}\) satisfies the semi-norm estimates \textit{(2)} and \textit{(3)}. To deal with the cases in which \(k\) is odd and even simultaneously, we employ the usual tactic defining \(\Lambda\) as integer part of \(\frac{k}{2}\) and setting \(\lambda=\frac{k+\alpha}{2}-\Lambda\), so that it suffices to prove that \([\partial^\beta(\psi_j\sqrt{F})]_\lambda(x)\leq Cr(x)^{\Lambda-|\beta|}\). This follows at once from item \textit{(4)} of \Cref{thm:party} together with slow variation,
\[
    [\partial^\beta(\psi_j\sqrt{F})]_\lambda(x)=\sqrt{F}[\partial^\beta\psi_j]_\lambda(x)\leq Cr_j^\frac{k+\alpha}{2}r_j^{-\lambda-|\beta|}=Cr_j^{\Lambda-|\beta|}\leq Cr(x)^{\Lambda-|\lambda|}.
\]
Consequently, \(\psi_j\sqrt{F}\) satisfies the estimates \textit{(1)} through \textit{(3)} required for our induction.

Regardless of the behaviour of \(f\) at \(x_j\), we see now that we can write \(\psi_j^2f\) as a sum of at most two squares of functions which satisfy the pointwise estimates \textit{(1)} through \textit{(3)}. After relabelling, we can therefore write
\begin{equation}\label{eq:infsum}
    f=\sum_{j=1}^\infty g_j^2
\end{equation}
for functions \(g_j\in C^\frac{k+\alpha}{2}(\mathbb{R})\). By estimate \textit{(5)} of \Cref{thm:party}, we can identify \(m_1=2N_{1}=30\) sub-collections of functions in the sum above which enjoy pairwise disjoint support. Fix any of these sub-collections, and let its functions be indexed by \(S\subseteq\mathbb{N}\). Observe that the function 
\begin{equation}\label{eq:recombsum}
    g_S=\sum_{j\in S}g_j
\end{equation}
belongs to \(C^\frac{k+\alpha}{2}(\mathbb{R}^n)\) by \Cref{lem:recomb}. Moreover, \(g_S\) satisfies the pointwise estimates \textit{(1)} through \textit{(3)} above owing to the argument employed in proving \Cref{lem:recomb}, and we can also observe that
\[
    g_S^2=\bigg(\sum_{j\in S}g_j\bigg)^2=\sum_{j\in S}g_j^2,
\]
since the functions in the sum have pairwise disjoint supports. Repeating this argument for each of the \(m_1\) sub-collections identified from \eqref{eq:infsum}, and relabelling the recombined functions as \(g_1,\dots,g_{m_1}\), we see that we can write
\[
    f=\sum_{j=1}^{m_1}g_j^2.
\]
Each \(g_j\) above satisfies estimates \textit{(1)} through \textit{(3)} on \(\mathbb{R}^n\), and we conclude that the base case of our induction holds. Specifically, this follows from taking \(k=2\) or \(k=3\).

Further, since the derivative estimates which require \eqref{eq:needed} are bypassed in one dimension since \(F_j\) is constant, we see that the preceding argument is valid for any \(k\geq 0\), meaning that our argument above also proves most of \Cref{thm:onedim}. The only shortcoming is that our argument above does not recover the optimal constant \(m_1=2\) found in \cite{Bony}.

\medskip

\subsection*{\normalsize Inductive Step in Higher Dimensions}

Now we proceed with the inductive stage of our argument, assuming for our inductive hypothesis that for every non-negative function \(f\in C^{k,\alpha}(\mathbb{R}^{n-1})\) with \(k=2\) or \(k=3\), there exist \(g_1,\dots,g_{m_{n-1}}\) satisfying estimates \textit{(1)}, \textit{(2)} and \textit{(3)} from the beginning of this section for which
\[
    f=\sum_{j=1}^{m_{n-1}}g_j^2.
\]
To this end we fix a non-negative function \(f\in C^{k,\alpha}(\mathbb{R}^n)\). Once again using \Cref{thm:party}, we are able to form the partition of unity induced by \(r_f\) as in \eqref{eq:controlfunc} to write
\[
    f=\sum_{j=1}^\infty\psi_j^2f.
\]

An identical argument to that employed for the base case shows that either \(\psi_j\sqrt{f}\) satisfies the required derivative estimates and belongs to the half-regular H\"older space, or we can write 
\[
    \psi_j^2f=(\psi_j\sqrt{f-F_j})^2+\psi_j^2F_j.
\]
From \Cref{lem:rootscors}, the first function on the right-hand side above satisfies the required estimates \textit{(1)} through \textit{(3)} and belongs to the appropriate H\"older space.

On the other hand, the remainder function \(F_j\) is non-negative and by \Cref{lem:crudeext} it can be extended from the support of \(\psi_j\) to a function in \(C^{k,\alpha}(\mathbb{R}^{n-1})\) which agrees with \(F_j\) on this support set. For the latter reason, we identify \(F_j\) with its extension to \(\mathbb{R}^{n-1}\), and from our inductive hypothesis, it follows that can decompose \(F_j\) as follows,
\begin{equation}\label{eq:inductivedecomps}
    F_j=\sum_{\ell=1}^{m_{n-1}}g_\ell^2.
\end{equation}
By hypothesis, the functions \(g_1,\dots,g_{m_{n-1}}\) above each satisfy the pointwise estimates \textit{(1)}, \textit{(2)} and \textit{(3)} with \(r_f\) replaced by \(r_{F_j}\). Further, the differential inequalities satisfied by \(F_j\) which we proved in \Cref{lem:implicitest} continue to hold for the extension of \(F_j\) on the support of \(\psi_j\). So we have for \(x\) in the support of \(\psi_j\) that
\[
    r_F(x')=\max\{F(x')^\frac{1}{k+\alpha},\sup_{|\xi|=1}[\partial^2_\xi F(x')]_+^\frac{1}{k-2+\alpha}\}\leq Cr_f(x),
\]
meaning that the estimates satisfied by \(g_1,\dots,g_{m_{n-1}}\) in \eqref{eq:inductivedecomps} actually hold when \(r_F\) is replaced by \(r_f\). The argument employed to prove \Cref{lem:rootscors} shows now that \textit{(1)} through \textit{(3)} are also satisfied by \(\psi_jg_\ell\) for each \(\ell=1,\dots,m_{n-1}\) and \(x\in\mathbb{R}^n\). That is,
\begin{itemize}
    \item[\textit{(1)}] \(|\partial^\beta[\psi_j g_\ell](x)|\leq Cr_f(x)^{\frac{k+\alpha}{2}-|\beta|}\),
    \item[\textit{(2)}] \([\partial^\beta(\psi_j g_\ell)]_{\frac{\alpha}{2},\mathbb{R}^{n-1}}(x)\leq Cr_f(x)^{\frac{k}{2}-|\beta|}\) if \(k\) is even,
    \item[\textit{(3)}] \([\partial^\beta (\psi_jg_\ell)]_{\frac{1+\alpha}{2},\mathbb{R}^{n-1}}(x)\leq Cr_f(x)^{\frac{k-1}{2}-|\beta|}\) if \(k\) is odd.
\end{itemize}
It follows from these estimates that \(\psi_j g_\ell\in C^\frac{k+\alpha}{2}(\mathbb{R}^n)\), and thus we can write \(\psi_j^2F_j\) as a sum of at most \(m_{n-1}\) half-regular squares.

In summary, for each \(j\in\mathbb{N}\) we can write \(\psi_j^2f\) as a sum of \(m_{n-1}+1\) squares in \(C^\frac{k+\alpha}{2}(\mathbb{R}^n)\), and indeed if \(f\) is locally bounded below then only one square is needed. Combining the squares obtained for each \(j\) and relabelling, we see now that we may write
\[
    f=\sum_{j=1}^\infty g_j^2,
\]
for \(g_j\) satisfying \textit{(1)} through \textit{(3)} everywhere. By item \textit{(5)} of \Cref{thm:party}, we may now partition the sum above into \(N_n(m_{n-1}+1)\) sub-collections of functions which enjoy pairwise disjoint supports. Recombining and relabelling these functions exactly as we did in the one-dimensional setting, we see that we can set \(m_n=N_n(m_{n-1}+1)\) to write
\[
    f=\sum_{j=1}^{m_n}g_j^2
\]
for functions \(g_1,\dots,g_{m_n}\) which inherit the required differential inequalities. Finally, since \(f\) was any non-negative \(C^{k,\alpha}(\mathbb{R}^n)\) function for \(k=2\) or \(k=3\) this completes our inductive step, and it follows that \Cref{thm:c3main} holds for every \(n\).\hfill\qedsymbol

\section{Decompositions Over Open Sets}\label{sec:ext}

In the study of partial differential equations, and in many other settings, it is often useful to restrict attention to functions defined on an open set \(\Omega\subset\mathbb{R}^n\). As Bony points out in \cite{Bony2}, the decomposition theorem of Fefferman \& Phong in \cite{Fefferman-Phong} for \(C^{3,1}(\mathbb{R}^n)\) functions can be extended to \(C^{3,1}_{\mathrm{loc}}(\Omega)\), whenever \(\Omega\) is an open subset of \(\mathbb{R}^n\). Recall that \(C^{k,\alpha}_{\mathrm{loc}}(\Omega)\) is the set of functions whose \(k^\mathrm{th}\)-order derivatives are \(\alpha\)-H\"older continuous on compact subsets of \(\Omega\).

Bony does not give an explicit proof of this claim, but points to a helpful lemma in \cite{Bony} which can be used to obtain the desired extension of Fefferman \& Phong's result. Motivated by Bony's remark, in this section we prove the following generalization of \Cref{thm:c3main}.

\begin{thm}\label{thm:locdecomp}
Let \(f\in C_{\mathrm{loc}}^{k,\alpha}(\Omega)\) be non-negative on an open set \(\Omega\), for \(k\leq 3\) and \(\alpha\leq 1\). Then
\[
    f=\sum_{j=1}^{m_n}g_j^2
\]
for functions \(g_1,\dots, g_{m_n}\in C^{\frac{k+\alpha}{2}}_{\mathrm{loc}}(\Omega)\). Moreover, \(m_n\) depends only on the dimension \(n\).
\end{thm}

To prove this generalization, we require the following modification of \cite[Lemma 2.1]{Bony}, which enables us to extend \(C^{k,\alpha}_{\mathrm{loc}}(\Omega)\) functions by zero to all of \(\mathbb{R}^n\). Since our version differs somewhat from that appearing in \cite{Bony}, and since the original result is stated and proved in French, we furnish a translation in the form of our modified result, along with a proof.

\begin{lem}[Bony]\label{lem:extbyzero}
Let \(p\) be a positive continuous function defined on an open set \(\Omega\subset\mathbb{R}^n\). There exists a function \(\varphi\in C^\infty(\Omega)\) such that \(\varphi>0\) on \(\Omega\) and for every multi-index \(\beta\),
\begin{equation}\label{eq:boundarydecay}
    \lim_{x\rightarrow\partial\Omega}\frac{\partial^\beta \varphi(x)}{p(x)}=0.
\end{equation}
That is, all derivatives of the function \(\varphi\) decay faster than \(p\) approaching the boundary of \(\Omega\).
\end{lem}

\begin{proof}
We define \(\varphi\) in such a way that derivatives of \(\varphi\) decay exponentially approaching the boundary of \(\Omega\), and to do this we require some component functions. For \(x\in\Omega\) define \(\delta(x)\) as the distance from \(x\) to \(\partial\Omega\). Explicitly, 
\[
    \delta(x)=\inf_{y\in\partial\Omega}|x-y|.
\]
Also let \(\psi\) be a non-negative, smooth function compactly supported in the ball of radius \(\frac{1}{3}\) about the origin, and for \(x\in\Omega\) define a function \(q\) as a local minimum of \(p\),
\[
    q(x)=\inf_{y\in B(x,\frac{1}{2}\delta(x))}p(y).
\]
Equipped with these functions, we claim that the following function has the desired properties,
\[
    \varphi(x)=\int_{\Omega}\frac{q(z)}{\delta(z)^n}\psi\bigg(\frac{x-z}{\delta(z)}\bigg)e^{-\frac{1}{\delta(z)}}dz.
\]

First, observe that the integrand may only be positive if \(x\in B(z,\frac{1}{3}\delta(z))\). If this inclusion holds then \(\delta(x)>\frac{2}{3}\delta(z)\), and the integrand above is nonzero only when \(|x-z|<\frac{1}{3}\delta(z)<\frac{1}{2}\delta(x)\), meaning that the integrand is supported on the ball \(B(x,\frac{1}{2}\delta(x))\subset\Omega\). Thus for \(x\in\Omega\),
\[
    \varphi(x)=\int_{B(x,\frac{1}{2}\delta(x))}\frac{q(z)}{\delta(z)^n}\psi\bigg(\frac{x-z}{\delta(z)}\bigg)e^{-\frac{1}{\delta(z)}}dz.
\]
Additionally, if \(x\not\in B(z,\frac{1}{2}\delta(z))\) then \(|x-z|\geq \frac{1}{2}\delta(z)\) and \(x\not\in B(z,\frac{1}{3}\delta(z))\), meaning that the integrand vanishes. Consequently on the domain of integration we have \(q(z)\leq p(x)\) and
\[
    |\partial^\beta\varphi(x)|=\bigg|\int_{\Omega}\frac{q(z)}{\delta(z)^n}\partial^\beta_x\psi\bigg(\frac{x-z}{\delta(z)}\bigg)e^{-\frac{1}{\delta(z)}}dz\bigg|=\bigg|\int_{B(x,\frac{1}{2}\delta(x))}\frac{q(z)}{\delta(z)^n}\partial^\beta_x\psi\bigg(\frac{x-z}{\delta(z)}\bigg)e^{-\frac{1}{\delta(z)}}dz\bigg|.
\]
Since \(\psi\in C_0^\infty(\mathbb{R}^n)\) by assumption, for each \(\beta\) we have \(|\partial^\beta\psi|\leq C_\beta\) uniformly, and it follows that
\[
    \bigg|\partial^\beta_x\psi\bigg(\frac{x-z}{\delta(z)}\bigg)\bigg|\leq \frac{C_\beta}{\delta(z)^{|\beta|}}.
\]
Employing this estimate together with our bound on \(q\), and using that \(\frac{1}{2}\delta(x)\leq \delta(z)\leq \frac{3}{2}\delta(x)\) for \(z\in B(x,\frac{1}{2}\delta(x))\), we see that
\[
    |\partial^\beta\varphi(x)|\leq C_\beta p(x)\int_{B(x,\frac{1}{2}\delta(x))}\frac{e^{-\frac{1}{\delta(z)}}}{\delta(z)^{n+|\beta|}}dz\leq  2^{n+|\beta|}C_\beta p(x)\frac{e^{-\frac{2}{3\delta(x)}}}{\delta(x)^{n+|\beta|}}\int_{B(x,\frac{1}{2}\delta(x))}dz.
\]
The last integral is the volume of the ball \(B(x,\frac{1}{2}\delta(x))\), which is proportional to \(\delta(x)^n\). It follows that for \(x\in\Omega\) we have
\begin{equation}\label{eq:nicebound}
    |\partial^\beta\varphi(x)|\leq  C p(x)\frac{e^{-\frac{2}{3\delta(x)}}}{\delta(x)^{|\beta|}}
\end{equation}
for \(C\) depending only on \(n\), \(\beta\), and \(\psi\). We see now that the claimed decay result at the boundary holds, since for any multi-index \(\beta\) we have
\[
    \lim_{\delta(x)\rightarrow0^+}\frac{|\partial^\beta\varphi(x)|}{p(x)}\leq C\lim_{\delta(x)\rightarrow0^+}\frac{e^{-\frac{2}{3\delta(x)}}}{\delta(x)^{|\beta|}}=0.
\]
Finally, we note that \(q\) is strictly positive on \(B(x,\frac{1}{2}\delta(x))\), as are the remaining terms of the integrand when \(\delta(x)>0\), meaning that \(\varphi>0\) in \(\Omega\).
\end{proof}

It is a straightforward consequence of the Mean Value Theorem that for \(\varphi\) as above, the following limit also holds for any positive \(\alpha\) with \(0<\alpha\leq 1\),
\begin{equation}\label{eq:alphadecay}
    \lim_{x\rightarrow\partial\Omega}\frac{[\partial^\beta \varphi]_\alpha(x)}{p(x)}=0.
\end{equation}
Additionally, we note that for any multi-index \(\beta\) the derivative \(\partial^\beta\varphi/p\) is bounded on \(\Omega\) by \eqref{eq:nicebound}. Equipped with these results, we now establish our sum of squares theorem for \(C_{\mathrm{loc}}^{k,\alpha}(\Omega)\) functions.

\begin{proof}[Proof of \Cref{thm:locdecomp}]
Given a non-negative function \(f\in C^{k,\alpha}_{\mathrm{loc}}(\Omega)\) defined on an open set \(\Omega\) in \(\mathbb{R}^n\), we can now show that \(f\) is decomposable as a sum of half-regular squares. To do this we make a selection for \(p\) that is similar to that made by Bony in \cite{Bony},
\[
    p(x)=\bigg(1+\sum_{|\beta|\leq k}|\partial^\beta f(x)|+\sum_{|\beta|\leq k}[\partial^\beta f]_\alpha(x)\bigg)^{-1}.
\]
From this choice of \(p\), it follows that \(|\partial^\beta f(x)|\leq p(x)^{-1}\) and \([\partial^\beta f]_\alpha(x)\leq p(x)^{-1}\) for every \(x\in\Omega\) and \(\beta\) such that \(|\beta|\leq k\). Note that \(p\) is both positive and continuous in \(\Omega\) since \(f\in C^{k,\alpha}_{\mathrm{loc}}(\Omega)\) by assumption, so it satisfies the hypotheses of \Cref{lem:extbyzero} and we can find a function \(\varphi\in C^\infty(\Omega)\) that is positive on \(\Omega\) and which satisfies \eqref{eq:boundarydecay}.

Next we show that \(\varphi f\in C^{k,\alpha}(\mathbb{R}^n)\). By the form of the sub-product rule given in \Cref{lem:sharpsubprod}, we have the pointwise estimate
\[
    [\partial^\beta(\varphi f)]_\alpha(x)\leq \sum_{\gamma\leq\beta }\binom{\beta}{\gamma}([\partial^{\beta-\gamma} f]_\alpha(x)|\partial^\gamma\varphi(x)|+[\partial^\gamma \varphi]_\alpha(x)|\partial^{\beta-\gamma}f(x)|).
\]
The norms and semi-norms containing \(f\) above are all bounded above by \(p(x)^{-1}\), so we have
\[
    [\partial^\beta(\varphi f)]_\alpha(x)\leq \sum_{\gamma\leq\beta}\binom{\beta}{\gamma}\bigg(\frac{|\partial^\gamma\varphi(x)|}{p(x)}+\frac{[\partial^\gamma \varphi]_\alpha(x)}{p(x)}\bigg).
\]
It follows that \([\partial^\beta\varphi f]_\alpha(x)\) is bounded on \(\Omega\) and decays to zero at the boundary whenever \(|\beta|=k\) thanks to \eqref{eq:nicebound} and \eqref{eq:alphadecay}, meaning that \([\partial^\beta(\varphi f)]_{\alpha,\Omega}<\infty\). Extending \(\varphi f\) by zero outside of \(\Omega\), we obtain a function (which we also call \(\varphi f\)) that belongs to \(C^{k,\alpha}(\mathbb{R}^n)\). By \Cref{thm:c3main} it follows that we can write
\[
    \varphi f=\sum_{j=1}^{m_n} g_j^2
\] 
for \(g_1,\dots,g_{m_n}\in C^\frac{k+\alpha}{2}(\mathbb{R}^n)\). Since \(\varphi\) is smooth and bounded below on any compact subset \(K\) of \(\Omega\), we find that on such sets we have for each \(j=1,\dots,m_n\) that
\[
    \tilde{g}_j=\frac{g_j}{\sqrt{\varphi}}\in C^\frac{k+\alpha}{2}(K),
\]
and since \(K\) is any compact subset we see in turn that \(\tilde{g}_j\in C^\frac{k+\alpha}{2}_{\mathrm{loc}}(\Omega)\). Since we are able to write
\[
    f=\sum_{j=1}^m \tilde{g_j}^2,
\]
we see that \(f\) is a sum of half-regular squares on compact subsets of \(\Omega\), as claimed.
\end{proof}

\section{Bounds on Decomposition Size}\label{sec:sizebounds}

The number of squares needed in a sum of squares decomposition, which we denote by \(m_n\), is finite for non-negative \(C^{k,\alpha}(\mathbb{R}^n)\) functions when \(k\leq 3\) thanks to Theorem \ref{thm:c3main}. Indeed, we can find an upper bound for \(m_n\) that depends only on the dimension \(n\). If \(n=1\), Bony shows in \cite{Bony2} that \(m_1=2\) is optimal. In higher dimensions, our methods require the use of more squares, owing to the inductive nature of our argument and the growth of the bounded overlap constant \(N_n\) in \Cref{thm:party} as \(n\) becomes large. In this section, we give upper bounds on \(m_n\).

Recall that in the proof of \Cref{thm:c3main}, we could locally decompose \(\psi_j^2 f\in C^{k,\alpha}(\mathbb{R}^n)\) into a sum of squares of at most \(m_{n-1}+1\) functions, and using item \textit{(5)} of \Cref{thm:party} we were able to recombine to write the infinite sum
\[
    \sum_{j=1}^\infty\psi_j^2
\]
as a finite sum of squares of functions. As such, we can write \(f\) as a sum of squares of at most \(m_n\leq N_n(1+m_{n-1})\) functions, where \(m_1=2\) thanks to the work of Bony in \cite{Bony2}. Solving this recursion is a straightforward task which yields the sequence of estimates
\[
    m_n\leq 2N_n^{n-1}+\frac{N_n^n-N_n}{N_n-1}.
\]
Using the bound \(N_n\leq 15^n\) obtained in the proof of \Cref{thm:party}, which seems to be far from optimal since our construction uses balls instead of dyadic cubes, we obtain the crude estimate
\[
    m_n\leq 2\cdot15^{n^2-n}+\frac{15^{n^2}-15^n}{15^n-1}.
\]

This grows very fast with \(n\), which is unsurprising since there is little evidence to suggest our construction is optimal. A separate approach may require many fewer squares, and little is known of examples requiring a maximal number of squares in general. While two squares are required in one dimension, the number of squares required for general \(n\) is unknown.

It may be fruitful to investigate the minimal number of squares required to decompose an arbitrary function, rather than focusing exclusively on upper bounds. One possible way to do this could be finding polynomials which can be decomposed into no fewer than a given number of squares; such a technique is employed in the next chapter to produce counterexamples to our main theorem when \(k\) becomes too large, so it is plausible that the behaviour of polynomial sums of squares can afford more information about the sizes of our decompositions.

Though we do not pursue them in this work, there are also  two potential approaches which we can identify that may be able to improve our upper bounds for \(m_n\). First, by using better sets in \Cref{thm:party} (e.g. cubes instead of balls) it is possible to reduce the constant \(N_n\), though in general this number should still increase exponentially with \(n\) since the bounded overlap constant for cubes in \(\mathbb{R}^n\) is roughly \(2^n\).

Second, we note that the coefficient \(m_{n-1}+1\) only enters the recursion for \(m_n\) in the case that \(f\) is very small near a nonzero local minimum; otherwise a single square suffices locally. For functions which only have a small number of nonzero local minima, it seems that we can reduce the need for \(m_{n-1}+1\) functions in a decomposition to a relatively small number of cases thereby improving our bound on \(m_n\).

In any case, there is substantial room for improvement concerning the number of squares need in a regularity preserving decomposition, and the problem of finding an optimal number seems likely to involve some interesting mathematics; our upper bound arguments involve notions from Euclidean geometry and measure theory, while the lower bounds (at least conceptually) use techniques from algebra and the convex geometry induced by polynomial sums of squares which we explore in the next chapter. See \Cref{sec:applics} for further discussion of this problem.

\chapter{Non-Decomposable Functions}\label{chap:poly}

In the previous chapter, we found that non-negative functions in \(C^{k,\alpha}(\mathbb{R}^n)\) can be decomposed into sums of half-regular squares when \(k\leq 3\). Now we show that if \(k\geq 4\) and \(\alpha>0\), then there exist non-negative functions in \(C^{k,\alpha}(\mathbb{R}^n)\) which cannot be decomposed into sums of half-regular squares. In doing so we show that \Cref{thm:c3main} is essentially optimal, since the decomposition it provides can fail if we do not restrict to the H\"older spaces \(C^{2,\alpha}(\mathbb{R}^n)\) and \(C^{3,\alpha}(\mathbb{R}^n)\).

Bony gives an example of a non-decomposable function in \cite{Bony}, and several more appear in \cite{SOS_I}. Each of these is constructed using non-negative polynomials which cannot be written as sums of squares of polynomials. The first example of such a polynomial to appear in the literature is the Motzkin polynomial from \cite{Motzkin}, which we examine in detail in \Cref{sec:nd},
\[
    M(x,y)=x^4y^2+x^2y^4-3x^2y^2+1.
\]

Non-negative polynomials of even degree \(d\) over \(n\) variables which are not sums of squares are known to exist in many settings, thanks to the work of Hilbert in \cite{Hilbert}, but Hilbert's existence result is not constructive. Some examples of non-decomposable polynomials are exhibited by Reznik \cite{reznik_extpsd}, Robinson \cite{robinson}, and Choi \& Lam \cite{CL2}, but cumulatively these examples only cover a handful of cases in which \(n\) and \(d\) are both small. Our contribution is to devise a procedure for constructing examples in the cases not covered in the literature; see \Cref{sec:gen}.

This chapter serves two purposes. First, we clarify the connection between polynomials which are not sums of squares and \(C^{k,\alpha}(\mathbb{R}^n)\) functions which cannot be decomposed into half regular squares; indeed, we show that the former functions are examples of the latter type. Second, we present new examples of polynomials which are not sums of squares, and a technique for constructing them. This gives many functions which violate \Cref{thm:c3main} when \(k\geq 4\).

\section{Non-Decomposable Polynomials}\label{sec:nd}

Hilbert showed in \cite{Hilbert} and \cite{Hilbert2} that for every even degree \(d\geq 4\) and in every dimension \(n\geq 2\), there exist non-negative polynomials of degree \(d\) over \(n\) variables which cannot be written as sums of squares of polynomials, except when \(d=4\) and \(n=2\), in which case non-negative polynomials can be written as sums of three squares.

\medskip
\begin{center}
Summary of Hilbert's Results in \cite{Hilbert} and \cite{Hilbert2}.\\ Is every non-negative polynomial of degree \(d\) over \(n\) variables a sum of squares of polynomials?\\
\medskip
\begin{tabular}{|c|c|c|c|c|}
\hline
\diagbox[height=2em,width=3em]{\(n\)}{\(d\)} & \parbox[c]{2cm}{\hfil 2} & \parbox[c]{2cm}{\hfil 4}& \parbox[c]{2cm}{\hfil \(\geq 6\)}\\
\hline
1 & Yes&Yes&Yes\\
2 & Yes&Yes&No\\
\(\geq 3\) & Yes&No&No\\
\hline
\end{tabular}
\end{center}

Later, Artin confirmed in \cite{Artin} that all non-negative polynomials can be decomposed into sums of squares of rational functions, resolving Hibert's \(17^{\mathrm{th}}\) problem. Non-decomposable polynomials are hardly an anomaly; Blekherman proved in \cite{Blekherman} that for each fixed \(n\), degree \(d\) polynomials that are not sums of squares become increasingly common as \(d\) becomes large. Indeed, this frequency is quantified in \cite[Theorems 2.1 \& 2.2]{Blekherman}. In \Cref{sec:gen} we find that it is easier to produce examples for which \(d\) is large, affirming the observation that non-decomposable polynomials are abundant at large degrees.

Our interest in these polynomials is rooted in the following useful fact, which we formalize in \Cref{lem:contpoly} below: if a degree \(d\) polynomial \(P\) is not a sum of squares of polynomials, then it cannot be written as a sum of squares of functions in \(C^\frac{d+\alpha}{2}(\mathbb{R}^n)\) for any \(\alpha>0\). Thus, by finding such a polynomial \(P\) we immediately obtain a \(C^{d,\alpha}(\mathbb{R}^n)\) function which is not a sum of half-regular squares. We prove this using the following technical result, which is modelled after \cite[Lemma 1.3]{Bony} and \cite[Lemma 5.2]{SOS_I}.

\begin{lem}\label{lem:contpoly}
Let \(P\) be a non-negative polynomial of even degree \(k\) which is not a sum of squares of polynomials, let \(\alpha>0\), and let \(\Omega\subset\mathbb{R}^n\) be compact. For every \(m\in\mathbb{N}\) and \(N>0\), there exists a number \(\delta>0\) such that 
\[
    \qquad\sup_{\Omega}\bigg|\sum_{j=1}^mg_j^2-P\bigg|<\delta\qquad\implies\qquad\sum_{j=1}^m\|g_j\|_{C_b^\frac{k+\alpha}{2}(\Omega)}>N.
\]
\end{lem}

\begin{proof}
Fix \(N\) and assume to the contrary that for every \(\delta>0\) there exist \(g_{1,\delta},\dots,g_{m,\delta}\) such that 
\[
    \sum_{j=1}^m\|g_{j,\delta}\|_{C_b^\frac{k+\alpha}{2}(\Omega)}\leq N\qquad\mathrm{and}\qquad\sup_{\Omega}\bigg|\sum_{j=1}^mg_{j,\delta}^2-P\bigg|<\delta.
\]
For \(\ell\in\mathbb{N}\) choose functions \(g_{1,\ell},\dots,g_{m,\ell}\) that satisfy the inequalities above with \(\delta=\frac{1}{\ell}\), and for fixed \(j\) consider the sequence \(\{g_{j,\ell}\}_{\ell}\). This sequence is uniformly bounded in the norm of the half-regular H\"older space since 
\[
    \|g_{j,\ell}\|_{C_b^\frac{k+\alpha}{2}(\Omega )}\leq N
\]
for every \(\ell\). Recall that the embedding of \(C_b^\frac{k+\alpha}{2}(\Omega)\) into \( C_b^\frac{k}{2}(\Omega)\) is compact by \Cref{lem:cptmb}, so there exists a subsequence of \(\{g_{j,\ell}\}_{\ell}\) that converges uniformly to a function \(g_j\) in the target space. Passing to this subsequence, we see that
\[
    \sup_\Omega\bigg|\sum_{j=1}^mg_j^2-P\bigg|=\lim_{\ell\rightarrow\infty}\sup_\Omega\bigg|\sum_{j=1}^mg_{j,\ell}^2-P\bigg|\leq\lim_{\ell\rightarrow\infty}\frac{1}{\ell}=0,
\]
meaning that \(P\) is a sum of squares in \(C_b^\frac{k}{2}(\Omega)\).

Now we argue that each \(g_j\) is in fact a polynomial. If the highest-order term in the Taylor expansion of any \(g_j\) has degree exceeding \(\frac{k}{2}\), then the sum of squares \(P=g_1^2+\cdots+g_m^2\) would include a monomial term whose degree exceeds \(k\), since the coefficient on the highest order term is a sum of squares of real numbers, hence strictly positive. This contradicts the fact that \(P\) has degree \(k\), meaning that each \(g_j\) is a polynomial of degree at most \(\frac{k}{2}\) and \(P\) is a sum of squares of polynomials. This contradicts our initial assumption on \(P\).
\end{proof}

Equipped with this result, we draw a useful conclusion about non-decomposable polynomials.

\begin{cor}\label{cor:notsos}
Let \(P\) be a non-negative polynomial of even degree \(k\) which is not a sum of squares of polynomials. Then \(P\) is not a sum of squares of functions in \(C^\frac{k+\alpha}{2}(\mathbb{R}^n)\) for \(\alpha>0\).
\end{cor}

\begin{proof}
Assume that there exist functions \(g_1,\dots,g_m\in C^\frac{k+\alpha}{2}(\mathbb{R}^n)\) such that
\[
    P=\sum_{j=1}^m g_j^2.
\]
Then restricting to any compact set \(\Omega\subset\mathbb{R}^n\) we necessarily have for each \(j\) that \(g_j\in C_b^{\frac{k+\alpha}{2}}(\Omega)\). On the other hand, \Cref{lem:contpoly} implies that for any \(N\in\mathbb{R}\),
\[
    \sum_{j=1}^m\|g_j\|_{C_b^\frac{k+\alpha}{2}(\Omega)}>N.
\]
It follows that for some \(j\) one of the norms above is infinite. Since \(g_j\) and its derivatives are bounded, we conclude that for some \(\beta\) whose order is the integer part of \(\frac{k+\alpha}{2}\), we have
\[
    [\partial^\beta g_j]_{\frac{k+\alpha}{2}-|\beta|,\Omega}=\infty.
\]
It follows that a derivative of \(g_j\) is not \(\alpha\)-H\"older continuous on \(\Omega\), meaning that \(g_j\not\in C^\frac{k+\alpha}{2}(\mathbb{R}^n)\). This contradicts our initial assumption, and the claim follows.
\end{proof}

Since the Motzkin polynomial is not a sum of squares of polynomials, as we verify in the next section, we find now that \(C^{6,\alpha}(\mathbb{R}^2)\) contains a non-decomposable function.

\begin{cor}
For any \(\alpha>0\), the Motzkin polynomial \(M(x,y)=x^4y^2+x^2y^2-3x^2y^2+1\) is a non-negative \(C^{6,\alpha}(\mathbb{R}^2)\) function which is not a sum of squares in \(C^{3,\frac{\alpha}{2}}(\mathbb{R}^2)\).  
\end{cor}

It is important to observe that the only properties of the Motzkin polynomial we used for this corollary were non-negativity, and the fact that \(M\) is not a sum of squares. Polynomials on \(\mathbb{R}^n\) with these properties abound at high degrees when \(n\geq 2\), thanks to the results of Hilbert \cite{Hilbert} and Blekherman \cite{Blekherman}. Despite providing no examples, Hilbert's argument is nevertheless sufficient for us to draw a striking conclusion from the sequence of results obtained above. The following sweeping result seems to have been overlooked in many earlier works on sums of squares, save for the cases \(k=4\) and \(k=6\) which are discussed in \cite{Bony} and \cite{SOS_I}.

\begin{cor}\label{cor:existence}
Let \(k\geq 4\) and \(n\geq 2\) and assume that one of these inequalities is strict. For any \(\alpha>0\) there exist non-negative functions in \(C^{k,\alpha}(\mathbb{R}^n)\) which are not sums of half-regular squares.
\end{cor}

\begin{proof}
If \(k\) is even this is immediate from \Cref{cor:notsos} and \cite{Hilbert}. If \(k\) is odd, it suffices to repeat the argument above with a non-negative and non-decomposable polynomial \(P\) of degree \(k-1\), recalling our definition of half-regular H\"older spaces in \eqref{eq:halfreg}. 
\end{proof}

This existence result does the job of verifying that \Cref{thm:c3main} does not hold for \(k\geq 4\), but it is unsatisfying since it gives no concrete examples. For the remainder of this chapter, we work to obtain explicit polynomials which are not sums of squares for each admissible pairing of \(n\) and \(k\) in \Cref{cor:existence}. Through \Cref{cor:notsos}, we can readily conclude that such examples are not decomposable as sums of half-regular squares.

\section{Explicit Polynomial Examples}\label{sec:examples}

To prove that the Motzkin polynomial is not a sum of squares of polynomials, and to construct explicit examples, we introduce some basic concepts and techniques from convex geometry.

\begin{defn}
Let \(P(x)=c_1 x^{q_1}+\cdots+c_mx^{q_m}\) be a polynomial on \(\mathbb{R}^n\) for constants \(c_j\neq0\) and \(q_j\in\mathbb{N}_0^n\). The Newton polytope \(C_P\) of \(P\) is the convex hull in \(\mathbb{R}^n\) generated by the points \(q_1,\dots,q_m\). That is, \(C_P\) is the smallest convex set in \(\mathbb{R}^n\) which contains each \(\alpha_j\). 
\end{defn}

For example, the Newton polytope of the Motzkin polynomial is the closed triangle \(C_M\) in \(\mathbb{R}^2\) with corners at \((0,0)\), \((2,4)\) and \((4,2)\). Frequently, we employ a more useful realization of \(C_P\) as the set of all convex combinations of the lattice points \(q_1,\dots,q_m\). That is,
\[
    C_P=\bigg\{\sum_{j=1}^m\lambda_jq_j\;:\;\sum_{j=1}^m\lambda_j=1\;\mathrm{and}\;0\leq\lambda_j\leq 1\bigg\}.
\]
Observe that if \(P(0)\neq 0\) then \(C_P\) has a vertex fixed at the origin, and if \(m=n+1\), then the weights \(\lambda_1,\dots,\lambda_{n+1}\) turn out to be uniquely determined. We verify this fact momentarily in \Cref{thm:posweights} before employing it to construct polynomials which are not sums of squares.

Our next result will also be useful in our construction, and it follows as a straightforward consequence of the weighted arithmetic-mean geometric-mean inequality (Theorem \ref{thm:amgm}). Simply put, the following lemma states that if a multi-index \(p\) can be realized as a weighted average of multi-indices \(q_1,\dots,q_m\), each having even entries, then we can construct a non-negative polynomial which contains monomials associated to each \(q_j\) and \(p\). This is a generalized variant of the result used by Motzkin in \cite{Motzkin} to verify non-negativity of the polynomial \(M\).

\begin{lem}\label{lem:amgm}
Let \(q_1,\dots,q_m\in 2\mathbb{N}_0^n\), let \(\lambda_1,\dots,\lambda_m\) be such that \(0\leq \lambda_j\leq 1\) for each \(j\) and \(\lambda_1+\cdots+\lambda_m=1\), and assume that \(p=\lambda_1q_1+\cdots+q_m\alpha_m\in \mathbb{N}_0^n\). Then for every \(x\in\mathbb{R}^n\),
\[
    x^p\leq \sum_{j=1}^m \lambda_jx^{q_j}.
\]
Moreover, equality holds above when \(x^{q_1}=\cdots=x^{q_m}\) and in particular when \(x=(1,\dots,1)\).
\end{lem}

\begin{proof}
Let \(p\) be as above and note that \(x^\beta=x^{\lambda_1q_1}\cdots x^{\lambda_mq_m}\) for any \(x\in\mathbb{R}^n\). Given a point \(x=(x_1,\dots,x_n)\in\mathbb{R}^n\), we set \(\Tilde{x}=(|x_1|,\dots,|x_n|)\) so that by Theorem \ref{thm:amgm},
\[
    x^p= \prod_{j=1}^mx^{\lambda_jq_j}\leq \prod_{j=1}^m(\Tilde{x}^{q_j})^{\lambda_j}\leq \sum_{j=1}^m\lambda_j\Tilde{x}^{q_j}=\sum_{j=1}^m\lambda_jx^{q_j}.
\]
From the equality case of the AMGM inequality, we see that equality holds above when \(x^p=\Tilde{x}^p\) and \(x^{q_1}=\cdots=x^{q_m}\). The first inequality and the last identity above hold since \(\Tilde{x}^{q_j}=x^{q_j}\geq0\) whenever \(q_j\in 2\mathbb{N}_0^n\), and since \(x\in\mathbb{R}^n\) was arbitrary we are finished.
\end{proof}

Using this result, it is straightforward to verify non-negativity of the Motzkin polynomial. The monomial terms of \(M\) correspond to the multi-indices \(q_1=(0,0)\), \(q_2=(4,2)\), \(q_3=(2,4)\), and \(p=(2,2)\). Since we can write \(p=\frac{1}{3}(q_1+q_2+q_3)\), it follows from \Cref{lem:amgm} that for every \(x\in\mathbb{R}^2\) we have \(3x^p\leq x^{q_1}+x^{q_2}+x^{q_3}\). In standard notation this reads \(3x^2y^2\leq 1+x^4y^2+x^2y^4\) or equivalently, \(M(x,y)\geq0\).

In the course of constructing new counterexamples like the Motzkin polynomial, the preceding result allows us to select coefficients which ensure that the polynomial we extract from a well-chosen Newton polytope is non-negative everywhere. In \cite{reznik_extpsd}, Reznik makes an important connection between Newton polytopes and polynomial sums of squares which we summarize in the following lemma. The proof of this result can be found in \cite[Theorem 1]{reznik_extpsd}, and it employs some basic convex geometry which we avoid repeating.

\begin{lem}[Reznik]\label{lem:rezlem}
Let \(P\) be a polynomial with Newton polytope \(C_P\). If there exist polynomials \(g_1,\dots,g_m\) such that \(P=g_1^2+\cdots+g_m^2\) then 
\[
    \bigcup_{j=1}^m C_{g_j}\subseteq \frac{1}{2}C_P.
\]
Here, \(\frac{1}{2}C_{P}\) denotes the dilated polytope \(\{\frac{1}{2}x:x\in C_{P}\}\).
\end{lem}

The use of this connection between Newton polytopes and polynomial sums of squares is leveraged by Reznik in \cite{reznik_extpsd}, and later the connection is also used extensively by Choi, Lam, and Reznik in \cite{CLR} to study the Pythagoras number of a ring. This is the smallest number of squares needed in the sum of squares decomposition of any decomposable element, and the analog in our study of decompositions of H\"older functions is the number \(m_n\) estimated in \Cref{sec:sizebounds}.

Employing the useful lemma above, we can argue in a similar fashion to Reznik to show that the Motzkin polynomial is not a sum of squares. This argument is not new, and variants appear in both \cite{reznik_extpsd} and \cite{Powers}, however it is important to understand since it is a foundation of our more general construction.

\begin{cor}
The polynomial \(M(x,y)=x^2y^4+x^4y^2-3x^2y^2+1\) is not a sum of squares.
\end{cor}

\begin{proof}
Assume toward a contradiction that \(M=g_1^2+\cdots+g_m^2\) and note that \(C_P\) is the triangular region in \(\mathbb{R}^2\) with vertices at \((0,0)\), \((4,2)\) and \((2,4)\). By \Cref{lem:rezlem}, for each \(j\) the polytope \(C_{g_j}\) must be a subset of the triangle with vertices at \((0,0)\), \((2,1)\) and \((1,2)\). Since the only lattice points which belong to to this closed triangle are the vertices themselves and \((1,1)\), the square of \(g_j\) must take the form \(g_j^2=(c_1+c_2xy+c_3x^2y+c_4xy^2)^2\), and expanding gives
\[
    g_j^2=c_1^2+c_2^2x^2y^2+c_3^2x^4y^2+c_4^2x^2y^4+2xy(c_1c_2+c_1c_3x+c_1c_4y+c_2c_3x^2y+c_2c_4xy^2+c_3c_4x^2y^2).
\]
The monomial \(x^2y^2\) only appears above in the term \(c_2^2x^2y^2\). It follows that the coefficient of \(x^2y^2\) in \(g_1^2+\cdots+g_m^2\) is a sum of squares, hence non-negative. This contradicts the fact that the coefficient of \(x^2y^2\) in \(M(x,y)\) is negative, and we conclude that \(M\) is not a sum of squares.
\end{proof}

Generalizing the idea of the preceding proof, we establish the following criterion which allows us to determine whether a given Newton polytope may correspond to a polynomial which is not a sum of squares.

\begin{cor}\label{cor:hullcond}
Let \(P\) be a polynomial which contains a monomial \(kx^p\) for which \(k<0\)
and such that \(p\) is not the sum of any two distinct points in \(\frac{1}{2}C_P\). Then \(P\) is not a sum of squares. 
\end{cor}

The conditions above are not necessary for a polynomial not to be a sum of squares. Observe that if a polynomial is not a sum of squares then neither are any of its translations, so
\[
    P(x,y)=M(x+1,y+1)=x^2 y^4+4x^2 y^3+3x^2 y^2+x^4 y^2+2x^4 y+x^4-2x^2 y-2x^2+1
\]
is not a sum of squares. However it is easily verified by direct calculation that every integer lattice point belonging to \(C_P\) can be written as the average of two distinct lattice points in \(C_P\). It follows that the polynomials which satisfy \Cref{cor:hullcond} do not exhaust the collection of non-decomposable polynomials.

Nevertheless, equipped with \Cref{cor:hullcond} we can construct many polynomials which are not sums of squares. Using \Cref{lem:amgm}, we ensure that these are non-negative and only vanish at a single point. To illustrate how the latter properties are achieved, we observe by \Cref{lem:amgm} that \(M(x,y)\) is non-negative and \(M(x,y)=0\) when \(x^2=y^2=1\). Further, the lemma indicates that \(x^2y^4+1-2xy^2\geq 0\) with equality when \(x=1\) and \(y^2=1\), while \(x^4y^2+1-2x^2y\geq0\) with equality when \(x^2=1\) and \(y=1\). Adding these polynomials to \(M(x,y)\), we see that the following is a non-negative polynomial with a single zero at \((1,1)\),
\[
    P(x,y)=2x^2 y^4+2x^4 y^2-3x^2 y^2-2x y^2-2x^2 y+3
\]
Moreover, \(C_P=C_M\) and \(P\) has a negative coefficient on the critical monomial \(x^2y^2\), so \(P\) is not a sum of squares of polynomials. A translation of \(P\) allows us to place the zero at the origin.

Convex hulls in \(\mathbb{R}^n\) with the special property outlined in \Cref{cor:hullcond} are the Newton polytopes of polynomials which are not sums of squares. In the next section we show how such polytopes can be constructed systematically, and using Lemma \ref{lem:amgm} we extract non-negative polynomials from these sets. Most often, we work with polytopes in \(\mathbb{R}^n\) which have \(n+1\) vertices, since doing so involves especially simple calculations as we outline in the following theorem.

\begin{thm}\label{thm:posweights}
Let \(C\) be a polytope in \(\mathbb{R}^n\) with \(n+1\) lattice point vertices, \(q_1,\dots,q_{n+1}\). If \(p\) is a lattice point in the interior of \(C\), then there exist \(\lambda_1,\dots,\lambda_{n+1}>0\) such that
\begin{equation}\label{eq:conv}
    p=\sum_{j=1}^{n+1}\lambda_jq_j\qquad\textrm{and}\qquad\sum_{j=1}^{n+1}\lambda_j=1.
\end{equation}
Moreover, if \(q_{n+1}\) is fixed at the origin and the vectors \(q_1,\dots,q_n\) are linearly independent then \(\lambda_1,\dots,\lambda_{n+1}\) are uniquely determined and can be found by setting \(Q=[q_1 \cdots q_n]\) and computing
\begin{equation}\label{eq:solving}
    [\lambda_1\cdots\lambda_n]=(Q^{-1}p)^T\qquad\textrm{and}\qquad\lambda_{n+1}=1-\sum_{j=1}^n\lambda_j.
\end{equation}
\end{thm}

\begin{proof}
Since \(C\) is a convex hull which contains \(p\), the identities of \eqref{eq:conv} hold for some set of weights \(\lambda_1,\dots,\lambda_{n+1}\geq 0\). If \(\lambda_j=0\) for some \(j\) then \(p\) must be a convex combination of at most \(n\) points, meaning it belongs to one of the faces of \(C\) and not the interior; this contradicts our assumption that \(p\) is an interior point. 

If \(q_{n+1}=0\) and \(q_{1},\dots,q_n\) are linearly independent then \(C\) has non-empty interior; simply consider the point
\[
    p=\frac{1}{n}\sum_{j=1}^n q_j,
\]
or indeed any weighted average of the vertices \(q_1,\dots,q_n\). The identities of \eqref{eq:solving} are thus obtained by solving the linear system in \eqref{eq:conv}.
\end{proof}

The uniqueness of solutions to \eqref{eq:solving} gives us a simple criterion for determining whether an integer lattice point \(r\in\mathbb{R}^n\) belongs to a given polytope \(C\) with \(n\) linearly independent vertices. This result allows us to efficiently search for the polytopes outlined in \Cref{cor:hullcond}.

\begin{cor}
Let \(C\) be a convex hull in \(\mathbb{R}^n\) with one vertex at the origin and the remaining vertices at linearly independent points \(q_1,\dots,q_n\). Set \(Q=[q_1\;\cdots\;q_n]\), and given \(r\in\mathbb{R}^n\) define \([\lambda_1\;\cdots\;\lambda_n]^T=Q^{-1}r\) and \(\lambda_{n+1}=1-(\lambda_1+\cdots+\lambda_n)\). Then exactly one of the following holds,
\begin{itemize}
    \item[(1)] If \(\lambda_1,\dots,\lambda_{n+1}>0\) then \(r\) is interior to \(C\),
    \item[(2)] If \(\lambda_1,\dots\lambda_{n+1}\geq0\) and \(\lambda_j=0\) for some \(j\) then \(r\) is a boundary point of \(C\),
    \item[(3)] If \(\lambda_j<0\) for some \(j\) then \(r\) is exterior to \(C\).
\end{itemize}
\end{cor}

The result above reduces determination of whether a point lies in \(C\) to a simple matrix-vector multiplication. This makes it straightforward to check whether \(C\) has the property outlined in \Cref{cor:hullcond}, by searching for a lattice point \(p\in C\) which is not the sum of two distinct lattice points in \(\frac{1}{2}C\). If no such point is found, we simply try another polytope with \(n+1\) vertices.


Finally, we arrive at the critical construction. Given a polytope \(C\) of \(n+1\) vertices, one fixed at the origin and the rest linearly independent, and an interior point \(p\in C\) which is not the sum of two distinct lattice points in \(\frac{1}{2}C\), we use \Cref{thm:posweights} to write \(p=\lambda_1q_1+\cdots+\lambda_nq_n\), where \(\lambda_j>0\) for each \(j\) and \(\lambda_1+\cdots+\lambda_n<1\). It follows from \Cref{lem:amgm} that the following polynomial is non-negative and not a sum of squares,
\begin{equation}\label{eq:shortone}
    P(x)=\sum_{j=1}^n\lambda_jx^{2q_j}+1-\sum_{j=1}^n\lambda_j-x^{2p}.
\end{equation}
To recover integer coefficients, it suffices to multiply this polynomial by \(|\mathrm{det}\;Q|\). Additionally, as we modified the Motzkin polynomial to have a single zero, we can add lower-order terms as follows to construct a polynomial that only has one zero,
\begin{equation}\label{eq:poly}
    Q(x)=\sum_{j=1}^n\lambda_jx^{2q_j}+1-\sum_{j=1}^n\lambda_j-x^{2p}+c\sum_{j=1}^n(x^{2q_j}+1-2x^{q_j}).
\end{equation}
It is easy to see that \(C_Q=C_P\) for any \(c>0\), so \(Q\) is not a sum of squares. Moreover, we have \(Q(1,\dots,1)=0\) and \(Q\) is strictly positive elsewhere by \Cref{lem:amgm}.

For example, one can verify that the Newton polytope \(C\) with vertices at \((0,0)\), \((6,2)\), \((2,4)\) contains the point \((2,2)\) which is not the sum of two distinct points in \(\frac{1}{2}C\), and undergoing the construction above gives the non-negative polynomial
\[
    P(x,y)=x^6y^2+2x^2y^4-5x^2y^2+2,
\]
which is not a sum of squares of polynomials. Moreover, the following modified form is also not a sum of squares, and it is only zero at at the point \((1,1)\),
\[
    Q(x,y)=2 x^6 y^2+3 x^2 y^4-5 x^2 y^2-2 x^3 y-2 x y^2+4.
\]
In summary, to obtain a non-negative polynomial which is not a sum of squares, it suffices to find a polytope \(C\) in \(\mathbb{R}^n\) with \(n+1\) vertices at non-negative integer lattice points in \(\mathbb{R}^n\), one fixed at the origin, which satisfies the hypothesis of \Cref{cor:hullcond}. Given such a polytope, the desired polynomial is provided by \eqref{eq:poly}.

Momentarily we develop a systemic way of constructing these polytopes. First, we present the following algorithm for finding polynomials of even degree \(d\) on \(\mathbb{R}^n\) that are not sums of squares. For shorthand, we let \(v\) denote the column vector in \(\mathbb{R}^n\) in which each entry is one. For a vector \(w\) we write \(w\geq0\) to indicate that \(w\) has non-negative entries, and if \(z\) is another vector we write \(w\geq z\) if each entry in \(w\) exceeds the corresponding entry in \(z\).

\medskip

\noindent\textit{Algorithm For Finding Non-Negative, Non-SOS Polynomials (Direct Search)}
\begin{itemize}
    \item[(1)] Choose \(n\) linearly independent non-negative lattice points in \(q_1,\dots,q_n\in\mathbb{R}^n\) such that \(q_{j,1}+\cdots+q_{j,n}\leq \frac{1}{2}d\), with equality for some \(j\). Set \(Q=[q_1\cdots q_n]\) and compute \(Q^{-1}\).
    \item[(2)] Fix a lattice point \(p\) in the box with corners at the origin and \(2(\max_jq_{j,1},\dots,\max_{j}q_{j,n})\).
    \item[(3)] If \(0\leq Q^{-1}p\) and \(1<v^TQ^{-1}p\leq 2\) then \(p\) is in the candidate hull -- proceed to step (3). If either inequality fails, return to step (2) and choose a new point.
    \item[(4)] Fix a lattice point \(t\) in the box with corners at the origin and \((\max_jq_{j,1},\dots,\max_{j}q_{j,n})\). 
    \begin{itemize}
        \item[(4a)] If \(t\neq p\), \(0\leq Q^{-1}t\leq Q^{-1}p\) and \(v^TQ^{-1}p-1\leq v^TQ^{-1}t\leq 1\), then \(p\) is the average of two distinct points in the candidate hull. Return to step (2) and choose a new point.
        \item[(4b)] If no such \(t\) is found after all points are checked, stop; \(p\) has the desired property.
    \end{itemize}
    \item[(7)] If no viable point \(p\) is found after all possible choices in step (2), return to step (1).
\end{itemize}

\medskip

A successful search gives points \(q_1,\dots,q_n\) and \(p\) which can be used to compute \(\lambda_1,\dots,\lambda_{n+1}\) from \eqref{eq:solving}, and in turn these give the polynomial \eqref{eq:poly} which is non-negative, not a sum of squares, and it is only zero at a single point. Re-scaling by \(|\mathrm{det}\;Q|\) also gives integer coefficients. Using this algorithm, we are able to find many non-negative polynomials which cannot be written as sums of squares. We summarize these examples in the table on the next page, noting that most could not be found in existing literature (we indicate those that were). For brevity, we do not include the terms which enforce a single zero, since these make the expressions quite long.

\newpage
\renewcommand{\arraystretch}{1.3}
\begin{center}
Non-Negative Polynomials Which Are Not Sums of Squares (\(n\leq 4\), \(d\leq 20\)).

\begin{tabular}{|c|c|c|}
    \hline
    Dimension \(n\)  & Degree \(d\) & Polynomial\\
    \hline
    2 & 6 & \(x^4y^2+x^2y^4-3x^2y^2+1\)\;\;(\cite{Motzkin})\\
    2 & 8 & \(x^6y^2+2x^2y^4-5x^2y^2+2\)\\
    2 & 10 & \(x^4y^6+x^2y^6-3x^2y^4+1\)\\
    2 & 12 & \(2x^8y^4+13y^8-16xy^7+1\)\\
    2 & 14 & \(x^4y^{10}+x^2y^2-3x^2y^4+1\)\\
    2 & 16 & \(x^6y^{10}+y^2-3x^2y^4+1\)\\
    2 & 18 & \(x^{14}y^4+x^4y^2-3x^6y^2+1\)\\
    2 & 20 & \(3x^{10}y^{10}+20y^6-30xy^5+7\)\\
    \hline
    3 & 4 & \(x^2y^2+y^2z^2+x^2z^2-4xyz+1\)\;\;(\cite{CL2})\\
    3 & 6 & \(x^4z^2+4x^2z^4+3y^4z^2-12xyz^2+4\)\\
    3 & 8 & \(x^2y^2z^4+x^2y^4z^2+x^4y^2z^2-4x^2y^2z^2+1\)\\
    3 & 10 & \(x^2y^2z^6+x^4y^2z^2+x^2y^4-4x^2y^2z^2+1\)\\
    3 & 12 & \(x^4y^2z^6+x^4y^4z^2+y^2-4x^2y^2z^2+1\)\\
    3 & 14 & \(x^4y^4z^6+x^6y^2+x^6y^2z^2-4x^4y^2z^2+1\)\\
    3 & 16 & \(x^8y^6z^2+y^2+z^6-4x^2y^2z^2+1\)\\
    3 & 18 & \(x^6y^8z^4+z^6+x^2z^6-4x^2y^2z^4+1\)\\
    3 & 20 & \(x^{14}y^2z^4+x^2z^2+y^6z^2-4x^4y^2z^2+1\)\\
    \hline
    4 & 4 & No example provided by our algorithm -- see \Cref{sec:gen}\\
    4 & 6 & \(3x^2z^2w^2+y^2z^4+2y^2w^2+2x^2y^2-10wxyz+2\)\\
    4 & 8 & \(x^2y^2z^4+y^4z^4+x^2y^2w^4+2x^2z^4w^2-8xyz^2w+3\)\\
    4 & 10 & \(x^6y^2z^2+x^2w^8+y^8w^2+x^2z^8-5x^2y^2z^2w^2+1\)\\
    4 & 12 & \(x^6y^4w^2+x^6y^4+z^6w^6+y^4z^4+z^2w^4-5x^2y^2z^2w^2+1\)\\
    4 & 14 & \(x^2y^6z^4w^2+x^4z^4w^4+x^4y^4w^2+z^2w^2-5x^2y^2z^2w^2+1\)\\
    4 & 16 & \(x^8z^4w^4+x^8y^2z^2+x^4y^6x^2+y^2z^2w^6-5x^4y^2z^2w^2+1\)\\
    4 & 18 & \(x^8y^4+x^2y^4z^2w^2+x^4z^8w^6+x^6w^2+x^4y^4z^4w^2-6x^4y^2z^2w^2+1\)\\
    4 & 20 & \(x^{10}y^6w^2+y^6z^{10}+y^6z^6w^8+y^6z^4-x^2y^4z^4w^2+1\)\\
    \hline
\end{tabular}
\end{center}
\bigskip
\setlength\parindent{15pt}

Unfortunately, the algorithm written above slows considerably as \(n\) and \(d\) become large. Moreover, we quickly learned that this crude exhaustive search actually fails to find examples in all cases; the table of examples above does not contain an example when \(n=4\) and \(d=4\). To fill in these `gaps,' and to find examples for any degree and dimension, we now turn to the the construction of classes of general counterexamples.

\section{General Constructions}\label{sec:gen}

Our brute force search from the last section yields many examples of non-decomposable polynomials, but it is not suitable for finding examples when \(n\) and \(d\) are large. Thankfully, it is often easy to find examples of arbitrarily large degree for fixed \(n\). To illustrate, observe that for each \(m\geq1\) the triangle with vertices at \((0,1)\), \((1,m+2)\), and \((2,2m-1)\) has no boundary lattice points except the vertices. Moreover, the only interior lattice point is \((1,m)\), which is the average of the three vertex points. Since no two vertices sum to \((2,2m)\), the doubling of this triangle has the property outlined in \Cref{cor:hullcond}. Thus the following non-negative generalization of the Motzkin polynomial is not a sum of squares,
\[
    P_m(x,y)=x^4y^{4m-2}+x^2y^{2m+1}-3x^2y^{2m}+1.
\]

An identical argument works in three dimensions if one considers the tetrahedron with vertices at \((1,1,m+1)\), \((2,1,2m)\), \((1,2,m-1)\), and \((0,0,0)\), which contains the interior average \((1,1,m)\). In fact, the projection of this tetrahedron onto the first two variables is exactly the polytope for the Motzkin polynomial, meaning that polynomial
\[
    Q_m(x,y,z)=x^2y^2z^{2m+2}+x^4y^2z^{4m}+x^2y^4z^{2m-2}-4x^2y^2z^{2m}+1
\]
is not a sum of squares for any \(m\in\mathbb{N}\). For clarity, we draw \(C_{Q_1}\) and \(C_{Q_2}\) below with relevant points and point out that \(Q_1(x,y,z)\) resembles the example of Choi \& Lam from \cite{CL2}. Several more classes of examples like these are found by Reznik in \cite[\S 7]{reznik_extpsd} using a different technique.
\smallskip
\begin{center}
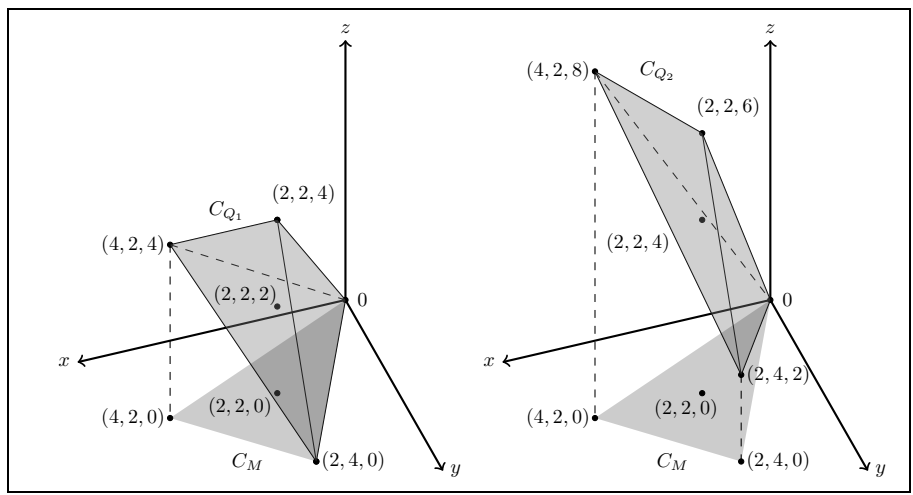
Newton Polytopes for \(Q_1\) (left) and \(Q_2\) (right).\\[0.5em] 
\framebox{
\tdplotsetmaincoords{50}{160}
\begin{tikzpicture}[tdplot_main_coords,line join=round,scale=0.75, every node/.style={scale=0.75}]
\path (2,2,4) coordinate (V)
(0,0,0)  coordinate (A)
(4,2,4)  coordinate (C)
(2,4,0)  coordinate (B)
(4,2,0)  coordinate (D)
(2,4,0)  coordinate (E)
(2,2,4)  coordinate (F)
(2,2,0)  coordinate (G)
(2,2,2)  coordinate (H);
\draw (B) -- (C)  -- (V) -- cycle;
\foreach \p in {A,B,C,V,D,E,F,G,H}
\draw[fill=black] (\p) circle (1.3pt);
\draw[dashed] (B) -- (E);
\draw[dashed] (C) -- (D);
\draw (A) -- (V);
\draw (A) -- (B);
\draw[dashed] (C) -- (A);

\draw[thick,->] (A) -- (5,0,0) node[left]{$x$};
\draw[thick,->] (A) -- (0,5,0) node[right]{$y$};
\draw[thick,->] (A) -- (0,0,6) node[above]{$z$};
\path[fill=gray, fill opacity = 0.2] (A) -- (C) -- (V) -- cycle;
\path[fill=gray, fill opacity = 0.2] (A) -- (B)  -- (C) -- cycle;
\path[fill=gray, fill opacity = 0.2] (B) -- (C)  -- (V) -- cycle;
\path[fill=gray, fill opacity = 0.2] (A) -- (B)  -- (V) -- cycle;
\path[fill=black, fill opacity = 0.2] (A) -- (D)  -- (E) -- cycle;
\path (A)+(0:3mm) node{$0$};
\path (E)+(0:6.5mm) node{$(2,4,0)$};
\path (D)+(180:6.5mm) node{$(4,2,0)$};
\path (C)+(180:6.5mm) node{\((4,2,4)\)};
\path (V)+(45:6.5mm) node{\((2,2,4)\)};
\path (G)+(200:7mm) node{\((2,2,0)\)};
\path (V)+(170:9mm) node{\(C_{Q_1}\)};
\path (E)+(180:12mm) node{\(C_M\)};
\draw (H)+(160:6mm) node{\((2,2,2)\)};
\end{tikzpicture}
\begin{tikzpicture}[tdplot_main_coords,line join=round,scale=0.75, every node/.style={scale=0.75}]
\path (2,2,6) coordinate (V)
(0,0,0)  coordinate (A)
(4,2,8)  coordinate (C)
(2,4,2)  coordinate (B)
(4,2,0)  coordinate (D)
(2,4,0)  coordinate (E)
(2,2,4)  coordinate (F)
(2,2,0)  coordinate (G);
\draw (B) -- (C)  -- (V) -- cycle;
\foreach \p in {A,B,C,V,D,E,F,G}
\draw[fill=black] (\p) circle (1.3pt);
\draw[dashed] (B) -- (E);
\draw[dashed] (C) -- (D);
\draw (A) -- (V);
\draw (A) -- (B);
\draw[dashed] (C) -- (A);
\draw[thick,->] (A) -- (5,0,0) node[left]{$x$};
\draw[thick,->] (A) -- (0,5,0) node[right]{$y$};
\draw[thick,->] (A) -- (0,0,6) node[above]{$z$};
\path[fill=gray, fill opacity = 0.2] (A) -- (C) -- (V) -- cycle;
\path[fill=gray, fill opacity = 0.2] (A) -- (B)  -- (C) -- cycle;
\path[fill=gray, fill opacity = 0.2] (B) -- (C)  -- (V) -- cycle;
\path[fill=gray, fill opacity = 0.2] (A) -- (B)  -- (V) -- cycle;
\path[fill=black, fill opacity = 0.2] (A) -- (D)  -- (E) -- cycle;
\path (A)+(0:3mm) node{$0$};
\path (B)+(0:6.5mm) node{$(2,4,2)$};
\path (E)+(0:6.5mm) node{$(2,4,0)$};
\path (D)+(180:6.5mm) node{$(4,2,0)$};
\path (C)+(180:6.5mm) node{\((4,2,8)\)};
\path (V)+(45:6.5mm) node{\((2,2,6)\)};
\path (G)+(225:4mm) node{\((2,2,0)\)};
\path (C)+(0:11mm) node{\(C_{Q_2}\)};
\path (E)+(180:12mm) node{\(C_M\)};
\draw[dashed] (F)+(200:12mm) node{\((2,2,4)\)};
\end{tikzpicture}
}
\medskip
\end{center}

These examples highlight a critical idea: the projection of an \(n\)-dimensional Newton polytope \(C\) into \(\mathbb{R}^{n-1}\) along a coordinate axis is also a polytope with lattice point vertices. Moreover, if the projection has the property outlined in \Cref{cor:hullcond} (i.e. it has an interior point which is not the average of two distinct points in the hull), then so too does \(C\). To verify that a polytope satisfies the condition of \eqref{cor:hullcond} then, it suffices to check its projections along the coordinate axes.

On the other hand, given a polytope \(C\) in \(\mathbb{R}^n\) with one vertex fixed at the origin, the rest at linearly independent points \(q_1,\dots,q_n\), and a point \(p\in C\) which is not the sum of two distinct lattice points in \(\frac{1}{2}C\), we can easily construct a polytope in \(\mathbb{R}^{n+1}\) corresponding to a non-decomposable polynomial. We do this by choosing non-negative even integers \(r_0,\dots,r_n\) and \(s\) so that \((p,s)\) is an interior point of the convex hull \(C'\) generated by the origin and the linearly independent points \((0,r_0),(q_1,r_1),\dots,(q_n,r_n)\in \mathbb{R}^{n+1}\). There are many ways to do this, and using these points we easily obtain a non-negative and non-decomposable polynomial.

To illustrate, we found earlier that \(x^4y^6+x^2y^6-3x^2y^4+1\) is not a sum of squares, since the point \(p=(2,4)\) is not the sum of two distinct even lattice points in the contraction by \(\frac{1}{2}\) of the closed planar triangle with vertices at \(q_1=(4,6)\), \(q_2=(2,6)\) and \(q_3=(0,0)\). Replacing \(p\) with \(\Tilde{p}=(2,4,2)\), and the other points respectively with \(\Tilde{q}_1=(4,6,2)\), \(\Tilde{q}_2=(2,6,2)\), \(\Tilde{q}_3=(0,0,4)\), we can set \(\Tilde{q}_4=(0,0,0)\) and it is easy to check using \Cref{thm:posweights} that \(\Tilde{p}\) belongs to the convex hull \(C\) induced by \(\Tilde{q}_1,\dots,\Tilde{q}_4\). Indeed, we can write
\[
    \Tilde{p}=\frac{1}{3}\Tilde{q}_1+\frac{1}{3}\Tilde{q}_2+\frac{1}{6}\Tilde{q}_3+\frac{1}{6}\Tilde{q}_4.
\]
There is nothing special about the third coordinates selected here; they were only chosen so that \(\Tilde{p}\in C'\), and our choice for \(C'\) and \(\Tilde{p}\) was far from unique. Since \(\Tilde{p}\) is not the sum of two distinct lattice points in \(\frac{1}{2}C'\) (this property is inherited from \(p\) and \(C\)), we find using \eqref{eq:shortone} that the following non-negative polynomial is not a sum of squares,
\[
    P(x,y,z)=2x^4y^6z^2+2x^2y^2z^2+z^4-6x^2y^4z^2+1.
\]

The preceding construction always increases the dimension of our polynomials by one, while the amount by which the degree increases will depend on the magnitudes of the constants \(r_0,\dots,r_n\) one chooses. Henceforth we focus on the most challenging case of constructing polynomials of minimal degree (i.e. degree \(d=4\) over \(n\geq 3\) variables), from which higher-degree examples can be easily constructed via the process described above. To do this we use a technique called homogenization.

\begin{defn}
Let \(P\) be a non-negative polynomial of even degree \(d\) on \(\mathbb{R}^n\). Then the homogenization of \(P\) is the non-negative polynomial on \(\mathbb{R}^{n+1}\) defined by
\[
    hP(x_1,\cdots,x_{n+1})=x_{n+1}^dP\bigg(\frac{x_1}{x_{n+1}},\cdots,\frac{x_n}{x_{n+1}}\bigg).
\]
\end{defn}

Using homogenization, we now devise a transformation which increases \(n\) by one while leaving \(d\) fixed. The idea is that homogenizing a degree \(d\) polynomial over \(n\) variables effectively `lifts' it's Newton polytope up to the hyperplane \(x_1+\cdots+x_{n+1}=d\) in \(\mathbb{R}^{n+1}\). By adding some well-chosen lower-order terms, we can `fill out' this homogenized polytope to give it volume.

Let \(C_P\subset\mathbb{R}^n\) denote a Newton polytope of a polynomial \(P\), and assume that \(C_P\) contains a point \(p\) as in \Cref{cor:hullcond}. That is, \(p\) is not the sum of two distinct lattice points in \(\frac{1}{2}C_P\). The polytope \(C_{hP}\) is a segment of a hyperplane in \(\mathbb{R}^{n+1}\) which can be obtained by replacing each vertex \(q=(q_1,\dots,q_n)\) of \(C_P\) with a new point \(q'=(q_1,\dots,q_n,q_{n+1})\), where \(q_{n+1}\) is chosen so that \(q_1+\cdots+q_{n+1}=d\). For instance, if \(0\) is a vertex of \(C_P\) then it is mapped to \((0,\dots,0,d)\), while extremal points for which \(q_1+\cdots+q_n=d\) are unchanged.

Next we modify the polytope \(C_{hP}\) by appending the origin to it -- this is equivalent to adding a constant to the homogenized polynomial. We do this by adding \(1+x^{2q'}-x^{q'}\) for some vertex \(q'\) of \(C_{hP}\). Denoting by \(C_Q\) the convex hull generated by the points \(0\) and \(q_1',\dots,q_n'\), we argue that this polytope and \(p'\) have the property outlined in \Cref{cor:hullcond}.

Observe that the critical point \(p\in C_P\) is mapped by homogenization to a point \(p'\) on the hyperplane of \(C_{hP}\), meaning that \(p'\in C_Q\). Hence \(p'\) can be written as a convex combination of the vertex points \(q_1',\dots,q_n'\), and it cannot be realized as the sum of any two distinct points in \(\frac{1}{2}C_Q\) since this is not possible in the first \(n\) variables by assumption. It follows that the following polynomial is non-negative and not a sum of squares
\[
    Q(x)=hP(x)+\sum_{j=1}^{n+1}c_j(x^{2q_j'}+1-x^{q_j'})=\sum_{j=1}^{n+1}\lambda_jx^{2q_j'}-x^{p'}+\sum_{j=1}^{n+1}c_j(x^{2q_j'}+1-2x^{q_j'}).
\]
Here, \(c_1,\dots,c_{n+1}\) are non-negative constants and not all zero. To illustrate, if we begin with the minimal degree four counter example in three dimensions \(P(x,y,z)=x^2y^2+y^2z^2+x^2z^2-4xyz+1\) of Choi \& Lam \cite{CL2}, then one of the many examples this procedure gives  is
\[
    Q(w,x,y,z)=x^2y^2+y^2z^2+x^2z^2-4wxyz+2w^4-2w^2+1.
\]

It is easy to check that \(Q(1,1,1,1)=0\), that \(Q\) is positive elsewhere, and \(Q\) cannot be written as a sum of squares. The latter fact can be verified directly by analyzing the Newton polytope of \(Q\), but it is much easier to observe that if \(Q=g_1^2+\cdots+g_m^2\) for polynomials \(g_1,\dots,g_m\) then
\[
    Q(1,x,y,z)=x^2y^2+y^2z^2+x^2z^2-4xyz+1=g_1(x,y,z,1)^2+\cdots+g_m(x,y,z,1)^2.
\]
This contradicts the fact that \(P\) is not a sum of squares. The construction above gives us a polynomial of degree four over four variables which is not a sum of squares, providing the missing entry in the table of polynomials above.

Iterating this construction affords examples of polynomials of degree \(d=4\) over \(n\geq 3\) variables which are not sums of squares of polynomials. Repeating the argument employed in \(n=2\) and \(n=3\) dimensions at the beginning of the section also gives transformations which increase the degree, rather than the dimension. Using these techniques, it is possible to construct polynomials which are not sums of squares in all possible cases.

\subsection{Discussion}

We leave much of the theory of polynomial sums of squares untouched, however insight into our main application can be gleaned from some known results. Lasserre \cite{lasserre} shows that polynomial sums of squares are dense in the space of polynomials equipped with the norm
\[
    \bigg\|\sum_\alpha c_\alpha x^\alpha\bigg\|=\sum_{\alpha}|c_\alpha|,
\]
and in particular this implies that on any compact set a given polynomial can be uniformly approximated by a sum of squares. Later, we find that an analogous result holds in the setting of H\"older continuous functions; see \Cref{thm:almostthere}.

It is interesting that non-negative polynomials of degree \(d=4\) over \(n=2\) variables can always be decomposed, for in the two dimensional setting we later find that higher-order decompositions are more straightforward, in the sense that fewer restrictions need to be imposed on a function for a decomposition to be possible. Moreover, the non-existence of a polynomial counter example in the setting of \(C^{4,\alpha}(\mathbb{R}^2)\) makes it unclear whether or not the functions in that space are all sums of half-regular squares; see the brief discussion of \Cref{sec:maybe}.

To summarize this chapter on non-decomposable functions, we have found many new examples of polynomials which are non-negative and not sums of squares, and ways to construct yet more. Moreover, we have demonstrated each of these correspond to a function which cannot be written as a sum of squares of half-regular functions. Equipped with these examples, we can return to our study of H\"older continuous functions.
\chapter{Higher Order Decompositions}\label{chap:higher}

In \Cref{chap:decomps} we studied decompositions in \(C^{k,\alpha}(\mathbb{R}^n)\) for \(k\leq3\), and now we move on to higher-order H\"older spaces. If \(k\geq 4\) then the function \(r\) given by \eqref{eq:controlfunc} contains at least three terms instead of two, and the proof of \Cref{thm:c3main} fails since the local behaviour of \(f\) can no longer be restricted to two cases -- in particular, when \(f\) is not bounded below, we can no longer use the Implicit Function Theorem to extract the local minimizer function \(F\). As Korobenko \& Sawyer point out in \cite{SOS_I}, this causes the inductive technique introduced in \cite{Fefferman-Phong} to fail, and indeed if \(k\geq 4\) then there exist non-negative functions in \(C^{k,\alpha}(\mathbb{R}^n)\) which cannot be written as sums of half-regular squares -- we found such functions in the last chapter. As such, to form regularity preserving sum of squares decompositions of \(C^{k,\alpha}(\mathbb{R}^n)\) for \(k\geq 4\), it is necessary to impose additional hypotheses on \(f\).

This chapter presents some results concerning partial decompositions; the first shows that a non-negative \(C^{k,\alpha}(\mathbb{R}^n)\) function \(f\) can be decomposed as sums of squares plus a remainder which can be made arbitrarily small (at least in magnitude). Another result (\Cref{thm:seconddmain}) allows us to decompose \(f\) as a sum of squares which retain slightly less than half the regularity of \(f\), but this result is exclusive to functions which satisfy certain differential inequalities. It turns out that in two dimensions at least, we can assume less restrictive inequalities yet still recover half-regular decompositions -- see \Cref{thm:nooooo}.

Little is known about the decompositions of higher order functions in more than one dimension, and our work expands upon results in the known literature, in particular \cite{SOS_I}. However, we fall short of a total characterization of the decomposable functions. We hope that these results may provide deeper insight into the decompositions problem, and perhaps be a starting point for further investigation into a characterization.

\section{Partial Decompositions}

We begin our exploration of higher order decompositions with a qualitative remark; for each fixed \(n\), we should expect the necessary conditions for existence of a half-regular decomposition to become more restrictive as \(n\) and \(k\) become large, owing to the abundance of polynomial counter-examples noted by Blekherman in \cite{Blekherman} and by our constructions in the previous chapter. Similarly, we expect that \(C^{k,\alpha}(\mathbb{R}^n)\) functions that are not sums of half-regular squares to become increasingly common for large \(k\), and this will be reflected by a stronger condition needed to ensure that decompositions exist.

Before examining some sufficient conditions, we point out that the proof of \Cref{thm:c3main} already affords an interesting and potentially useful decomposition result that requires no extra hypotheses on \(f\).

\newpage

\begin{thm}\label{thm:almostthere}
Let \(f\) be a non-negative \(C^{k,\alpha}_b(\mathbb{R}^n)\) function, and let \(\varepsilon>0\) be given. There exists a non-negative function \(h\in C_b^{k,\alpha}(\mathbb{R}^n)\) such that \(\sup_{\mathbb{R}^n}|h|\leq\varepsilon\) and 
\[
    f=\sum_{j=1}^{m_n}g_j^2+h
\]
for functions \(g_1,\dots,g_{m_n}\in C_b^\frac{k+\alpha}{2}(\mathbb{R}^n)\). Moreover, the constant \(m_n\) depends only on \(n\).
\end{thm}

An equivalent way of rephrasing this result is as follows: if \(f\in C^{k,\alpha}(\mathbb{R}^n)\) is non-negative and has  bounded derivatives, there exists a function \(\Tilde{f}\) which is also in \(C^{k,\alpha}(\mathbb{R}^n)\), \(\Tilde{f}\) is a sum of squares of half regular functions, and \(0\leq \Tilde{f}\leq f\leq \Tilde{f}+\varepsilon\) everywhere. This is a slightly stronger result than one obtains by simply taking a square root \(f+\varepsilon\), since the function \(\Tilde{f}\) is controlled by \(f\) from above. 


\begin{proof}
Let \(B(x_j,\nu r_j)\) be as in \Cref{thm:party} for small \(\nu\), and recall from the proof of \Cref{thm:c3main} that if \(f(x_j)\geq \omega\nu r_j^{k+\alpha}\) for some \(j\), then 
\[
    \psi_j\sqrt{f}\in C^\frac{k+\alpha}{2}(\mathbb{R}^n).
\]
Combining the functions for which this inclusion holds as in \eqref{eq:recombsum}, we obtain a sum of at most \(m_n=N_n=15^n\) squares \(g_1,\dots,g_{m_n}\) which retain half the regularity of \(f\). Now we examine the remainder \(f-(g_1^2+\cdots+g_{m_n}^2)\). If
\[
    f(x_j)<\omega\nu r_j^{k+\alpha}, 
\]
then from \Cref{cor:supplementary} it follows that \(f(x)\leq \frac{3}{2}\omega\nu r_j^{k+\alpha}\) for \(x\in B(x_j,\nu r_j)\). Since \(r_j=r(x_j)\) is bounded whenever \(f\) and its derivatives are uniformly bounded, we can choose \(\nu\) small enough to ensure that \(f\) is as small as we like on \(B(x_j,\nu r_j)\). Since \(\psi_j\leq 1\) by construction, we see that we may choose \(\nu\) sufficiently small that
\begin{equation}\label{eq:tinnnnny}
    \sup_{\mathbb{R}^n}\psi_j^2f\leq \frac{\varepsilon}{N_n}.
\end{equation}
Now, let \(S\) be the collection of \(j\in\mathbb{N}\) for which \(f(x_j)<\omega\nu r_j^{k+\alpha}\) and define
\[
    h=\sum_{j\in S}\psi_j^2 f.
\]
Owing to our construction of the half-regular functions \(g_1,\dots,g_{m_n}\) above, we are able to write
\[
    f=\sum_{j=1}^{m_n}g_j^2+h.
\]
Since each \(x\in\mathbb{R}^n\) is in at most \(N_n\) of the balls on which the functions \(\psi_j\) are supported, it follows from \eqref{eq:tinnnnny} that \(h\leq \varepsilon\) everywhere. Moreover, since each \(\psi_j\) is smooth and \(f\in C_b^{k,\alpha}(\mathbb{R}^n)\), we conclude by item \textit{(5)} of \Cref{thm:party} together with the recombination \Cref{lem:recomb} that \(h\in C_b^{k,\alpha}(\mathbb{R}^n)\), as claimed.
\end{proof}

This result has two major shortcomings when compared to \Cref{thm:c3main}. First, it requires the assumption that \(\|f\|_{C_b^{k,\alpha}}<\infty\), as opposed to the weaker condition required for \Cref{thm:c3main},
\[
    \sum_{|\beta|=k}[\partial^\beta f]_{\alpha,\mathbb{R}^n}<\infty.
\]
Second, and more importantly, it only shows that \(f\in C_b^{k,\alpha}(\mathbb{R}^n)\) is `almost' decomposable, in the sense that we can find a decomposition that uniformly approximates \(f\). However, \Cref{thm:almostthere} gives no control of the derivatives of the small remainder \(h\), and the examples of the previous chapter show that in general we cannot expect \(\sqrt{h}\) to be half as regular as \(f\). 

The following useful result of Bony et al. from \cite{Bony} indicates that in one dimension, no additional hypotheses are needed for a regularity preserving decomposition to exist. Recall that we proved this result in \Cref{sec:c3proof}.

\begin{thm}[Bony]\label{thm:onedim}
If \(f\in C^{k,\alpha}(\mathbb{R})\) is non-negative, it is a sum of two squares in \(C^\frac{k+\alpha}{2}(\mathbb{R})\).
\end{thm}

In \cite[Theorem 4.5]{SOS_I}, Sawyer \& Korobenko show that if the even-order derivatives of a non-negative function \(f\in C^{4,\alpha}(\mathbb{R}^n)\) satisfy pointwise inequalities that depend on \(\alpha\), and which resemble the Malgrange inequalities we proved in \Cref{thm:malgrangeineq}, then \(f\) can be decomposed as a sum of squares. However, the decomposition functions do lose some regularity in H\"older scale, meaning that they are not all half as regular as \(f\). Our next result, \Cref{thm:seconddmain}, is a generalization of the result from \cite{SOS_I}. We prove it using an inductive argument, and by assuming some similarly restrictive pointwise properties of \(f\).

Plainly, this theorem states that we can decompose \(f\) as a sum of squares of H\"older continuous functions, but our decomposition loses regularity on H\"older scale in high dimensions, and moreover the theorem imposes a lot of conditions on \(f\). This is unsurprising, since even the roots of non-negative smooth functions need not inherit more than one derivative -- recall the example from the introduction of this thesis.

\begin{thm}\label{thm:seconddmain}
Let \(f\) be a non-negative function in \(C^{k,\alpha}(\mathbb{R}^n)\) for \(n\geq2\), \(k\geq 4\), and \(0<\alpha\leq 1\). Assume that the following inequalities hold pointwise on \(\mathbb{R}^n\) for a constant \(\eta<\frac{k-2+\alpha}{k+\alpha}\) and for every even \(\ell\) such that \(2<\ell\leq k\),
\begin{equation}\label{eq:induc}
    |\nabla^2 f(x)|\leq Cf(x)^\eta\qquad \textrm{and}\qquad|\nabla^\ell f(x)|\leq Cf(x)^\frac{k-\ell+\alpha}{k+\alpha}.
\end{equation}
Then there exists a positive number \(\alpha_n<\alpha\) and functions \(g_1,\dots,g_{m_n}\in C^\frac{k+\alpha_n}{2}(\mathbb{R}^n)\) such that
\[
    f=\sum_{j=1}^{m_n} g_j^2.
\]
Moreover, \(\alpha_n\) is given by taking \(\alpha_2=\alpha\) and recursively defining \(\alpha_{j+1}=\frac{k\eta}{k-2+\alpha_j-\eta}\).
\end{thm}

\textit{Remark}: If \eqref{eq:induc} holds for \(\eta\geq \frac{k-2+\alpha}{k+\alpha}\) then \(f(x)\geq cr(x)^{k+\alpha}\) everywhere and \(\sqrt{f}\in C^\frac{k+\alpha}{2}(\mathbb{R}^n)\) by \Cref{lem:local1}. Thus, this result only yields new information if \(\eta\) is small enough. It is unclear whether the inequalities in \eqref{eq:induc} are sharp in general, in the sense that any weaker set of inequalities can be satisfied by a function which is non-decomposable in the sense described above. It is shown in \cite{SOS_I} that when \(k=4\) the inequalities of \eqref{eq:induc} are sharp for a class of functions which Sawyer \& Korobenko call `H\"older monotone.' In any case, it is not clear if \eqref{eq:induc} can be strengthened to ensure that there exists a decomposition that sacrifices no regularity on H\"older scale, and which is not simply comprised of \(\sqrt{f}\) on its own.

\begin{proof}
As in the proof of \Cref{thm:c3main}, we proceed using induction on \(n\). The base case \(n=1\) is covered by \Cref{thm:onedim}, in which the differential inequalities \eqref{eq:induc} are not needed. This leaves us to deal with the inductive step.

Thus, we assume now that every non-negative function in \(C^{k,\alpha}(\mathbb{R}^{n-1})\) which satisfies \eqref{eq:induc} can be decomposed as a sum of squares in the space \(C^\frac{k+\alpha_{n-1}}{2}(\mathbb{R}^{n-1})\), and we let \(f\in C^{k,\alpha}(\mathbb{R}^n)\) satisfy \eqref{eq:induc}. It follows that
\begin{equation}\label{eq:equiv27}
    r_f(x)=\max_{\substack{0\leq j\leq k,\\j\;\mathrm{even}}}\bigg\{\sup_{|\xi|=1}[\partial^j_\xi f(x)]_+^\frac{1}{k-j+\alpha}\bigg\}\leq C\max\bigg\{f(y)^\frac{1}{k+\alpha},\sup_{|\xi|=1}[\partial^2_\xi f(y)]_+^\frac{1}{k-2+\alpha}\bigg\}.
\end{equation}
Via \Cref{thm:party} the slowly-varying function \(r_f\) induces a partition of unity; as usual, we fix one of the balls \(B(x_j,\nu r_j)\) on which a partition function \(\psi_j\) is supported, and we consider the behaviour of \(f\) at its center. If \(f(x_j)\geq \omega\nu r_j^{k+\alpha}\), where \(r_j=r_f(x_j)\), then \(f\) is bounded below and \Cref{lem:rootscors} shows that 
\[
    \psi_j\sqrt{f}\in C^\frac{k+\alpha}{2}(\mathbb{R}^n).
\]
On the other hand, if \(f(x_j)< \omega\nu r_j^{k+\alpha}\) then \eqref{eq:equiv27} implies that
\[
    f(x_j)< C\omega\nu \max\bigg\{f(x_j),\sup_{|\xi|=1}[\partial^2_\xi f(x_j)]_+^\frac{k+\alpha}{k-2+\alpha}\bigg\},
\]
and for \(\nu\) small enough (recall that \(\nu\) may be chosen as small as we like) this yields a contradiction if the maximum above equals \(f(x_j)\). It follows that 
\[
    \sup_{|\xi|=1}[\partial^2_\xi f(x_j)]_+^\frac{1}{k-2+\alpha}\leq r_j\leq C\sup_{|\xi|=1}[\partial^2_\xi f(x_j)]_+^\frac{1}{k-2+\alpha},
\]
and moreover since \(r_j>0\) it follows that \(\sup_{|\xi|=1}\partial^2_\xi f(x_j)>0\). Thus the argument employed to prove \Cref{thm:c3main} allows us to locally define a non-negative function \(F_j\) for which 
\[
    \psi_j\sqrt{f-F_j}\in C^\frac{k+\alpha}{2}(\mathbb{R}^n).
\]

To apply the inductive hypothesis and complete our decomposition of \(f\), we must show that the remainder function \(F_j\) satisfies the differential inequalities \eqref{eq:induc} in \(n-1\) dimensions. If we consider the control function for \(F_j\) defined by
\[
    r_{F_j}(x')=\max_{\substack{0\leq j\leq k,\\j\;\mathrm{even}}}\bigg\{\sup_{|\xi|=1}[\partial^j_\xi F_j(x')]_+^\frac{1}{k-j+\alpha}\bigg\},
\]
then the estimates of \Cref{lem:implicitest} show that \(r_{F_j}(x')\leq Cr_f(x)\), and since \(r_f\) is slowly varying on \(B(x_j,\nu r_j)\), we find that \(r_f(x)\leq Cr_f(x',X(x'))\). Combining these estimates with the first inequality of \eqref{eq:induc}, we find altogether that
\[
    r_F(x')\leq Cr_f(x',X(x'))\leq C\sup_{|\xi|=1}[\partial^2_\xi f(x',X(x'))]_+^\frac{1}{k-2+\alpha}\leq Cf(x',X(x'))^\frac{\eta}{k-2+\alpha}=CF_j(x')^\frac{\eta}{k-2+\alpha}.
\]
In particular, this implies that \(|\nabla^2 F_j(x')|\leq C r_{F_j}(x')^{k-2+\alpha}\leq CF_j(x')^\eta\), showing that \(F_j\) inherits the first differential inequality of \eqref{eq:induc} from \(f\). Additionally, for even \(\ell>2\) we see that
\[
    |\nabla ^\ell F_j(x')|\leq r_{F_j}(x)^{k-\ell+\alpha}\leq CF_j(x')^\frac{\eta(k-\ell+\alpha)}{k-2+\alpha}.
\]

Now, observe that if \(f(x_j)< \omega\nu r_j^{k+\alpha}\), as we are assuming, then \(f(x_j)\leq Cf(x_j)^\eta\) and \(f(x_j)^{1-\eta}\leq C\). After rescaling if necessary, we can assume without loss of generality that \(F_j(x')\leq f(x)\leq 1\) where \(F_j\) is defined. Defining \(\tilde{\alpha}=\frac{k\eta}{k-2+\alpha-\eta}\) so that \(\frac{\eta}{k-2+\alpha}=\frac{\tilde{\alpha}}{k+\tilde{\alpha}}\), we see that \(\tilde{\alpha}<\alpha\) when \(\eta<\frac{k-2+\alpha}{k+\alpha}\) and
\[
    \frac{\eta(k-\ell+\alpha)}{k-2+\alpha}=\frac{k-\ell+\alpha}{k+\tilde{\alpha}}\geq \frac{k-\ell+\tilde{{\alpha}} }{k+\tilde{\alpha}}.
\]
Thus for every even \(\ell \) with \(2<\ell\leq k\) we have that \(|\nabla ^\ell F_j(x')|\leq CF_j(x')^\frac{k-\ell+\tilde{\alpha} }{k+\tilde{\alpha}}\). Moreover an identical set of inequalities is satisfied by the extension of \(F_j\) to all of \(\mathbb{R}^{n-1}\) via \Cref{lem:crudeext}, and since \(F_j\in C^{k,\alpha}(\mathbb{R}^{n-1})\) and it is compactly supported, it also belongs to \(C^{k,\tilde{\alpha}}(\mathbb{R}^{n-1})\). It follows by our inductive hypothesis we can write \(F_j\) as a sum of squares,
\[
    F_j=\sum_{\ell=1}^{m_{n-1}} g_\ell^2,
\]
where each \(g_\ell\) belongs to \(C^{\frac{k+\tilde{\alpha}_{n-1}}{2}}(\mathbb{R}^{n-1})=C^{\frac{k+\alpha_n}{2}}(\mathbb{R}^{n-1})\). Since \(\psi_j\) is smooth and compactly supported, \(\psi_jg_\ell\) also has compact support and belongs to the same H\"older space. Using item \textit{(5)} of \Cref{thm:party}, we see now that we can combine the locally decomposed functions as in the proof of \Cref{thm:c3main} to write
\[
    f=\sum_{j=1}^{m_n}g_j^2,
\]
where \(g_j\in C^{\frac{k+\alpha_n}{2}}(\mathbb{R}^n)\) for each \(j\) and \(m_n\leq N_n(1+m_{n-1})\), as we wished to show.
\end{proof}

The first inequality in \eqref{eq:induc} is only used to show that \(F_j\) inherits the remaining differential inequalities from \(f\), and since decomposition of functions over \(\mathbb{R}\) require no differential inequalities by \Cref{thm:onedim}, we can omit the first differential inequality in the case \(n=2\). The proof above thus yields the following special case in two dimensions which involves no parameter \(\eta\), and which generalizes \cite[Theorem 4.7]{SOS_I}.

\begin{thm}\label{thm:nooooo}
Let \(f\) be a non-negative function in \(C^{k,\alpha}(\mathbb{R}^2)\) for \(k\geq 4\) and \(0<\alpha\leq 1\). Assume that the following inequalities hold pointwise on \(\mathbb{R}^2\)
\begin{equation}\label{eq:induc2}
    |\nabla^\ell f(x)|\leq Cf(x)^\frac{k-\ell+\alpha}{k+\alpha}.
\end{equation}
Then there exist  functions \(g_1,\dots,g_{m_n}\in C^\frac{k+\alpha}{2}(\mathbb{R}^2)\) such that \( f=g_1^2+\cdots+g_{m_n}^2.\)

\end{thm}

\newpage

\section{The Special Case of \texorpdfstring{\(C^{4,\alpha}(\mathbb{R}^2)\)}{}}\label{sec:maybe}

The result of Hilbert in \cite{Hilbert} shows that every non-negative polynomial of degree four over \(\mathbb{R}^2\) can be written as a sum of squares of polynomials, meaning that the techniques of \Cref{chap:poly} give no example of a function in \(C^{4,\alpha}(\mathbb{R}^2)\) which is not a sum of half-regular squares. While we do not claim that every function in this space is not a sum of half-regular squares, there is evidence to suggest that \Cref{thm:nooooo} holds with considerably weaker inequalities than \eqref{eq:induc2} in this setting. Indeed, in proving that a non-negative function \(f\in C^{4}(\mathbb{R}^2)\) is a sum of squares of \(C^{2}(\mathbb{R}^2)\) functions in \cite{Bony}, Bony requires only that the fourth derivatives of \(f\) vanish whenever \(f\) and its second derivatives simultaneously vanish. 

Thus, it may be the case that \(C^{4,\alpha}(\mathbb{R}^n)\) functions can be decomposed under weaker assumptions than those we used in the previous section. Our workings in this thesis do not indicate what such conditions look like, and this special case awaits further examination.
\chapter{Supporting Results}\label{chap:alg}

In this chapter we collect results which are called upon in the main body of the work, but which differ enough in character from our study of decompositions that they would distract from the central narrative of this thesis. With the potential exception of \Cref{thm:specialodds}, the results of this chapter are far from new -- we choose not to confine them to an appendix solely because we furnish proofs which could not be found in any references. There is no unified theme to this chapter, and accordingly the reader may wish to visit these sections as specific results are called upon elsewhere in the work.

\section{An Algebraic Result}

The first result we give is required in the proof of \Cref{thm:cauchylike}, and it ensures that for a given set of non-negative numbers, we can choose a set of weights whose odd powers enjoy the vanishing property of equation \eqref{eq:wowza}. This result is used to eliminate selective odd-order terms in the Taylor expansions of non-negative functions.

\begin{thm}\label{thm:specialodds}
Let \(\ell\) be any odd integer and set \(s=\frac{\ell+1}{2}\). There exist numbers \(\eta_1,\dots,\eta_s\in\mathbb{R}\) and \(x_1,\dots,x_s\geq 0\) such that for each odd \(j\leq \ell\),
\begin{equation}\label{eq:wowza}
    \sum_{i=1}^sx_i\eta_i^j=\begin{cases}
    0 & j<\ell,\\
    1 & j=\ell.
    \end{cases}
\end{equation}
\end{thm}

\begin{proof}
It suffices to show that \(Mx=e_s\) has a non-negative solution \(x\), where \(e_s\) is the last standard basis vector in \(\mathbb{R}^s\) and the matrix \(M\) is given by
\[
    M=\begin{bmatrix}
    \eta_1&\eta_2&\cdots &\eta_s\\
    \eta_1^3&\eta_2^3&\cdots&\eta_s^{3}\\
    \vdots&\vdots&\ddots&\vdots\\
    \eta_1^\ell &\eta_2^\ell &\cdots&\eta_s^\ell 
    \end{bmatrix}.
\]
Our aim is to select \(\eta_1,\dots,\eta_s\) so that \(M\) is invertible and the entries of \(M^{-1}e_s\) are non-negative. Denoting by \(x_k\) the \(k^\mathrm{th}\) entry in column \(s\) of \(M^{-1}\), and assuming for now that \(M\) is nonsingular, we can write
\[
    x_k=\frac{(-1)^{s+k}\mathrm{det}(M_{sk})}{\mathrm{det}(M)},
\]
where \(M_{sk}\) is the matrix minor of \(M\) obtained by deleting row \(s\) and column \(k\) of \(M\).

Using the properties of \(M\), we establish a formula for \(\mathrm{det}(M)\). Note that this is a homogeneous polynomial in the variables \(\eta_1,\dots,\eta_s\) of degree \(s^2\). If \(\eta_i=0\) for any \(i\) then \(M\) is singular, meaning that \(\eta_i\mid \mathrm{det}(M)\). Similarly, if \(i\neq j\) and \(\eta_i=\pm\eta_j\) then \(M\) has two linearly dependent columns, meaning once again that \(\mathrm{det}(M)=0\) and \((\eta_i^2-\eta_j^2)\mid \mathrm{det}(M)\) whenever \(i\neq j\). It follows that \(\mathrm{det}(M)=PQ\) where \(P\) and \(Q\) are polynomials and \(P\) is given by
\[
    P=\bigg(\prod_{i=1}^s\eta_i\bigg)\bigg(\prod_{i=1}^s\prod_{j=1}^{i-1}(\eta_i^2-\eta_j^2)\bigg).
\]
This polynomial has degree \(s^2\), and it follows from the Fundamental Theorem of Algebra that \(Q\) has degree zero and is constant. Moreover, the first term in the expansion of \(P\) equals the product down the diagonal of \(M\). Comparing coefficients, we see that \(Q=1\) and \(\mathrm{det}(M)=P\).

We can compute \(\mathrm{det}(M_{sk})\) with the formula above, since this minor assumes the same form as \(M\). Omitting appropriate terms from the determinant formula for \(M\), we get
\[
    \mathrm{det}(M_{sk})=\bigg(\prod_{\substack{1\leq i\leq s\\
    i\neq k}}\eta_i\bigg)\bigg(\prod_{\substack{1\leq j< i\leq s\\i,j\neq k}}(\eta_i^2-\eta_j^2)\bigg)
\]
Equipped with this identity and the formula for \(x_k\) above, we are now able to write
\[
    x_k=(-1)^{s+k}\bigg(\eta_k\prod_{j=1}^{k-1}(\eta_k^2-\eta_j^2)\prod_{i=k+1}^s(\eta_i^2-\eta_k^2)\bigg)^{-1}=\bigg(\eta_k\prod_{\substack{1\leq i\leq s\\i\neq k}}(\eta_k^2-\eta_i^2)\bigg)^{-1}
\]
To ensure that each \(x_k\) is positive, it suffices to choose \(\eta_1,\dots,\eta_s\) so that \(|\eta_1|<\cdots<|\eta_s|\) and \(\mathrm{sgn}(\eta_k)=(-1)^{s+k}\). Therefore taking \(\eta_k=(-1)^{s+k}k\) gives \eqref{eq:wowza}.
\end{proof}

Our choice of \(\eta_k\) above suffices for the purpose of proving \Cref{thm:cauchylike}, but other choices may be better suited to other applications.

As one example, the result above might be useful for understanding the coefficients of non-negative polynomials over \(\mathbb{R}\) via a similar argument to that employed in proving \Cref{thm:cauchylike}. In particular, let
\[
    P(x)=\sum_{j=0}^dc_jx^j
\]
be a non-negative polynomial on \(\mathbb{R}\) of even degree \(d\), fix \(\ell=d-1\), and choose constants \(x_1,\dots,x_s\) and \(\eta_1,\dots,\eta_s\) as in \eqref{eq:wowza}. Since \(x_iP(-\eta_ix)\geq 0\) for each \(i\) we have for all \(x\in\mathbb{R}\) that
\[
    0\leq \sum_{i=1}^sx_iP(-\eta_ix)=\sum_{i=1}^sx_i\sum_{j=0}^dc_j(-\eta_i)^jx^j=\sum_{\substack{0\leq j\leq d,\\j\textrm{even}}}^dc_jx^j\bigg(\sum_{i=1}^sx_i\eta_i^j\bigg)-c_\ell x^\ell
\]
Replacing \(x_iP(-\eta_ix)\) with \(x_iP(\eta_ix)\) and repeating this argument, we find after rearranging and taking a maximum that
\[
    |c_\ell|\leq \sum_{\substack{0\leq j\leq d,\\j\textrm{even}}}^dc_jx^{j-\ell}\bigg(\sum_{i=1}^sx_i\eta_i^j\bigg).
\]

By bounding the \(s\)-dependant terms in the inequality above and optimizing over \(x\), it is possible to obtain a bound on \(c_\ell\). Iterating this argument also affords bounds on all other odd-order coefficients of \(P\). Moreover, by using a complete Vandermonde matrix in \Cref{thm:specialodds} instead of one omitting even rows, the interested reader might also glean information about the sums in \eqref{eq:wowza} for even values of \(j\), and thereby recover explicit bounds in the preceding application and \Cref{thm:cauchylike}.

\section{Useful Estimates}

In this short section we establish some straightforward results from elementary real analysis relating to properties of the maximum, the supremum, and the triangle inequality. Their proofs are not difficult, but the results warrant justification which we include here. Our first lemma states a simple property of the maximum.

\begin{lem}\label{lem:maxproplem}
For \(m\in\mathbb{N}\) let \(a_1,\dots,a_m\) and \(b_1,\dots,b_m\) be real numbers. Then
\begin{equation}\label{eq:maxprop}
    |\max_{j\leq m}a_j-\max_{j\leq m}b_j|\leq \max_{j\leq m}|a_j-b_j|.
\end{equation}
\end{lem}

\begin{proof}
We use induction on \(m\), proving the non-trivial case \(m=2\) as our base case. Note that
\[
    \max\{a_1,b_1\}=\max\{a_1-a_2+a_2,b_1-b_2+b_2\}\leq \max\{a_1-a_2,b_1-b_2\}+\max\{a_2,b_2\}.
\]
Rearranging this, we see that \(\max\{a_1,b_1\}-\max\{a_2,b_2\}\leq \max\{a_1-a_2,b_1-b_2\}\), and the same inequality clearly holds when \(a_1\) is swapped with \(a_2\) and \(b_1\) with \(b_2\). Consequently,
\begin{align*}
    |\max\{a_1,b_1\}-\max\{a_2,b_2\}|&=\max\{\max\{a_1,b_1\}-\max\{a_2,b_2\},\max\{a_2,b_2\}-\max\{a_1,b_1\}\}\\
    &\leq \max\{\max\{a_1-a_2,b_1-b_2\},\max\{a_2-a_1,b_2-b_1\}\}\\
    &=\max\{a_1-a_2,b_1-b_2,a_2-a_1,b_2-b_1\}\\
    &=\max\{|a_1-a_2|,|b_1-b_2|\}.
\end{align*}
This establishes the base case. Suppose next that inequality \eqref{eq:maxprop} holds for \(m\) terms. Using the preceding argument we obtain the bound
\begin{align*}
    |\max_{j\leq m+1}a_j-\max_{j\leq m+1}b_j|&=|\max\{\max_{j\leq m}a_j,a_{m+1}\}-\max\{\max_{j\leq m}b_j,b_{m+1}\}|\\
    &\leq \max\{|\max_{j\leq m}a_j-\max_{j\leq m}b_j|,|a_{m+1}-b_{m+1}|\},
\end{align*}
and with this it follows from our inductive hypothesis that
\begin{align*}
    |\max_{j\leq m+1}a_j-\max_{j\leq m+1}b_j|\leq \max\{ \max_{j\leq m}|a_j-b_j|,|a_{m+1}-b_{m+1}|\}= \max_{j\leq m+1}|a_j-b_j|  .  
\end{align*}
Hence inequality \eqref{eq:maxprop} holds with \(m+1\) terms, and the claimed estimates follow by induction.
\end{proof}

Now we establish a similar property for the supremum.

\begin{lem}\label{lem:supprop}
For any non-negative function \(g:\mathbb{R}^n\times\mathbb{R}^n\rightarrow\mathbb{R}\) and for any two points \(x,y\in\mathbb{R}^n\),
\[
    |\sup_{\xi\in\mathbb{R}^n} g(x,\xi)-\sup_{\xi\in\mathbb{R}^n} g(y,\xi)|\leq \sup_{\xi\in\mathbb{R}^n} |g(x,\xi)-g(y,\xi)|.
\]
\end{lem}

\begin{proof}
For any non-negative function \(g\) we can use subadditivity of the supremum to get
\[
    \sup_{\xi\in\mathbb{R}^n} g(x,\xi)=\sup_{\xi\in\mathbb{R}^n} (g(x,\xi)-g(y,\xi)+g(y,\xi))\leq \sup_{\xi\in\mathbb{R}^n} |g(x,\xi)-g(y,\xi)|+\sup_{\xi\in\mathbb{R}^n} g(y,\xi).
\]
Therefore we can rearrange to get
\[
    \sup_{\xi\in\mathbb{R}^n} g(x,\xi)-\sup_{\xi\in\mathbb{R}^n} g(y,\xi)\leq \sup_{\xi\in\mathbb{R}^n} |g(x,\xi)-g(y,\xi)|.
\]  
An identical argument interchanging \(y\) and \(x\) shows that 
\[
    \sup_{\xi\in\mathbb{R}^n} g(y,\xi)-\sup_{\xi\in\mathbb{R}^n} g(x,\xi)\leq \sup_{\xi\in\mathbb{R}^n} |g(x,\xi)-g(y,\xi)|
\]
The claimed estimate then follows from taking a maximum.
\end{proof}

Finally, we establish a variant of the triangle inequality for products of functions.

\begin{lem}\label{lem:proddiff}
Let \(f_1,\dots,f_k\) be bounded functions on \(\mathbb{R}^n\), none of which is identically zero. Then for any two points \(x,y\in\mathbb{R}^n\),
\[
    \bigg|\prod_{j=1}^kf_j(x)-\prod_{j=1}^kf_j(y)\bigg|\leq \bigg(\prod_{j=1}^k\sup_{\mathbb{R}^n}|f_j|\bigg)\sum_{j=1}^k\frac{|f_j(x)-f_j(y)|}{\sup_{\mathbb{R}^n}|f_j|}.
\]
\end{lem}

\begin{proof}
We argue by induction on \(k\). There is nothing to show when \(k=1\), and when \(k=2\) we use the triangle inequality to write
\[
    |f_1(x)f_2(x)-f_1(y)f_2(y)|\leq |f_1(x)||f_2(x)-f_2(y)|+|f_2(y)||f_1(x)-f_1(y)|.
\]
Upon taking a supremum of \(|f_1(x)|\) and \(|f_2(y)|\), we get the desired estimate when \(k=2\). Suppose next that the estimate holds for \(k\) functions. Using the triangle inequality, we have
\begin{align*}
    \bigg|\prod_{j=1}^{k+1}f_j(x)-&\prod_{j=1}^{k+1}f_j(y)\bigg|=\bigg|f_{k+1}(x)\prod_{j=1}^{k}f_j(x)-f_{k+1}(y)\prod_{j=1}^{k}f_j(y)\bigg|\\
    &=\bigg|f_{k+1}(x)\prod_{j=1}^{k}f_j(x)-f_{k+1}(y)\prod_{j=1}^{k}f_j(x)+f_{k+1}(y)\prod_{j=1}^{k}f_j(x)-f_{k+1}(y)\prod_{j=1}^{k}f_j(y)\bigg|\\
    &\leq \bigg(\prod_{j=1}^{k}|f_j(x)|\bigg)|f_{k+1}(x)-f_{k+1}(y)|+|f_{k+1}(y)|\bigg|\prod_{j=1}^{k}f_j(x)-\prod_{j=1}^{k}f_j(y)\bigg|.
\end{align*}
Applying the inductive hypothesis, we see now that
\[
    \bigg|\prod_{j=1}^{k+1}f_j(x)-\prod_{j=1}^{k+1}f_j(y)\bigg|\leq \bigg(\prod_{j=1}^{k}|f_j(x)|\bigg)|f_{k+1}(x)-f_{k+1}(y)|+\bigg(\prod_{j=1}^{k+1}\sup_{\mathbb{R}^n}|f_j|\bigg)\sum_{j=1}^k\frac{|f_j(x)-f_j(y)|}{\sup_{\mathbb{R}^n}|f_j|}.
\]
The required estimate then follows by taking a supremum in the product on the right, completing the inductive step and showing that the claimed estimate holds for all \(k\) as claimed. 
\end{proof}

It is easy to see that the preceding results still hold if \(\mathbb{R}^n\) is replaced with any domain \(\Omega\subseteq\mathbb{R}^n\).

\section{H\"older Continuous Functions}\label{sec:holdexamples}

Our main results for H\"older spaces would not be very useful if there were no interesting functions to which they could apply. In this section we provide some concrete (and not completely trivial) examples of \(\alpha\)-H\"older continuous functions for various values of \(\alpha\). In doing so, we are also able to illustrate various techniques useful in proving H\"older continuity.

We begin with a straightforward and well-known example of an \(\alpha\)-H\"older continuous function.

\begin{clm}
If \(\alpha\in (0,1]\) then \(f(x)=|x|^\alpha\) defined for \(x\in\mathbb{R}^n\) belongs to \(C^\alpha(\mathbb{R}^n)\), and \([f]_{\alpha,\mathbb{R}^n}=1\).
\end{clm}

\begin{proof}
To verify this, we assume without loss of generality that \(x,y\in\mathbb{R}^n\) and \(|x|>|y|\), so that 
\[
    0<1-\bigg(\frac{|y|}{|x|}\bigg)^\alpha\leq 1-\frac{|y|}{|x|}\leq \bigg(1-\frac{|y|}{|x|}\bigg)^\alpha
\] 
since \(\alpha\leq 1\) and \(\frac{|y|}{|x|}<1\) by assumption. From the estimates above it follows now that
\[
    \frac{|f(x)-f(y)|}{|x-y|^\alpha}\leq  \frac{||x|^\alpha-|y|^\alpha|}{||x|-|y||^\alpha}=\frac{1-(\frac{|y|}{|x|})^\alpha}{(1-\frac{|y|}{|x|})^\alpha}\leq \frac{1-\frac{|y|}{|x|}}{1-\frac{|y|}{|x|}}=1.
\]
The same estimate clearly holds if we interchange \(x\) and \(y\). Since \(x\) and \(y\) were arbitrary, it follows from the estimate above that \([f]_{\alpha,\mathbb{R}^n}\leq 1\). Moreover this holds with equality, for if we assume that \([f]_\alpha<1\), then we would have \(|x|^\alpha=|f(x)|\leq [f]_{\alpha,\mathbb{R}^n}|x|^\alpha<|x|^\alpha\), a contradiction for \(x\neq 0\). It follows that \([f]_{\alpha,\mathbb{R}^n}=1\), as claimed.
\end{proof}

Next we move on to a more interesting example: Cantor's `Middle Thirds' function, sometimes colloquially referred to as the `Devil's Staircase.' This is famously a continuous bijection on \([0,1]\) whose derivative is zero except on a set of measure zero; as such, it serves as a pathological counterexample to the intuitive but incorrect idea that a continuous increasing function should be increasing on some interval.

The Cantor function, which we denote by \(F\), is a fractal curve which turns out to be H\"older continuous. Additionally, it is proved in \cite[Chapter 9]{Teschl} that by taking \(f_0(x)=x\) on \([0,1]\) and defining \(f_{n+1}(x)=Tf_n(x)\), where \(T\) is the transformation
\[
    Tf(x)=\begin{cases}
    \hfil\frac{1}{2}f(3x) & 0\leq x\leq \frac{1}{3},\\
    \hfil\frac{1}{2} & \frac{1}{3}<x<\frac{2}{3},\\
    \frac{1}{2}(1+f(3x-2)) & \frac{2}{3}\leq x\leq 1,
    \end{cases}
\]
then we obtain a sequence of functions \(\{f_n\}\) which converge uniformly to \(F\) on \([0,1]\). By solving an exercise in \cite{Teschl}, we furnish a proof to the following well-known property of \(F\).

\begin{clm}
The Cantor function \(F\) belongs to \(C^{\log_32}([0,1])\) and \([F]_{\log_32,[0,1]}\leq1\).
\end{clm}

\begin{proof} Given a continuous bijection \(f\) of \([0,1]\) such that \(f(0)=0\), \(f(1)=1\) and \(f\in C^{\alpha}([0,1])\) for \(0<\alpha\leq 1\), we first examine the H\"older continuity of \(Tf\). In particular, we show that
\[
    [Tf]_\alpha\leq \frac{3^\alpha[f]_\alpha}{2}.
\]
This is done by considering six exhaustive cases which depend on the locations of \(x,y\in[0,1]\).

\noindent\underline{\textit{Case 1}}: If \(0\leq x\leq \frac{1}{3}\) and \(0\leq y\leq \frac{1}{3}\) then
\[
    |Tf(x)-Tf(y)|=\frac{1}{2}|f(3x)-f(3y)|\leq \frac{[f]_\alpha}{2}|3x-3y|^\alpha=\frac{3^\alpha[f]_\alpha}{2}|x-y|^\alpha.
\]
\noindent\underline{\textit{Case 2}}: If \(0\leq x\leq \frac{1}{3}\) and \(\frac{1}{3}< y<\frac{2}{3}\) then
\[
    |Tf(x)-Tf(y)|=\frac{1}{2}|f(3x)-1|=\frac{1}{2}|f(3x)-f(1)|\leq \frac{[f]_\alpha}{2}|3x-1|^\alpha\leq \frac{3^\alpha[f]_\alpha}{2}|x-y|^\alpha,
\]
where the last inequality holds since \(|3x-1|=3|x-\frac{1}{3}|\leq 3|x-y| \).

\noindent\underline{\textit{Case 3}}: If \(0\leq x\leq \frac{1}{3}\) and \(\frac{2}{3}\leq y\leq 1\) then 
\[
    |Tf(x)-Tf(y)|=\frac{1}{2}|f(3x)-f(1)+f(0)-f(3y-2)|\leq \frac{3^\alpha[f]_\alpha}{2}\bigg|x-\frac{1}{3}\bigg|^\alpha+\frac{3^\alpha[f]_\alpha}{2}\bigg|y-\frac{2}{3}\bigg|^\alpha.
\]
To get the required inequality, we must show that \(|x-\frac{1}{3}|^\alpha+|y-\frac{2}{3}|^\alpha\leq |x-y|^\alpha\) for \(\alpha\) sufficiently large. To this end we observe that for \(\alpha\leq 1\) the function \(Q:[0,\frac{1}{3}]\times[\frac{2}{3},1]\rightarrow\mathbb{R}\) defined by
\[
    Q(x,y)=\bigg|\frac{x-\frac{1}{3}}{x-y}\bigg|^\alpha+\bigg|\frac{y-\frac{2}{3}}{x-y}\bigg|^\alpha
\]
achieves a unique global maximum at \((0,1)\). It follows that on \([0,\frac{1}{3}]\times[\frac{2}{3},1]\) we have the bound \( Q(x,y)\leq Q(0,1)=\frac{2}{3^\alpha}\). If \(\alpha\geq\log_32\) then we see that \(Q(x,y)\leq \frac{2}{3^\alpha}\leq 1\) in the domain of interest, meaning that again,
\[
    |Tf(x)-Tf(y)|\leq \frac{3^\alpha[f]_\alpha}{2}|x-y|^\alpha.
\]
\underline{\textit{Case 4}}: If \(\frac{1}{3}<x<\frac{2}{3}\) and \(\frac{1}{3}<y<\frac{2}{3}\) then
\[
    |Tf(x)-Tf(y)|=\bigg|\frac{1}{2}-\frac{1}{2}\bigg|=0\leq \frac{3^\alpha[f]_\alpha}{2}|x-y|^\alpha.
\]
\noindent\underline{\textit{Case 5}}: If \(\frac{1}{3}<x<\frac{2}{3}\) and \(\frac{2}{3}\leq y\leq 1\) then
\[
    |Tf(x)-Tf(y)|=\frac{1}{2}|1-(1+f(3y-2))|=\frac{1}{2}|f(3y-2)-f(0)|\leq \frac{[f]_\alpha}{2}|3y-2|^\alpha\leq\frac{3^\alpha[f]_\alpha}{2}|x-y|^\alpha.
\]
\noindent\underline{\textit{Case 6}}: If \(\frac{2}{3}\leq y\leq 1\) and \(\frac{2}{3}\leq y\leq 1\) then
\[
    |Tf(x)-Tf(y)|=\frac{1}{2}|f(3x-2)-f(3y-2)|\leq \frac{[f]_\alpha}{2}|3x-3y|^\alpha=\frac{3^\alpha[f]_\alpha}{2}|x-y|^\alpha.
\]

Thus, if \(\log_32\leq\alpha\leq 1\) then for \(x,y\in[0,1]\) we have the following H\"older estimate on \(Tf\),
\[
    |Tf(x)-Tf(y)|\leq \frac{3^\alpha[f]_\alpha}{2}|x-y|^\alpha.
\]
Equipped with this property of \(T\), we now observe that if we fix \(\alpha=\log_32\) and if \([f]_\alpha\leq1\), then \(|Tf(x)-Tf(y)|\leq |x-y|^\alpha\) and \([Tf]_\alpha\leq 1\). The function \(f_0(x)=x\) satisfies \([f_0]_\alpha\leq 1\) since for \(x,y\in[0,1]\) we have \(|x-y|\leq 1\) and 
\[
    |f_0(x)-f_0(y)|=|x-y|=|x-y|^{1-\alpha}|x-y|^\alpha\leq |x-y|^\alpha.
\]
Consequently with \(\alpha=\log_32\) we have \([f_1]_\alpha=[Tf_0]_\alpha\leq 1\), and it follows by induction that \([f_n]_\alpha\leq 1\) for every \(n\). Using this estimate, we see that for every \(n\) the Cantor function \(F\) satisfies
\[
    |F(x)-F(y)|\leq |F(x)-f_n(x)|+|F(y)-f_n(y)|+|x-y|^\alpha.
\]
The left-hand side is independent of \(n\), and since the functions \(f_n\) converge to \(F\) pointwise we have that \(|F(x)-f_n(x)|\rightarrow0\) as \(n\rightarrow\infty\) for each \(x\in[0,1]\). It follows in the limit that
\[
    |F(x)-F(y)|\leq|x-y|^\alpha
\]
for every \(x,y\in[0,1]\), meaning that \(F\in C^{\log_32}([0,1])\) as we wished to show.
\end{proof}

This approach to computing H\"older continuity for fractal curves can be applied more generally. In particular, it suffices to find an appropriate map \(T\) corresponding to a given fractal curve, and to employ the technique used above to show that \([Tf]_\alpha\leq [f]_\alpha\) for some \(\alpha\). The map in question is most easily deduced from symmetries of the fractal curve in question. 

One such fractal curve is the Minkowski `Question Mark' function. This is often denoted `\(?\)' in the literature, however to avoid confusion (no pun intended) we denote this function by \(Q\). It has fixed points at \(0\) and \(1\) and it enjoys the following self-similarity relations for \(x\leq \frac{1}{2}\),
\begin{align*}
    &Q\bigg(\frac{x}{1+x}\bigg)=\frac{1}{2}Q(x),\\
    &Q(x)+Q(1-x)=1,
\end{align*}
see \cite{ALKAUSKAS_2009}. Salem shows in \cite{salem} that \(Q\in C^{\alpha}([0,1])\), with the H\"older exponent \(\alpha\) given by
\[
    \alpha=\frac{\log 2}{2\log(\frac{1+\sqrt{5}}{2})}.
\]
The proof of this result in \cite{salem} uses an altogether different technique from that employed above, but we note that owing to the symmetries listed above, \(Q\) is a fixed point of the map
\[
    Tf(x)=\begin{cases}
    \hfil\frac{1}{2}f(\frac{x}{1-x}) & x\leq\frac{1}{2}\\
    \frac{1}{2}+\frac{1}{2}f(\frac{2x-1}{x}) & x>\frac{1}{2}
    \end{cases}
\]
which preserves bijections on \([0,1]\). By taking \(f_0(x)=x\) and defining \(f_{n+1}=Tf_n\) as we did for the Cantor function, we obtain a sequence of functions which converge uniformly to \(Q\) since \(T\) is a contraction mapping -- we direct the interested reader to \cite{bezier} for a thorough investigation of this idea. Further, we close by noting that a similar approach can be used to study (and for our purposes, compute the H\"older continuity of) other fractals like the Blancmange curve.
\chapter{Conclusion and
Discussion }\label{chap:last}

\section{Summary of Findings}

This thesis studied non-negative H\"older continuous functions, to determine when they can be decomposed as sums of squares of functions that are `half-regular' in an appropriate sense. Building upon the work of Fefferman \& Phong in \cite{Fefferman-Phong}, Tataru in \cite{Tataru} and Sawyer \& Korobenko in \cite{SOS_I}, and taking inspiration from the work of many others, we have demonstrated that if \(f\) is a non-negative \(C^{k,\alpha}(\mathbb{R}^n)\) function for \(k\leq 3\) and \(0<\alpha\leq 1\), then there exists a decomposition of the form
\begin{equation}\label{eq:sosos}
    f=\sum_{n=1}^{m_n}g_j^2,
\end{equation}
where \(g_j\in C^\frac{k+\alpha}{2}(\mathbb{R}^n)\) for each \(j\), and the constant \(m_n\) depends only on \(n\). Formerly, the version of this result in \(C^{3,1}(\mathbb{R}^n)\) was well-known thanks to Fefferman \& Phong, and in \(C^\alpha(\mathbb{R}^n)\) for \(0<\alpha\leq 1\) it is a straightforward consequence of concavity; see \Cref{lem:smallpow}. The remaining cases do not seem to have been dealt with prior to our work; accordingly, their resolution in \Cref{thm:c3main} comprises the main contribution of this thesis.

In addition to this novel result, we also presented some progress toward an adaptation of our decomposition result to large values of \(k\) in \Cref{chap:higher}, building upon the work of Sawyer \& Korobenko in \cite{SOS_I}. While the characterization of decomposable \(C^{k,\alpha}(\mathbb{R}^n)\) functions remains an open problem for \(k\geq 4\) in every dimension \(n\geq 2\), we have shown in \Cref{thm:seconddmain} that \eqref{eq:sosos} can hold under some restrictive pointwise conditions on \(f\). The conditions we impose come at the expense of some regularity on H\"older scale, meaning that the decomposition we form is not truly `half-regular' in three or more dimensions.

While it is unclear whether the conditions of \Cref{thm:seconddmain} can be weakened, we showed in \Cref{chap:poly} that some additional restrictions on \(f\) are necessary if an analog to \Cref{thm:c3main} is to hold for large values of \(k\). This is because \eqref{eq:sosos} fails for some functions in \(C^{k,\alpha}(\mathbb{R}^n)\) when \(k\geq 4\) and \(n\geq 2\), and indeed our counterexamples in this setting are simply polynomials. By means of a careful construction motivated by the work of Reznik in \cite{reznik_extpsd}, we were able to provide a technique for finding explicit examples of such polynomials in all settings in which they exist, few of which could be found in the literature. It is hoped that our examples contribute to a greater understanding of sums of squares of polynomials and H\"older continuous functions.

Finally, throughout the thesis we presented several supplementary results which enjoy various degrees of novelty. \Cref{lem:genchain}, viewed as an expression of the structure of high order derivatives, is very well-known -- our contribution is to explicitly compute the constants involved. This might be useful in streamlining the computation of such derivatives. Likewise, the proof of \Cref{thm:specialodds} uses a well-known technique to derive the determinant of a Vandermonde-like matrix, but our application to extracting weighted summation identities seems to be new. It is hoped that such results, whether they are new or rediscovered, might be useful in other applications.

\section{Potential Applications}\label{sec:hypoelliptchap}

The starting point for this research was the work of Sawyer \& Korobenko in \cite{SOS_I}, in which they studied decompositions of \(C^{4,\alpha}(\mathbb{R}^n)\) functions and matrices. Their motivation in this, realized in the third paper of the series \cite{SOS_III}, was to study hypoellipticity of certain differential operators in divergence form. An operator \(L\) is said to hypoelliptic if \(Lu\in C^\infty(\mathbb{R}^n)\) for a distribution \(u\) implies that \(u\in C^\infty(\mathbb{R}^n)\). Put another way, \(L\) is hypoelliptic if a solution \(u\) to the differential equation \(Lu=f\) must be smooth when the data \(f\) is smooth. This is a significantly weakened form of the frequently exploited property of ellipticity, and it is poorly understood. In \cite{Christ_IDR} Christ suggests that a characterization of hypoelliptic operators is unlikely to exist, and warrants this suspicion with several counter-intuitive examples.

Nevertheless, for certain operators like those considered in \cite{SOS_III}, sufficient conditions for hypoellipticity have been obtained, and many of these rely in some way on a sum of squares decompositions or related concepts. H\"ormander's famous `bracket condition' for hypoellipticity deals with operators which can be decomposed as sums of squares of vector fields \(X_1,\dots,X_m\),
\[
    L=-\sum_{j=1}^m X_j^*X_j.   
\]
Here, \(X_j^*\) is the \(L^2(\mathbb{R}^n)\) adjoint of \(X_j\), and the right-hand side of the equation above becomes an algebraic sum of squares of dyads under a Fourier transform.

H\"ormander proves that \(L\) as above is hypoelliptic if the iterated commutators of the \(X_j\)'s span \(\mathbb{R}^n\) at every point in \cite{hormander}. Likewise, Christ proves some sufficient and sometimes necessary conditions for hypoellipticity of \(L\) in \cite{Christ_IDR}, even in situations where H\"ormander's result fails. Building upon the latter work, Sawyer \& Korobenko use a \(C^{4,\alpha}(\mathbb{R}^n)\) sum of squares decomposition to generalize Christ's result to situations in which the vector fields enjoy limited regularity. Even the work of Kohn in \cite{Kohn}, which does not explicitly involve sums of squares, nevertheless employs a variant of the Malgrange inequality \eqref{eq:malg}, which is central to showing that non-negative \(C^{1,\alpha}(\mathbb{R}^n)\) functions have half-regular roots; see \cite[Eq.(2.18)]{Kohn}.

It is not clear whether there is any fundamental connection between hypoellipticity and sums of squares, or if sums of squares are merely a convenient technique with which to study hypoelliptic operators. In any case, equipped with the expanded decomposition results of this thesis, it may be possible to revisit and build upon several of the results mentioned above, or at least to show that they apply to broader classes of differential operators. In particular, our higher-order decomposition theorems might be suitable for studying the hypoellipticity of equations of higher order, whereas many of the results above restrict to second-order operators.

In passing, we also note that in \cite{Guan}, Guan uses the \(C^{3,1}(\mathbb{R}^n)\) decomposition from Fefferman and Phong \cite{Fefferman-Phong} to study behaviour of the Monge-Amp\`ere equation, and similar decomposition result to those obtained in our work may be useful in expanding upon Guan's work. In particular, Guan shows that if the data function \(f\) of the Monge-Amp\`ere equation \(D^2u=f\) is a sum of squares of \(C^{1,1}(\Omega)\) functions on a suitable domain \(\Omega\), then the solution \(u\) is also in \(C^{1,1}(\Omega)\). In particular, it follows from \cite{Fefferman-Phong} that it is enough to have \(f\in C^{3,1}(\Omega)\). By generalizing to higher values of \(k\) and working with decompositions into other integer powers (e.g. cubes and fourth powers), our techniques may be suitable for adapting Guan's work to the study of higher-order regularity of solutions to the Monge-Amp\`ere equations.



\section{Some Open Questions}\label{sec:applics}

Briefly, we discuss some questions which are left open by this thesis, either due to constraints of time and length of this document, or due to difficulty of the problems themselves.

First, what is the smallest number \(m\) for which every sum of half-regular squares in \(C^{k,\alpha}(\mathbb{R}^n)\) is a sum of at most \(m\) squares? Borrowing (and perhaps adulterating) terminology from algebra, we might call this number the Pythagoras number of \(C^{k,\alpha}(\mathbb{R}^n)\), and we denote it here by \(p_{k,n}\). 

In \Cref{thm:malgrangeineq} we show that \(p_{1,n}=p_{2,n}=1\) for every \(n\), and Bony shows in \cite{Bony} that \(p_{k,1}=2\) for every \(k\geq 2\) (see \Cref{thm:onedim}). The work of \Cref{sec:sizebounds} also shows that
\[
    p_{2,n}\leq 2\cdot15^{n^2-n}+\frac{15^{n^2}-15^n}{15^n-1},
\]
but this bound is likely to be far from optimal. The same crude bound is obtained for \(p_{3,n}\), and it is reasonable to suspect from the similarity of the cases \(k=0,1\) and \(k=2,3\) that in general, \(p_{k,n}=p_{k+1,n}\) for even \(k\in\mathbb{N}\cup\{0\}\). Otherwise, we know little about \(p_{k,n}\) in general. 

In particular, no lower bound was found in the course of this research, though it is suspected that one might be found by using the Pythagoras number of polynomials owing to the counterexamples we obtained in \Cref{chap:poly}. These numbers have been studied in a more general setting by Choi, Lam \& Reznik in \cite{CLR}. It is unclear if these numbers are useful in any way, but it is apparent that finding them will require the development of some interesting mathematics.

A second open question is the following: does there exist a non-negative smooth function on \(\mathbb{R}^n\) which cannot be written as a sum of squares of smooth functions? As Sawyer \& Korobenko remark in \cite{SOS_I}, such an example is said to have been found by the late Paul Cohen, but a search of the literature afforded no such example and Pieroni affirms in \cite[Remark 5.1]{pieroni} that the problem remains open.

There are several directions which we can identify that may, with some work, yield an answer to this open question. First, by honing in on necessary conditions for \(C^{k,\alpha}(\mathbb{R}^n)\) functions to be decomposed into a sum of half-regular squares for arbitrarily large \(k\) (as we attempted in \Cref{thm:seconddmain}), it may be possible to construct a non-negative smooth function violating said conditions. Such a function would necessarily be non-decomposable as a sum of squares of smooth functions, affording the desired counterexample.

As a second approach, one could employ the non-decomposable polynomials found in \Cref{chap:poly} as candidates for such functions. The argument used in \Cref{sec:nd} to show that such polynomials are not sums of half-regular squares fails to show that these polynomials are not sums of squares of smooth functions, since there is no longer a highest-order term in the Taylor polynomial to consider. It may also be the case that a decomposition function is smooth and nowhere analytic, in which case the Taylor polynomial has little relation to the function. As such, an altogether different argument would be needed to approach the problem from this angle.

Of course, it may be the case that all non-negative smooth functions can in fact be decomposed as sums of squares of smooth functions. This seems feasible at least for functions of one dimension, thanks to the results of Bony in \cite{Bony}. However, owing to the counterexamples to our regularity-preserving decomposition results in higher dimensions, the behaviour of \(C^{\infty}(\mathbb{R}^n)\) functions for \(n\geq 2\) is less predictable.

\section{Closing Remark}

The study of regularity-preserving decompositions undertaken in this thesis was motivated by their potential for applications to the study of hypoelliptic operators -- in that sense, this work can be seen as a means to an end. However, in undertaking that study, it was quickly discovered that the problem harbours considerable underlying beauty, and the technical difficulties  the decompositions presented often pointed the way to rich mathematical structures. Working on this problem therefore quickly became an end unto itself.

Most grade school students encounter sums of squares for the first time upon learning of the Pythagorean theorem, and thereafter it is easy to think of them as fairly unremarkable, given their commonality in many areas of mathematics. Yet by placing sums of squares within new contexts, the simple `certificate of non-negativity' can take on a new life involving a great deal of novel properties and fresh insights. Such was the case with sums of squares in H\"older spaces.

The author is grateful to have been directed to this problem by his supervisor, Dr. Eric T. Sawyer. It was an approachable and straightforward problem to understand, yet sufficiently challenging that it induced the author to strike out and learn new mathematics with the aspiration of solving it. It is hoped that the contributions of this work on sums of squares and regularity preserving decompositions are found to be meaningful by some mathematicians, and that this work can be built upon to advance our understanding of the subjects it touches.
\newpage
\addcontentsline{toc}{chapter}{Bibliography}
\renewcommand\fbox{\fcolorbox{white}{white}}
\bibliographystyle{amsalpha}
\pagestyle{plain}
\bibliography{chap/refs}

\providecommand{\bysame}{\leavevmode\hbox to3em{\hrulefill}\thinspace}
\providecommand{\MR}{\relax\ifhmode\unskip\space\fi MR }
\providecommand{\MRhref}[2]{%
  \href{http://www.ams.org/mathscinet-getitem?mr=#1}{#2}
}
\providecommand{\href}[2]{#2}
\begin{thebibliography}{Pow17}

\bibitem[Alk09]{ALKAUSKAS_2009}
Giedrius Alkauskas, \emph{The moments of minkowski question mark function: The
  dyadic period function}, Glasgow Mathematical Journal \textbf{52} (2009),
  no.~1, 41--64. \MR{2587817}

\bibitem[Art27]{Artin}
Emil Artin, \emph{{\"U}ber die zerlegung definiter funktionen in quadrate},
  Abhandlungen aus dem Mathematischen Seminar der Universit{\"a}t Hamburg
  \textbf{5} (1927), 100--115. \MR{3069468}

\bibitem[BB06]{Bony}
Jean-Michel Bony{,} and
  Fabrizio\!\!\!\!\!\!\!\!\!\!\!\!\!\!\!\!\!\!\!\!\!\!\!\!\!\!\!\!\!\!\!\!\!\!\!\fbox{Fabrizio\;\;\;\;\;\;}
  \!\!\!\!\!\!\!\!\!\!\!{Broglia, Ferruccio Colombini, and Ludovico Pernazza},
  \emph{Nonnegative functions as squares or sums of squares}, Journal of
  Functional Analysis \textbf{232} (2006), 137--147.

\bibitem[Ble06]{Blekherman}
Grigoriy Blekherman, \emph{There are significantly more nonnegative polynomials
  than sums of squares}, Israel Journal of Mathematics \textbf{153} (2006),
  355--380. \MR{2254649}

\bibitem[Bon05]{Bony2}
Jean-Michel Bony, \emph{Sommes de carr\'{e}s de fonctions d\'{e}rivables},
  Bull. Soc. Math. France \textbf{133} (2005), no.~4, 619--639. \MR{2233698}

\bibitem[BR76]{behzad}
M.~Behzad and H.~Radjavi, \emph{Chromatic numbers of infinite graphs}, Journal
  of Combinatorial Theory, Series B \textbf{21} (1976), no.~3, 195--200.

\bibitem[Chr01]{Christ_IDR}
Michael Christ, \emph{Hypoellipticity in the infinitely degenerate regime},
  Complex analysis and geometry ({C}olumbus, {OH}, 1999), Ohio State Univ.
  Math. Res. Inst. Publ., vol.~9, de Gruyter, Berlin, 2001,
  pp.~\!\!\!\!\!\!\!\!\!\!\!\!\!\fbox{59--84.}\! \MR{1912731}

\bibitem[CL78]{CL2}
Man~Duen Choi and Tsit~Yuen Lam, \emph{Extremal positive semidefinite forms},
  Math. Ann. \textbf{231} (1978), no.~1, 1--18. \MR{498384}

\bibitem[CL94]{CLR}
Man-Duen Choi and Tsit {Lam, and Bruce Reznik}, \emph{Sums of squares of real
  polynomials}, Proceedings of Symposia in Pure Mathematics \textbf{58} (1994),
  103--126.

\bibitem[Dri03]{Driver}
Bruce~K. Driver, \emph{Analysis tools with applications}, Math 231 PDE Lecture
  Notes, University of California, San Diego, June 2003.

\bibitem[Fol01]{folland}
Gerald~B. Folland, \emph{Advanced calculus}, Featured Titles for Advanced
  Calculus Series, Pearson Education Taiwan, December 2001, 85-92.

\bibitem[FP78]{Fefferman-Phong}
C.~Fefferman and D.~H. Phong, \emph{On positivity of pseudo-differential
  operators}, Proc. Nat. Acad. Sci. U.S.A. \textbf{75} (1978), no.~10,
  4673--4674. \MR{507931}

\bibitem[GG13]{Gobbino}
Marina Ghisi and Massimo Gobbino, \emph{Higher order {Glaeser} inequalities and
  optimal regularity of roots of real functions}, Annali della Scuola Normale
  Superiore di Pisa - Classe di Scienze \textbf{Ser. 5, 12} (2013), no.~4,
  1001--1021. \MR{3184577}

\bibitem[Gla63]{Glaeser}
Georges Glaeser, \emph{Racine carr\'{e}e d'une fonction diff\'{e}rentiable},
  Ann. Inst. Fourier Grenoble \textbf{13} (1963), no.~2, 203--210. \MR{163995}

\bibitem[GT01]{GilbargTrudinger}
David Gilbarg and Neil~S. Trudinger, \emph{Elliptic partial differential
  equations of second order}, Classics in Mathematics, Springer-Verlag, Berlin,
  2001, Reprint of the 1998 edition. \MR{1814364}

\bibitem[Gua97]{Guan}
Pengfei Guan, \emph{{$C^2$} a priori estimates for degenerate
  {M}onge-{A}mp\`ere equations}, Duke Math. J. \textbf{86} (1997), no.~2,
  323--346. \MR{1430436}

\bibitem[Guz70]{guzman}
{Miguel de} Guzman, \emph{A covering lemma with applications to
  differentiability of measures and singular integral operators}, Studia
  Mathematica \textbf{34} (1970), 299--317.

\bibitem[Har06]{hardy}
Michael Hardy, \emph{Combinatorics of partial derivatives}, Electron. J.
  Combin. \textbf{13} (2006), no.~1, Research Paper 1, 13. \MR{2200529}

\bibitem[Hil88]{Hilbert}
David Hilbert, \emph{Ueber die {D}arstellung definiter {F}ormen als {S}umme von
  {F}ormenquadraten}, Math. Ann. \textbf{32} (1888), no.~3, 342--350.
  \MR{1510517}

\bibitem[Hil93]{Hilbert2}
{D}avid Hilbert, \emph{\"{U}ber tern\"{a}re definite {F}ormen}, Acta Math.
  \textbf{17} (1893), no.~1, 169--197. \MR{1554835}

\bibitem[H{\"o}r67]{hormander}
Lars H{\"o}rmander, \emph{Hypoelliptic second order differential equations},
  Acta Math. \textbf{119} (1967), 147--171. \MR{222474}

\bibitem[Koh98]{Kohn}
Joseph~J. Kohn, \emph{Hypoellipticity of some degenerate subelliptic
  operators}, J. Funct. Anal. \textbf{159} (1998), no.~1, 203--216.
  \MR{1654190}

\bibitem[KP13]{IFTref}
Steven~G. Krantz and Harold~R. Parks, \emph{The implicit function theorem},
  Modern Birkh\"{a}user Classics, Springer, New York, November 2013.
  \MR{2977424}

\bibitem[KS21]{SOS_III}
Lyudmila Korobenko and Eric~T. Sawyer, \emph{Sums of squares {III}:
  hypoellipticity in the infinitely degenerate regime}, 2021.

\bibitem[KS23]{SOS_I}
Lyudmila Korobenko and Eric Sawyer, \emph{Sums of squares {I}: {S}calar
  functions}, J. Funct. Anal. \textbf{284} (2023), no.~6, Paper No. 109827.
  \MR{4530899}

\bibitem[Las07]{lasserre}
Jean~B. Lasserre, \emph{A sum of squares approximation of nonnegative
  polynomials}, SIAM Review \textbf{49} (2007), no.~4, 651--669. \MR{2375528}

\bibitem[Mot67]{Motzkin}
Theodore~S. Motzkin, \emph{The arithmetic-geometric inequality}, Inequalities
  ({P}roc. {S}ympos. {W}right-{P}atterson {A}ir {F}orce {B}ase, {O}hio, 1965),
  Academic Press, New York, 1967,
  pp.~\!\!\!\!\!\!\!\!\!\!\!\!\!\fbox{205--224.}\! \MR{0223521}

\bibitem[MS15]{masur}
Mar{\'{\i}}a~J. Mart{\'{\i}}n{,} and Eric
  T.\!\!\!\!\!\!\!\!\!\!\!\!\!\!\!\!\!\!\!\!\!\!\!\!\!\!\!\!\!\!\!\fbox{Eric
  T.\;\;\;\;\;\;\;}~\!\!\!\!\!\!\!\!\!\!\!\!\!\!\! {Sawyer, Ignacio
  Uriarte{-}Tuero and Dragan Vukoti{\'{c}}}, \emph{The {K}rzy\.{z} conjecture
  revisited}, Adv. Math. \textbf{273} (2015), 716--745. \MR{3311774}

\bibitem[Pie07]{pieroni}
Federica Pieroni, \emph{On the real algebra of {D}enjoy-{C}arleman classes},
  Selecta Math. (N.S.) \textbf{13} (2007), no.~2, 321--351. \MR{2361097}

\bibitem[Pow17]{Powers}
Victoria Powers, \emph{Positive polynomials and sums of squares: theory and
  practice}, Real algebraic geometry, Panor. Synth\`eses, vol.~51, Soc. Math.
  France, Paris, 2017, pp.~\!\!\!\!\!\!\!\!\!\!\!\!\!\fbox{155--180.}\!
  \MR{3701213}

\bibitem[Rez78]{reznik_extpsd}
Bruce Reznick, \emph{Extremal {PSD} forms with few terms}, Duke Math. J.
  \textbf{45} (1978), no.~2, 363--374. \MR{480338}

\bibitem[Rob73]{robinson}
Raphael~M. Robinson, \emph{Some definite polynomials which are not sums of
  squares of real polynomials}, Selected questions of algebra and logic, Izdat.
  ``Nauka'' Sibirsk. Otdel., Novosibirsk, 1973, pp.~264--282. \MR{0337878}

\bibitem[Sal43]{salem}
R.~Salem, \emph{On some singular monotonic functions which are strictly
  increasing}, Trans. Amer. Math. Soc. \textbf{53} (1943), 427--439. \MR{7929}

\bibitem[Sim83]{simon}
Leon Simon, \emph{Lectures on geometric measure theory}, Proceedings of the
  Centre for Mathematical Analysis, Australian National University, vol.~3,
  Australian National University, Centre for Mathematical Analysis, Canberra,
  1983. \MR{756417}

\bibitem[Tat02]{Tataru}
Daniel Tataru, \emph{On the {F}efferman-{P}hong inequality and related
  problems}, Comm. Partial Differential Equations \textbf{27} (2002),
  no.~11-12, 2101--2138. \MR{1944027}

\bibitem[Tes16]{Teschl}
Gerald Teschl, \emph{Topics in real and functional analysis}, University of
  Vienna, Graduate Studies in Mathematics, 2016.

\bibitem[TG11]{bezier}
Konstantinos~I. Tsianos and Ron Goldman, \emph{Bezier and {B}-spline curves
  with knots in the complex plane}, Fractals \textbf{19} (2011), no.~1, 67--86.
  \MR{2776741}

\end{thebibliography}
\appendix
\chapter{Standard Results}\label{chap:appendix}
\pagestyle{thesis}

For convenience of the reader, we list some standard and well-known results in this section which are occasionally called upon in the main body of the thesis. They are drawn from various references, which we include with each result.

\begin{thm}[Rademacher]\emph{(\cite[Theorem 5.2]{simon})}\label{thm:rad}
If \(f\) is Lipschitz on \(\mathbb{R}^n\) then it is differentiable almost everywhere.
\end{thm}

\begin{defn}[Convexity \& Concavity]\label{lem:convexity}
A function \(f:\mathbb{R}^n\rightarrow\mathbb{R}\) is said to be a convex on a set \(\Omega\) if for each \(x,y\in\Omega\), and \(0\leq \lambda\leq1\),
\[
    f(\lambda x+(1-\lambda)y)\leq\lambda f(x)+(1-\lambda)f(y).
\]
On the other hand, a function \(f\) is said to be concave on \(\Omega\) if for each \(x,y\in\Omega\) and \(0\leq \lambda\leq1\),
\[
    \lambda f(x)+(1-\lambda)f(y)\leq f(\lambda x+(1-\lambda)y).
\]
\end{defn}

\begin{thm}[Weighted AMGM Inequality]\emph{(\cite{Motzkin})}\label{thm:amgm}
Let \(c_1,\dots,c_n\) be non-negative real numbers, and assume that \(\lambda_1,\dots,\lambda_n\) are non-negative and \(\lambda_1+\cdots+\lambda_n=1\). Then
\[
    \prod_{j=1}^nc_j^{\lambda_j}\leq \sum_{j=1}^n\lambda_j c_j.
\]
\end{thm}

\begin{cor}[Young's Inequality]\label{cor:young}
Let \(a\) and \(b\) be non-negative real numbers and let \(b,q\geq 1\) be such that \(\frac{1}{p}+\frac{1}{q}=1\). Then
\[
    ab\leq \frac{a^p}{p}+\frac{b^q}{q}.
\]
\end{cor}

\begin{proof}
Take \(c_1=a^p\), \(c_2=b^q\), \(\lambda_1=\frac{1}{p}\) and \(\lambda_2=\frac{1}{q}\) in Theorem \ref{thm:amgm}.
\end{proof}

\begin{thm}[Taylor's Theorem]\emph{(\cite[Theorem 2.68]{folland})}\label{thm:taylorthm}
Suppose \(f:\mathbb{R}^n\rightarrow\mathbb{R}\) is of class \(C^{k+1}(\Omega)\) on an open convex set \(\Omega\). If \(x\in\Omega\) and \(x+h\in\Omega\) then
\[
    f(x+h)=\sum_{|\beta\leq k|}\frac{\partial^\beta f(x)}{\beta!}h^\beta+(k+1)\sum_{|\beta|=k+1}\frac{h^\beta}{\beta!}\int_0^1(1-t)^k\partial^\beta f(x+th)dt.
\]
\end{thm}

\end{document}